\documentclass[10pt]{amsart}
\usepackage{times}
\usepackage{amssymb,amscd,amsthm, verbatim,amsmath,color,fancyhdr, mathrsfs}
\usepackage{tikz}
\usepackage{graphicx}
\usepackage{turnstile}
\usepackage{mathtools}
\usepackage{bm}
\usepackage{inconsolata}  

\usepackage[pagewise, mathlines]{lineno}
\usetikzlibrary{arrows.meta}
\usetikzlibrary{calc}

\usepackage[a4paper, 
left=2.75cm, right=2.75cm,
dvips]{geometry}
\usepackage[colorlinks=true,citecolor=blue]{hyperref}

\usepackage{pgfplots}
\pgfplotsset{compat=newest}
\usetikzlibrary{calc, intersections}

\usepackage{dhucs}
\usepackage[all]{xy}
\usepackage{float}
\usepackage{bbold}
\usepackage{enumerate}
\usepackage{listings}

\newcommand{\PP}{\mathcal{P}}
\newcommand{\MM}{\mathcal{M}}
\newcommand{\DD}{\mathcal{D}}

\newcommand{\lk}{\mathsf{lk}}
\newcommand{\st}{\mathsf{st}}

\newcommand{\NM}{\mathsf{NM}}
\newcommand{\PM}{\mathsf{PM}}
\newcommand{\FC}{\mathsf{FC}}
\newcommand{\BFC}{\mathsf{BFC}}

\newtheorem{thm}{Theorem}[section]

\newtheorem{lemma}[thm]{Lemma}
\newtheorem{claim}[thm]{Claim}
\newtheorem{prop}[thm]{Proposition}
\newtheorem{cor}[thm]{Corollary}

\theoremstyle{definition}

\newtheorem{remark}[thm]{Remark}

\address{Andreas F. Holmsen, \hfill \hfill \linebreak 
	Department of Mathematical Sciences,  \hfill \hfill \linebreak
	KAIST, 
	Daejeon, South Korea.  \hfill \hfill }
\email{andreash@kaist.edu}

\address{Seunghun Lee, \hfill \hfill \linebreak 
	Department of Mathematical Sciences,\hfill \hfill \linebreak
	Binghamton University,
	Binghamton, NY, USA.  \hfill \hfill }
\email{shlee@binghamton.edu}

\thanks{Both authors were partially supported by the National Research Foundation of Korea (NRF) grants funded by the Ministry of Education (NRF-2016R1D1A1B03930998) and the Ministry of Science and ICT (No. 2020R1F1A1A0104849011).}

\title{Leray numbers of complexes of graphs with bounded matching number}

\author{Andreas F. Holmsen and Seunghun Lee}

\begin{document}
	\maketitle
	
	\begin{abstract}
		Given a graph $G$ on the vertex set $V$, the {\em non-matching complex} of $G$, denoted by $\NM_k(G)$, is the family of subgraphs $G' \subset G$ whose matching number $\nu(G')$ is strictly less than $k$. As an attempt to extend the result by Linusson, Shareshian and Welker on the homotopy types of $\NM_k(K_n)$ and $\NM_k(K_{r,s})$ to arbitrary graphs $G$, we show that (i) $\NM_k(G)$ is $(3k-3)$-Leray, and (ii) if $G$ is bipartite, then $\NM_k(G)$ is $(2k-2)$-Leray. This result is obtained by analyzing the homology of the links of non-empty faces of the complex $\NM_k(G)$, which vanishes in all dimensions $d\geq 3k-4$, and all dimensions $d \geq 2k-3$ when $G$ is bipartite. As a corollary, we have the following
		rainbow matching theorem which generalizes a result by Aharoni, Berger, Chudnovsky, Howard and Seymour: Let $E_1, \dots, E_{3k-2}$ be non-empty edge subsets of a graph and suppose that $\nu(E_i\cup E_j)\geq k$ for every $i\ne j$. Then $E=\bigcup E_i$ has a rainbow matching of size $k$. Furthermore, the number of edge sets $E_i$ can be reduced to $2k-1$ when $E$ is the edge set of a bipartite graph.
	\end{abstract}
	
	\section{Introduction}
	\subsection{Background}
	A \textit{simplicial complex} $\mathsf{K}$ on the ground set $E$ is a family of subsets of $E$, which satisfies the hereditary property: if $\sigma \subseteq \tau$ and $\tau \in \mathsf{K}$, then $\sigma \in \mathsf{K}$. In the particular case when $\mathsf{K}$ is a simplical complex which consists of graphs on a fixed vertex set, we call $\mathsf{K}$ a \textit{graph complex}. In the case of graph complexes, we consider a fixed vertex set, and we identify a graph $G$ in the graph complex $\mathsf{K}$ with its edge set $E(G) \subseteq \binom{V(G)}{2}$. All graphs considered in this paper are finite, simple, and undirected. It is also assumed that the empty graph $\emptyset$, that is, the graph with no edges, belongs to the graph complex.
	
	There are many graph complexes, whose topological properties -- homology, homotopy types, connectedness, Cohen-Macaulayness, and Euler characteristic -- have been extensively studied. Such examples include the complex of matchings, forests, bipartite graphs, non-Hamiltonian graphs, not $k$-connected graphs, and $t$-colorable graphs. Interested readers may find a detailed survey on the topic in the monograph by Jonsson \cite{jonsson_complex_graph} (in particular, Chapter 7).  
	
	\vspace{1ex}
	
	In this paper we focus on the complex of graphs which do not have matchings of size $k$. Here is a precise definition. Let $G$ be a graph. The {\em matching number} $\nu(G)$ is the size of a maximum matching in $G$, that is, the maximum number of pairwise disjoint edges in $G$. Given a graph $G$ on the vertex set $V$ we define the {\em non-matching complex} of $G$,  $\NM_k(G)$, as the family of subgraphs $G'$ of $G$ whose matching number $\nu(G')$ is strictly less than $k$. That is,
	\[\NM_k(G) = \{ G' \subseteq G: \nu(G')< k  \}.\]
	
	When $G$ is a complete graph or a complete bipartite graph, the exact homotopy type of the non-matching complex is known.
	Linusson, Shareshian and Welker \cite{linusson_bounded_matching_complex} 
	showed that $\NM_k(K_n)$ and $\NM_k(K_{r,s})$ are homotopy equivalent to wedges of spheres of dimension $3k-4$ and $2k-3$, respectively, giving exact formulas for the number of spheres in the wedges. Here, $K_n$ is the complete graph on $n$ vertices, and $K_{r,s}$ is the complete bipartite graph with bipartition $V_1\cup V_2$ where $|V_1|=r$ and $|V_2|=s$. Note that it is assumed that $n\geq 2k$ and $r,s\geq k$, or else both complexes are just a simplex. One of our goals here is to extend their results to arbitrary graphs.
	
	\subsection{Main results} \label{subs:mains}
	One of the consequences of the results of Linusson et al. is that for $G=K_n$ or $G=K_{r,s}$, the non-vanishing reduced homology of $\NM_k(G)$ is concentrated in a single dimension. This is not the case in general though. For example, the non-matching complex $\mathsf{NM}_3(G)$ of the graph depicted in Figure \ref{fig:non-vanish} has non-vanishing homology in dimensions four and five. (We invite the reader to come up with their own proof of this fact.)
	
	Our first result shows that for any graph $G$, the dimension in which $\mathsf{NM}_k(G)$ has non-trivial homology is never greater than that of $\mathsf{NM}_k(K_n)$.
	\begin{figure}
		\centering
		\begin{tikzpicture}
		\begin{scope}
		\coordinate (v1) at (0:1.25cm);
		\coordinate (v2) at (60:1.25cm);
		\coordinate (v3) at (120:1.25cm);
		\coordinate (v4) at (180:1.25cm);
		\coordinate (v5) at (240:1.25cm);
		\coordinate (v6) at (300:1.25cm);
		\path (v1)--(v2) coordinate[midway] (v7);
		\draw   (v1) -- (v2)
		(v1) -- (v3)
		(v1) -- (v4)
		(v1) -- (v5)
		(v1) -- (v6);
		\draw   (v2) -- (v3)
		(v2) -- (v4)
		(v2) -- (v5)
		(v2) -- (v6);
		\draw   (v3) -- (v4)
		(v3) -- (v5)
		(v3) -- (v6);
		\draw   (v4) -- (v5)
		(v4) -- (v6);
		\draw   (v5) -- (v6);
		\fill[black] (v1) circle (.33ex);
		\fill[black] (v2) circle (.33ex);
		\fill[black] (v3) circle (.33ex);
		\fill[black] (v4) circle (.33ex);
		\fill[black] (v5) circle (.33ex);
		\fill[black] (v6) circle (.33ex);
		\fill[black] (v7) circle (.33ex);
		\end{scope}
		\end{tikzpicture}
		\caption{\small The complete graph on six vertices with one edge subdivided. The graph complex $\mathsf{NM}_3(G)$ has non-vanishing homology in dimensions four and five.}
		\label{fig:non-vanish}
	\end{figure}
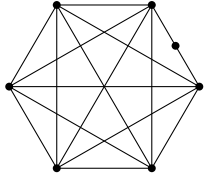

	\begin{thm}\label{leray-result}
		Let $k\geq 2$ be an integer and $G$ a graph. The complex $\NM_k(G)$ has vanishing homology in all dimensions $d\geq 3k-3$. Moreover, if $G$ is bipartite, then $\NM_k(G)$ has vanishing homology in all dimensions $d\geq 2k-2$.
	\end{thm}
	For a simplicial complex $\mathsf{K}$ 	let $\tilde{H}_i(\mathsf{K})$ denote the reduced homology of $\mathsf{K}$ with coefficients in a fixed field $\mathbb{F}$. The complex $\mathsf{K}$ is  {\em $d$-Leray} (over $\mathbb{F}$) if $\tilde{H}_i(\mathsf{L}) = 0 $ for all $i\geq d$ and for every induced subcomplex $\mathsf{L}\subseteq \mathsf{K}$. There is significant interest in the
	combinatorial properties of Leray complexes, especially in connection with Helly-type theorems \cite{kalai1, kalai2, kalai-meshulam3, kalai-meshulam1, colin}. The Leray property also comes up in commutative algebra where it corresponds to the Castelnuovo--Mumford regularity of a square-free monomial ideal \cite{kalai-meshulam2}.
	
	By observing that 
	the induced subcomplexes of $\mathsf{NM}_k(K_n)$ are precisely the complexes $\mathsf{NM}_k(G)$ where $G\subseteq K_n$,  Theorem \ref{leray-result} can be restated as: $\mathsf{NM}_k(K_n)$ is $(3k-3)$-Leray.
	
	The {\em link} of $\sigma \in \mathsf{K}$ is the complex $\lk_\mathsf{K}(\sigma) = \{\tau\subseteq E : \tau\cap \sigma =\emptyset, \tau\cup \sigma \in \mathsf{K}\}$. A well-known equivalence states that $\mathsf{K}$ is $d$-Leray if and only if $\tilde{H}_i(\lk_\mathsf{K}(\sigma)) = 0$ for every $i\geq d$ and $\sigma\in \mathsf{K}$ \cite[Proposition 3.1]{kalai-meshulam2}. (Note that $\mathsf{K} = \lk_\mathsf{K}(\emptyset)$.) Our second results shows that 
	%when the empty face is disregarded, 
	the bound in Theorem \ref{leray-result} can be slightly reduced when the empty face is excluded.
	
	\begin{thm} \label{link-result}
		Let $k\geq 2$ be an integer and $G$ a graph.
		The link of any non-empty face of the complex $\NM_k(G)$ has vanishing homology in all dimensions $d\geq 3k-4$. Moreover, if $G$ is bipartite, then the link of any non-empty face of the complex $\NM_k(G)$ has vanishing homology in all dimensions $d\geq 2k-3$.
	\end{thm}
	
	In fact, Theorem \ref{leray-result} can be deduced from Theorem \ref{link-result} by a simple application of the Mayer--Vietoris sequence. This reduction is independent of graph complexes and is given in Section \ref{homology}. Therefore the majority of this paper is devoted to the proof of Theorem \ref{link-result}. Our proof modifies and extends the methods by Linusson et al. \cite{linusson_bounded_matching_complex} which are based on discrete Morse theory and the Gallai--Edmonds decomposition theorem.
	
	\subsection{Applications} Although topological results on graph complexes are of significant interest in their own right, and sometimes
	require nontrivial graph-theoretical results, it is natural to wonder about the reverse direction. As Jonsson points out:
	%In Jonsson's monograph, it is written that
	\begin{quote}
		\textit{``Alas, we know very little about the existence of results in the other direction, i.e., proofs of nontrivial graph-theoretical theorems based on topological properties of certain graph complexes." \cite[page 13]{jonsson_complex_graph}}
	\end{quote}
	Indeed, Theorem \ref{link-result} was motivated by such an application. 
	
	Given a collection of edge sets $E_1, \dots, E_m$ of some underlying graph, a {\em rainbow matching} with respect to the collection is a matching in $E=\bigcup E_i$, where each edge of the matching is chosen from distinct $E_i$. (Note that the $E_i$ do not need to be disjoint.) Assuming $\nu(E_i) \geq k$ for all $i$, one may ask: {\em How many edge sets are needed to guarantee the existence of a rainbow matching of size $k$?} 
	
	A theorem by Drisko \cite{drisko} states that if the edge sets $E_i$ are chosen from $K_{k,n}$ with $k\leq n$, then $2k-1$ edge sets suffice. This result was generalized in  \cite{rainbow_matching_bipartite}, where it was shown that the same conclusion holds when the edge sets $E_i$ are chosen from $K_{n,n}$. Moreover, simple examples show that the number $2k-1$ is tight. The result was 
	further generalized to the setting of fractional matchings on $r$-uniform hypergraphs in \cite{aharoni_holzman_jiang}. 
	Our first application of Theorem \ref{link-result} is the following generalization of Drisko's theorem.
	
	\begin{thm}\label{drisko-general}
		Let $E_1, \dots, E_{2k-1}$ be non-empty edge subsets of a bipartite graph and suppose $\nu(E_i\cup E_j)\geq k$ for every $i\neq j$. Then $E = \bigcup E_i$ has a rainbow matching of size $k$.
	\end{thm}
	
	When the edge sets are not confined to a bipartite graph, Aharoni, Berger, Chudnovsky, Howard and Seymour \cite{rainbow_matching_general} showed that  $3k-2$ edges sets suffice.  
	Our second application of Theorem \ref{link-result} is the following generalization of the result from \cite{rainbow_matching_general}.
	\begin{thm}\label{rainbow-matching-result}
		Let $E_1, \dots, E_{3k-2}$ be non-empty edge subsets of a graph and suppose $\nu(E_i\cup E_j)\geq k$ for every $i\ne j$. Then $E = \bigcup E_i$ has a rainbow matching of size $k$.
	\end{thm}
	
	\begin{remark}
		While this manuscript was in preparation we learned that Theorems \ref{drisko-general} and \ref{rainbow-matching-result} have also been obtained in \cite{rainbow_matching_general_improvement,collaborative_bipartite}. However, their proof methods are combinatorial and differ from ours. 
	\end{remark}

	\subsection{Outline of paper}
	In section \ref{tools} we review several tools needed for the proof of Theorem \ref{link-result}. This involves discrete Morse theory and the Gallai--Edmonds decomposition theorem. 
	
	In section \ref{three special} we define three special families of graphs and state key results concerning acyclic matchings on these families with bounds on the sizes of the critical sets. The proofs of these results are given in sections \ref{section_size perfect m}, \ref{section_size-d0} and \ref{section_size-yz}.
	
	The proof of Theorem \ref{link-result} and the deduction of Theorem \ref{leray-result} is given in section \ref{proof of main}.
	The proofs of Theorems \ref{drisko-general} and \ref{rainbow-matching-result} will be given in section \ref{rainbow applications}, and we conclude with some remarks in section \ref{finals}.

	\subsection{Notation}
	Let $V$ and $W$ be disjoint sets of vertices. We denote the complete graph on $V$ by $K_V$, and the complete bipartite graph with vertex classes $V$ and $W$ by $K_{V,W}$. For a given graph $G$ on a vertex set containing $V\cup W$, let $G[V]$ be the induced subgraph of $G$ on $V$, and let $G[V,W]$ be the induced bipartite subgraph on vertex classes $V$ and $W$, that is,
	$$G[V,W]=\{e\in G: e\in K_{V,W} \}. $$
	When $V$ is empty, we set $K_V$ and $G[V]$ to be the empty graph $\emptyset$. Also, when $V$ or $W$ is empty, we set $K_{V,W}$ and $G[V,W]$ to be $\emptyset$.
	
	For a vertex $v$ of $G$, we use the standard notation $\deg_G(v)$ and $N_G(v)$
	to denote the degree of $v$ in $G$ and the neighborhood of $v$ in $G$, respectively. If $V$ is a subset of the vertex set of $G$, we let $N_G(V)$ denote the set of vertices not in $V$ which have at least one neighbor in $V$.
	For an edge $e$,  
	$G+e$ and $G-e$ denote the graph obtained by adding or deleting $e$ from $G$, respectively. Note that if $e \in G$, then $G+e = G$, and similarly, if $e\notin G$, then $G-e = G$. If $V_0$ is the vertex set of $G$ and $W\subset V_0$, then $G-W$ denotes the induced subgraph $G[V_0\setminus W]$.
	
	\section{Preliminaries} \label{tools}
	
	Here we give a brief outline of the main tools needed throughout the paper. We mainly follow the exposition and terminology from Jonsson's book \cite{jonsson_complex_graph}.
	
	\subsection{Discrete Morse theory} 
	Let $\mathsf{F}$ be a family of subsets of a finite ground set $E$. An {\em element matching} on $\mathsf{F}$ is a family $\mathcal{M}$ of ordered pairs $(\sigma, \tau)$ with $\sigma, \tau\in \mathsf{F}$ such that $\sigma\subsetneq \tau$, $|\tau\setminus\sigma|=1$, and any member of $\mathsf{F}$ is contained in at most one pair of $\mathcal{M}$. The sets in $\mathsf{F}$ that do not appear in any member of $\mathcal{M}$ are called {\em critical sets} (with respect to $\mathcal{M}$). If there are no critical sets, then $\mathcal{M}$ is called a {\em complete matching}. Whenever we speak of a matching on a family $\mathsf{F}$ we will always mean an element matching. (This should not be confused with a matching in a graph $G$ which means a set of pairwise disjoint edges.)
	
	Given an element matching $\mathcal{M}$ on $\mathsf{F}$, let $\mathcal{D} = \mathcal{D}(\mathsf{F}, \mathcal{M})$ denote the directed graph with vertex set $\mathsf{F}$ and directed edge from $\sigma$ to $\tau$ if and only if one of the following is satisfied:
	\begin{enumerate}
		\item $(\sigma, \tau)\in \mathcal{M}$
		\item $\tau\subsetneq \sigma$,  $|\sigma\setminus \tau|=1$, and $(\tau,\sigma)\notin \mathcal{M}$.
	\end{enumerate}
	
	In other words, the edges of $\mathcal{D}$ go between pairs of sets in $\mathsf{F}$ that differ by a single element of the ground set. Pairs that appear in $\mathcal{M}$ are directed from smaller to larger, while pairs that do not appear in $\mathcal{M}$ are directed from larger to smaller.  
	An element matching $\mathcal{M}$ is an {\em acyclic matching} if the directed graph $\mathcal{D}$ %(\mathsf{F}, \mathcal{M})$  
	is acyclic. Obviously, the empty matching is an acyclic matching.

	%\begin{eg} We give some basic examples of acyclic matchings. As a first example, let $E$ be a non-empty finite set and let $\mathsf{F}=2^E$. Fix an element $e\in E$ and define $$\MM = \{(\sigma, \sigma \cup \{e\}) : \sigma \subseteq E \setminus \{e\}\}.$$ It is obvious that $\MM$ is complete. To see that $\mathcal{M}$ is acyclic, note that if there is a directed cycle in $\mathcal{D}(\mathsf{F},\mathcal{M})$, then this cycle should use some pair $(\sigma, \sigma\cup \{e\})\in \MM$. As we follow the cycle, after we add $e$ to the set $\sigma$, the edge $e$ will be never removed, which makes it impossible to come come back to the set $\sigma$.
	
	%For the second example let $F=2^E \setminus \{\emptyset\}$. In this case, the same matching works, except that the subset $\{e\}$ does not appear in any pair, and so this is an acyclic matching with the single critical set $\{e\}$.
	%\end{eg}
	
	\medskip
	
	The relevant result for us is the `weak Morse inequalites' in the context of discrete Morse theory developed by Forman \cite{forman_discrete_morse_origin}. The following statement is taken from \cite{forman_lecture} (see Theorem 13 there), where it is stated in terms of discrete gradient vector fields which is a geometric name for acyclic matchings.
	
	\begin{thm} \label{morse_fundamental}
		Let $E$ be a finite set and $\mathsf{K}\subseteq 2^E$ be a simplicial complex. And let $H_i(\mathsf{K})$ be the homology of $\mathsf{K}$ with coefficients in a fixed field $\mathbb{F}$. Suppose that there is an acyclic matching $\MM$ on $\mathsf{K} \setminus \{\emptyset \}$. Then for every $i\geq 0$,
		$\dim H_i(\mathsf{K})$ is at most the number of critical sets with respect to $\MM$ of dimension $i$.
	\end{thm}
	
	Suppose there is an acyclic matching $\MM$ on a simplicial complex $\mathsf{K}$, and let $\MM'$ be the induced element matching on $\mathsf{K}\setminus \{\emptyset \}$. Clearly, $\MM'$ is also acyclic since $\DD(\mathsf{K}\setminus \{\emptyset \},\MM')$ is a directed subgraph of $\DD(\mathsf{K}, \MM)$. Therefore, we can apply Theorem \ref{morse_fundamental}. Especially we will be interested in the case when $i\geq 1$, where $\tilde{H}_i(\mathsf{K})=H_i(\mathsf{K})$, and the number of critical sets with respect to $\MM$ of dimension $i$ is same as the number of critical sets with respect to $\MM'$ of dimension $i$.
	
	\medskip
	
	In order to apply Theorem \ref{morse_fundamental}, we need an efficient way to show that a given element matching is acyclic. The following simple lemma gives such a criterion (See \cite[section 4.2]{jonsson_complex_graph}).
	
	\begin{lemma}[Cycle lemma]\label{cycle} Consider a family $\mathsf{F} \subseteq 2^E$ with an element matching $\mathcal{M}$. Then every directed cycle in $\mathcal{D}(\mathsf{F}, \mathcal{M})$ is of the form \[(\sigma_0, \tau_0, \sigma_1, \tau_1, \dots, \sigma_{t-1}, \tau_{t-1})\] where $t\geq 3$,  $\sigma_i, \sigma_{i+1} \subsetneq \tau_i$, $|\sigma_i|+1 = |\tau_j|$, and $(\sigma_i, \tau_i)\in \mathcal{M}$. (Indices are taken modulo $t$.)
		\begin{comment}
		Let $\PP$ be an order-convex subposet of the face poset of a simplicial complex $K$ and assume that $\MM$ is a matching in $\DD(\PP)$. Then, every directed cycle in $\DD_\MM(\PP)$ is of the form $(\sigma_0, \tau_0, \sigma_1, \tau_1
		, \dots, \sigma_{t-1}, \tau_{t-1})$, where
		\begin{enumerate}
		\item $r\geq 3$,
		\item for each $i\in [r-1]$, there is some $x_i \in \tau_i$ such that $\tau_i=\sigma_i\cup \{x_i\}$ and $(\tau_i,\sigma_i)\in \MM$,
		\item for each $i\in [r-1]$, there is some $y_i \in \tau_i$ such that $\sigma_{i+1} =\tau_i\setminus \{y_i\}$, and
		\item the multisets $\{x_i:i\in [r-1]\}$ and $\{y_i:i\in [r-1]\}$ are equal.
		\end{enumerate}
		\end{comment}
	\end{lemma}

	Here is a simple tool for producing an acyclic matching. (See \cite[Lemma 4.1]{jonsson_complex_graph}.)

	\begin{lemma}\label{boolean}
		Consider a family $\mathsf{F}\subseteq 2^E$ and an element $e_0\in E$. Define 
		\[\begin{array}{rcl}
		\mathsf{F_0} & = & \{\sigma : \sigma-e_0, \sigma+ e_0 \in \mathsf{F}\},\\
		\mathsf{F}_1 & = & \mathsf{F}\setminus \mathsf{F}_0.
		\end{array}\]
		There is a complete acyclic matching $\mathcal{M}_0$ on $\mathsf{F}_0$, and for any acyclic matching $\mathcal{M}_1$ on $\mathsf{F}_1$ the union $\mathcal{M} = \mathcal{M}_0 \cup \mathcal{M}_1$ is an acyclic matching on $\mathsf{F}$. Consequently, the critical sets with respect to $\mathcal{M}$ are precisely the critical sets with respect to $\mathcal{M}_1$.
	\end{lemma}

	By ordering the members of $\mathsf{F} \subseteq 2^E$ by inclusion we may view it as a poset. 
	The following is another useful tool for finding an acyclic matching. (See \cite[Lemma 4.2]{jonsson_complex_graph}.)
	
	\begin{lemma}[Cluster lemma]\label{cluster}
		Let $\mathsf{F}\subseteq 2^E$ and let $\varphi: \mathsf{F}\to Q$ be a monotone poset map where $Q$ is an arbitrary poset. For $q\in Q$, let $\MM_q$ be an acyclic matching on $\varphi^{-1}(q)$.
		Then $\MM=\bigcup_{q\in Q}\MM_q$ is an acyclic matching on $\mathsf{F}$.
	\end{lemma}
	
	Here we give two more tools for constructing acyclic matchings. The first one we call the {\em join construction}. Suppose we have a partition of the ground set $E = E_1\cup \cdots \cup E_m$. Given a family $\mathsf{F}_i\subseteq 2^{E_i}$ for every $i$, the join $\mathsf{F}_1 * \cdots * \mathsf{F}_m$ is the subfamily of $2^E$ defined as
	\[\mathsf{F}_1* \cdots * \mathsf{F}_m = \{\sigma_1 \cup \cdots \cup \sigma_m : \sigma_i\in \mathsf{F}_i\}.\]
	It is important to note that the family $2^E$ by definition contains $2^{|E|}$ distinct subsets, one of which is the empty set $\emptyset$, and that the family $\{\emptyset\} \subseteq 2^E$ should be distinguished from the {\em empty family} $2^E\setminus 2^E$. Suppose $\mathsf{F}_1, \dots, \mathsf{F}_m$ are subfamilies as above. If one of the $\mathsf{F}_i$ is the empty family, then we define the $\mathsf{F}_1* \cdots * \mathsf{F}_m$ to be the empty family.
	
	\medskip

	The following lemma is well-known, but for completeness we include a proof. 
	
	\begin{lemma}[Join Lemma]\label{product}
		Let $E$ be a finite set with partition $E=E_1\cup \cdots \cup E_m$ and for every $i$, let $\mathsf{F}_i\subseteq 2^{E_i}$ be a non-empty subfamily. Suppose $\MM_i$ is an acyclic matching on $\mathsf{F}_i$ with collection of critical sets $\mathsf{U}_i \subseteq \mathsf{F}_i$.
		Then
		there exists an acyclic matching %$\MM$ 
		on the join $\mathsf{F}_1 * \cdots * \mathsf{F}_m$ with collection of critical sets $\mathsf{U}_1 * \cdots * \mathsf{U}_m$. In particular, if one of the $\MM_i$ is complete, then $\mathsf{F}$ has a complete acyclic matching.
	\end{lemma}
	
	\begin{proof}
		After relabeling the parts of the partition, if necessary, we may assume $|\mathsf{U}_1| \leq \cdots \leq |\mathsf{U}_k|$. For each $i\in [m]$ define an element matching
		\[\mathcal{N}_i = \{(\alpha \cup \sigma \cup \beta, \alpha\cup \tau \cup \beta)\}, \]
		where $\alpha \in \mathsf{U}_1 * \cdots * \mathsf{U}_{i-1}$, $(\sigma, \tau)\in \mathcal{M}_i$, and  $\beta \in \mathsf{F}_{i+1} * \dots * \mathsf{F}_m$. 
		In other words, a member of the element matching $\mathcal{N}_i$ arises from a matching in the $i$th component, joined with critical sets from the first $i-1$ components and arbitrary sets from components $i+1, \dots, m$.
		%are critical sets, the $i$th component is a matching from $\mathcal{M}_i$, and the remaining components are arbitrary. 
		If we set $\mathcal{M} = \bigcup_{i=1}^m \mathcal{N}_i$, then it is clear that $\mathcal{M}$ is an element matching on $\mathsf{F} = \mathsf{F}_1 * \cdots * \mathsf{F}_m$ where $\mathsf{U}_1 * \cdots * \mathsf{U}_m$ is the collection of critical sets. Note that if $\mathcal{M}_1$ is a complete matching, then $\mathsf{U}_1$ is the empty family and therefore  $\mathcal{M}$ is a complete matching. It remains to show that $\mathcal{M}$ is acyclic. 
		
		For contradiction, suppose there is a directed cycle 
		\[(\sigma_0, \tau_0, \dots, \sigma_{t-1}, \tau_{t-1})\]
		satisfying Lemma \ref{cycle}. The directed edge $(\sigma_0,\tau_0)$ belongs to some $\mathcal{N}_i$, and therefore $\sigma_0$ and $\tau_0$ are critical in the first $i-1$ components. The set $\sigma_1$ is a subset of $\tau_0$ and is obtained by removing a single element $x$ from $\tau_0$. The element $x$ cannot be removed from a set in the first $i-1$ components of $\tau_0$, since then there would be no way to return to $\sigma_0$ via a matching among the first $i-1$ components. Thus $\sigma_1$ is also critical in the first $i-1$ components, and so are the other $\sigma_j$ for the same reason. And for the same reason again, none of the $\sigma_j$ (with $j\geq 1$) can be critical in its first $i$ or more components, since it would not be possible to return to $\sigma_0$. It follows that all the matchings $(\sigma_j,\tau_j)$ belong to $\mathcal{N}_i$, but this would imply that we only add and remove elements in $E_i$ while we traverse the directed cycle. Therefore, we have a directed cycle in $\DD(\mathsf{F}_i,\mathcal{M}_i)$ which is a contradiction. Thus $\mathcal{M}$ is an acyclic matching.
	\end{proof}
	
	The final tool we call the {\em projection construction}. Suppose we are given a partition of the ground set $E = \bigcup_{i\in I} E_i$. (In other words, the parts of the partition are indexed by the elements of $I$). We define a map
	\[\begin{array}{cccl}
	\pi:& 2^E  & \to & 2^{I} \\
	& \sigma & \mapsto & \{i \in I : \sigma \cap E_{i} \neq \emptyset\},
	\end{array}\]
	which we call the {\em projection map} corresponding to the partition $E = \bigcup_{i\in I}E_i$.

	\begin{lemma}[Projection Lemma]\label{iden}
		Let $E$ be a finite set with partition $E=\bigcup_{i\in I}E_i$ and let $\pi:2^E \to 2^I$ be the corresponding projection map. 
		Given a set $\tau\subseteq E$ and a family $\mathsf{Q}\subseteq 2^I$, define the family $\mathsf{F} = \{\sigma\subseteq E : \pi(\sigma)\in \mathsf{Q}, \tau\subseteq \sigma\}$. Then the following are true:
		\begin{enumerate}
			\item \label{proj lemma easy} $\pi(\mathsf{F})=\{\overline{\sigma}\in \mathsf{Q} : \pi(\tau)\subseteq \overline{\sigma}\}$. 
			\item \label{proj lemma main} Suppose $\pi(\mathsf{F})$ has an acyclic matching $\MM_{\pi(\mathsf{F})}$ with collection of critical sets $\mathsf{U}_{\pi(\mathsf{F})}$. Then 
			there exists an acyclic matching on $\mathsf{F}$ 
			with collection of critical sets $\mathsf{U}_F$, such that the restriction $\pi: \mathsf{U}_F \to \mathsf{U}_{\pi(F)}$ is an injection where $|\sigma|=|\pi(\sigma)|-|\pi(\tau)|+|\tau|$ for every $\sigma \in U_F$.
		\end{enumerate}
	\end{lemma}
	
	\begin{proof} For part {\em (\ref{proj lemma easy})}, it follows from the definition that $\pi(\mathsf{F}) \subseteq \{\overline{\sigma}\in \mathsf{Q}: \pi(\tau) \subseteq \overline{\sigma}\}$. For the reverse inclusion
		consider a set $\overline{\sigma}\in \mathsf{Q}$ such that $\pi(\tau)\subseteq \overline{\sigma}$.
		If we set $\sigma=\bigcup_{i\in \overline{\sigma}}E_i$, then $\pi(\sigma)=\overline{\sigma}$ and $\tau \subseteq \sigma$. 
		Hence, $\overline{\sigma}\in \pi(\mathsf{F})$.
		
		We now prove part {\em (\ref{proj lemma main})}. For a pair $(\gamma_1, \gamma_2) \in \MM_{\pi(\mathsf{F})}$, where $(\gamma_2\setminus \gamma_1) = \{i\}$ for some $i\in I$, define the family
		\[\mathsf{X}_{(\gamma_1, \gamma_2)}=\{\alpha \in \mathsf{F}: \pi(\alpha)=\gamma_1\}*2^{E_i}.\]
		Similarly, for a critical set $\gamma \in \mathsf{U}_{\pi(\mathsf{F})}$ where $\gamma = 
		\{i_1,\dots, i_{|\gamma|}\}\subseteq I$ define the family 
		\[\mathsf{X}_\gamma=\{\alpha \in \mathsf{F}: \pi(\alpha)=\gamma\}= \mathsf{P}(E_{i_1},\tau)*\cdots *\mathsf{P}(E_{i_{|\gamma|}},\tau) *\{\tau\}\]
		where 
		\[\mathsf{P}(E_i,\tau) = 
		\begin{cases}
		2^{E_i}\setminus \{\emptyset \} & \text{ when } (E_i \cap \tau) = \emptyset,\\
		2^{(E_i\setminus \tau)} & \text{ when } (E_i\cap \tau) \neq \emptyset \neq (E_i \setminus \tau),\\
		\{\emptyset\} & \text{ when } E_i\subseteq \tau.
		\end{cases}\]
		Note that this gives us a partition of $\mathsf{F}$ into  \[\mathsf{F} = ( \textstyle{\bigcup} \mathsf{X}_{(\gamma_1,\gamma_2)}) \cup ( \textstyle{\bigcup} \mathsf{X}_{\gamma} ),\] where $(\gamma_1, \gamma_2)$ ranges over all pairs in $\MM_{\pi(\mathsf{F})}$ and $\gamma$ ranges over all critical sets in $\mathsf{U}_{\pi(\mathsf{F})}$. 
		
		By Lemma \ref{boolean} we see that $\mathsf{P}(E_i, \tau)$ has an acyclic matching with a single critical set of size one when  $(E_i\cap \tau) = \emptyset$. 
		If $(E_i\cap \tau)\neq \emptyset$, then $\mathsf{P}(E_i, \tau)$ has a complete acyclic matching when $(E_i\setminus \tau) \neq \emptyset$, and an acyclic matching with a single critical set of size zero when $E_i\subseteq \tau$.
		
		By Lemma \ref{product} it follows that each of the families $\mathsf{X}_{(\gamma_1, \gamma_2)}$ has a complete acyclic matching $\MM_{(\gamma_1, \gamma_2)}$. 
		By the observations above, Lemma \ref{product} implies that there is an acyclic matching $\MM_\gamma$ on the family $\mathsf{X}_\gamma$ which is either complete, or has a single critical set whose size equals $|\tau|$ plus the number of terms in the join for which $E_{i_j}\cap \tau=\emptyset$. That is, there is a single critical set of size  $|\gamma|-|\pi(\tau)|+|\tau|$.
		
		We set $\MM=(\bigcup\MM_{(\gamma_1,\gamma_2)})\cup (\bigcup\MM_\gamma)$, where $(\gamma_1, \gamma_2)$ ranges over all pairs in $\MM_{\pi(\mathsf{F})}$ and $\gamma$ ranges over all critical sets in $\mathsf{U}_{\pi(\mathsf{F})}$. 
		Clearly, $\MM$ is an element matching on $\mathsf{F}$ with family of critical sets $\mathsf{U}_\mathsf{F}$ such that the restriction $\pi: \mathsf{U}_\mathsf{F}\to \mathsf{U}_{\pi(\mathsf{F})}$ is an injection where $|\gamma|=|\pi(\gamma)|-|\pi(\tau)|+|\tau|$ for every $\gamma \in \mathsf{U}_{\mathsf{F}}$.
		It remains to show that $\MM$ is acyclic. 
		
		For contradiction, suppose there is a directed cycle 
		%
		%
		%\[(\sigma_0, \tau_0, \dots,\sigma_{r-1}, %\tau_{r-1})\]
		satisfying Lemma \ref{cycle}.
		We traverse this cycle, keeping track of which part in our partition of $\mathsf{F}$ we are currently in, and record every directed edge $(\tau_i, \sigma_{i+1})$ which goes between distinct parts. It is easily seen that this results in a non-empty (circular) subsequence $(\sigma_{1}, \tau_{1}, \dots, \sigma_{t}, \tau_{t})$ together with a (circular) sequence of families $(X_1, \dots, X_t)$ where
		\begin{itemize}
			\item $X_i = \mathsf{X}_{(\gamma_1, \gamma_2)}$ or $\mathsf{X}_\gamma$, for some pair $(\gamma_1, \gamma_2) \in \MM_{\pi(F)}$ or some $\gamma \in \mathsf{U}_{\pi(F)}$,
			\item $\sigma_{i}, \tau_{i} \in X_i$ for every $i\in [t]$,
			\item $\tau_{i} \supsetneq \sigma_{i+1}$ and $X_i\ne X_{i+1}$ (indices are taken modulo $t$), and
			\item $s=|\tau_i|=|\sigma_{j}|+1$ for all $i,j\in [t]$. \end{itemize}
		Note that $\pi(\tau_i) \neq \pi(\sigma_{i+1})$ or else we would have $X_i = X_{i+1}$. It follows that $|\pi(\tau_i)|= |\pi(\sigma_{i+1})|+1$ for all $i\in [t]$, which implies $s' = |\pi(\tau_i)| = |\pi(\sigma_j)|+1$ for all $i,j\in [t]$, since the sequence is circular. Note that if $\sigma_i, \tau_i \in \mathsf{X}_{\gamma}$, then $\pi(\sigma_i) = \pi(\tau_i)$. Therefore it must be the case that every $X_i$ is of the type $\mathsf{X}_{(\gamma_1, \gamma_2)}$.
		% 
		%
		%
		%
		%Note that $\pi(\tau_i)\supseteq \pi(\sigma_{i+1})$ for every $i$. 
		%
		%Consider the sequence $(|\pi(\sigma_1)|, |\pi(\tau_1)|, \dots, |\pi(\sigma_r)|, |\pi(\tau_r)|)$. 
		%{\em ... to be completed}
		%By the second and third condition, we have $\pi(\tau_i)\supsetneq \pi(\sigma_{i+1})$ for every $i\in [r]$. And by the fourth condition, we must have $|\pi(\tau_i)|= |\pi(\sigma_{i+1})|+1$ for every $i\in [r]$. Also, by the first and second condition, the difference between $|\pi(\sigma_i)|$ and $|\pi(\tau_i)|$ is bounded by 1. To go through the circular sequence of sizes from $|\pi(\sigma_1)|$ and come back to itself, we should have $|\pi(\tau_i)|=|\pi(\sigma_i)|+1$ for every $i\in [r]$, which implies that every $X_i$ is $X_{(\gamma_1,\gamma_2)}$ for some pair $(\gamma_1,\gamma_2)\in \MM_{\pi(F)}$. 
		But this means that we can find a subsequence of $(\pi(\sigma_1), \pi(\tau_1), \dots, \pi(\sigma_r), \pi(\tau_r))$ which induces a directed cycle in $\mathcal{D}(\pi(\mathsf{F}), \MM_{\pi(\mathsf{F})})$, which contradicts the assumption that $\MM_{\pi(\mathsf{F})}$ is an acyclic matching on $\pi(\mathsf{F})$.
	\end{proof}

	\subsection{The Gallai--Edmonds decomposition} \label{section_ge}
	Let $G$ be a graph on the vertex set $V$. There is a canonical partition of the vertex set \[V = D \cup A \cup C,\] which is useful for describing the structure of all maximum matchings in $G$. The parts, $D$, $A$, and $C$, are defined as
	\begin{align*}
	D &= D(G) = \{v\in V : \nu(G-v) = \nu(G)\}, \\ 
	A &= A(G) = N_G(D), \\ 
	C &= C(G) = V\setminus (D\cup A).
	\end{align*}
	We further partition $D$ into subparts $D = D_1\cup \cdots \cup D_r$ such that the each induced subgraph $G[D_i]$ is a connected component of $G[D]$. Each $D_i$ is called a {\em component of $D$}.
	
	This canonical partition of the vertex set of $G$ is called the {\em Gallai--Edmonds decomposition of $G$}, and is denoted as $(D_1, \dots, D_r ; A; C)$. 
	
	%Note that if  $G$ is bipartite, then there is no factor critical subgraph having more than one vertex, so we have only trivial components of $D$. Thus, for a bipartite graph $G$, we simply denote the Gallai-Edmonds decomposition by $(D, A, C)$.
	
	\begin{remark}
		The Gallai-Edmonds decomposition of $G$ is often expressed only as $(D; A; C)$. For our purpose it will be important to take the components of $D$ into account, and by our notation we have  $D = \bigcup D_i$. 
	\end{remark}
	
	Let $V$ be a vertex set. We say that a graph $M$ on $V$ \textit{is a matching on $V$} if $\deg_M(v)\leq 1$ for every $v\in V$, and that a vertex $v$ is covered by $M$ if $\deg_M(v)=1$. For a subset $W\subseteq V$, we say that $W$ is covered by $M$ if $w$ is covered by $M$ for every $w\in W$. (Note that when $W$ is empty, the empty graph $\emptyset$ is vacuously a matching covering $W$). Furthermore, we say that $M$ is \textit{a perfect matching on $V$} if $V$ is covered by $M$.  
	
	A graph $G$ on the vertex set $V$ is called \textit{factor critical on $V$} if for every vertex $v\in V$, the graph $G-v$ has a perfect matching on $V\setminus \{v\}$. It is easily seen that if $G$ is factor critical on $V$, then $G$ is connected and $|V|$ must be odd. (Note that if $|V|=1$, then the empty graph $\emptyset$ is factor critical on $V$.) 
	
	With these notions in place, the Gallai--Edmonds decomposition of a graph has the following properties. For a more detailed discussion, see \cite{lovasz_plummer}. 
	\begin{thm}[Gallai--Edmonds Decomposition Theorem] \label{ge-decomposition}
		Let $G$ be a graph on the vertex set $V$ with Gallai--Edmonds decomposition $(D_1,\dots, D_r;A;C)$. Let $D = \bigcup D_i$. Then the following hold.
		\begin{enumerate}[(1)]
			\item For each $D_i$, $G[D_i]$ is factor critical on $D_i$.
			\item $G[C]$ has a perfect matching on $C$.
			\item \label{ge bipartite condition} For every $i\in [r]$, there is a matching $M_i$ in $G[D\setminus D_i,A]$ covering $A$ such that $|N_{M_i}(A)\cap D_j|\leq 1$ for every $j\in [r]$.
			
			\item \label{ge equation} $G[D]$ has exaclty $|A|+|V| - 2 \nu(G)$ connected components, that is,  
			\[r = |A|+|V| - 2 \nu(G).\]
		\end{enumerate}
	\end{thm}
	
	\begin{remark}
		Note that by (\ref{ge bipartite condition}) of Theorem \ref{ge-decomposition}, the number of components of $D$ is strictly greater than $|A(G)|$ whenever $A(G)$ is non-empty. Also the equation in (\ref{ge equation}) of Theorem \ref{ge-decomposition} can be rewritten as
		\[\textstyle{\sum_{i=1}^r}(|D_i|-1) + 2|A| + |C| = 2\nu(G).\]
	\end{remark}
	
	\begin{remark} \label{ge consequences}
		One consequence of Theorem \ref{ge-decomposition} is the following description of the maximum matchings in the graph $G$: Each maximum matching of $G$ consists of
		\begin{itemize}
			\item a perfect matching on $G[C]$,
			\item an edge $ad_a$ for each $a\in A$, where $d_a \in D$ and where $d_a$ and $d_b$ are in distinct component of $D$ for distinct vertices $a, b\in A$, and
			\item a matching of size $(|D_i|-1)/2$ on each component $D_i$ of $D$.
		\end{itemize}
	\end{remark}

	It is useful to know how the Gallai--Edmonds decomposition of a graph is affected by adding or deleting a single edge. One such criterion is given by the following.
	
	\begin{lemma} \label{same-ge} Let $G$ be a graph with  Gallai--Edmonds decomposition $(D_1, \dots, D_r;A;C)$. If $e \in K_A \cup K_{A,C}$, then $G+e$ and $G-e$ have the same Gallai--Edmonds decomposition as $G$. 
	\end{lemma}
	
	\begin{proof} We first prove that $G+e$ has the same Gallai--Edmonds decomposition when $e \notin G$. Note that a maximum matching in $G+e$ does not use the edge $e$, because any matching containing the edge $e$ would cover less vertices of $D$ than the maximum matchings in $G$. Therefore $G$ and $G+e$ have exactly the same sets of maximum matchings. Since the part $D$ of the Gallai--Edmonds decomposition is completely determined by the collection of maximum matchings in $G$, it follows that $D(G+e) = D(G)$. Since $e$ is not incident to any vertex in $D$, it follows that $A(G+e) = A(G)$ and $C(G+e) = C(G)$. Finally, adding the edge $e$ to $G$ does not change the connected components of $D(G)$, and therefore the Gallai-Edmonds decompositions are the same. The proof for $G-e$ when $e\in G$ is similar and we leave it to the reader.  
	\end{proof}
	
	\subsection*{Prescribed Gallai--Edmonds decompositions} Suppose we are given a family $\mathsf{F}$ of graphs on a  vertex set $V$, and
	we want to find an acyclic matching on $\mathsf{F}$. The main technique, introduced in \cite{linusson_bounded_matching_complex},  is to partition $\mathsf{F}$ according to their Gallai--Edmonds decompositions and then find acyclic matchings for each individial part.
	
	\vspace{1ex}
	
	For the family $\mathsf{F}$, 
	let $\mathsf{F}_{(D_1, \dots, D_r; A; C)}\subseteq \mathsf{F}$ denote the subfamily of graphs with Gallai--Edmonds decomposition $(D_1, \dots, D_r; A; C)$. Note that for certain partitions of $V$ the subfamily $\mathsf{F}_{(D_1, \dots, D_r; A; C)}$ could be empty, but the collection of all the non-empty subfamilies gives us a partition of $\mathsf{F}$. In the specific case when all the graphs in $\mathsf{F}$ have the same matching number, we have the following.
	
	\begin{lemma}\label{combine-ge}
		Let $\mathsf{F}$ be a family of graphs on the vertex set $V$, where all members of $\mathsf{F}$ have the same matching number. Suppose for each non-empty subfamily $\mathsf{F}_{(D_1, \dots, D_r;A;C)} \subseteq \mathsf{F}$ we have an acyclic matching. Then, the union of these acyclic matchings is an acyclic matching on $\mathsf{F}$.
	\end{lemma}
	
	\begin{proof}
		Let $\mathcal{M}$ denote the union of the acyclic matchings. If $\mathcal{M}$ is not acyclic, then by Lemma \ref{cycle} there exists a directed cycle 
		\[(\sigma_0, \tau_0, \sigma_1, \tau_1, \dots, \sigma_{t-1}, \tau_{t-1})\]
		where $\sigma_i$ and $\tau_i$ have the same Gallai--Edmonds decompositions, and $\sigma_{i+1} \subsetneq \tau_i$, for every $i$ (indices are taken modulo $t$). We are going to show that the assumption on the matching number of the graphs in $\mathsf{F}$ implies that $\tau_i$ and $\sigma_{i+1}$ also have the same Gallai--Edmonds decomposition. Therefore such a directed cycle would belong to a single subfamily $\mathsf{F}_{(D_1, \dots, D_r; A; C)}$, contradicting the assumption that each of these matchings were acyclic.
		
		\medskip
		
		Consider graphs $G_1\subseteq G_2$ on the same vertex set with $\nu(G_1)=\nu(G_2)$. Observe that any maximum matching in $G_1$ is also a maximum matching in $G_2$, which implies that \[D(G_1)\subseteq D(G_2).\] Furthermore, for any vertex $v\in D(G_1)$ we have $N_{G_1}(v) \subseteq N_{G_2}(v)$ which implies that
		\[D(G_1)\cup A(G_1)\subseteq D(G_2)\cup A(G_2).\]
		
		Returning to the directed cycle $(\sigma_0, \tau_0, \sigma_1, \tau_1, \dots, \sigma_{t-1}, \tau_{t-1})$, the observation above implies that
		\[D(\sigma_0) = D(\tau_0) \supseteq D(\sigma_1) = D(\tau_1) \supseteq \cdots \supseteq D(\sigma_{t-1}) = D(\tau_{t-1}) \supseteq D(\sigma_0),\]
		and therefore $D = D(\sigma_i) = D(\tau_j)$ for all $i$ and $j$. This in turn implies, by the same argument, that $A = A(\sigma_i) = A(\tau_j)$ and $C = C(\sigma_i) = C(\tau_j)$ for all $i$ and $j$.
		
		It remains to show that the components of $D(\tau_i)$ and $D(\sigma_{i+1})$ are the same. By Lemma \ref{cycle},   $\sigma_{i+1}$ is obtained from $\tau_i$ by removing a single edge. It follows that the only change that could occur when we pass from $\tau_i$ to $\sigma_{i+1}$ is that we increase the number of components of $D$. But since the number of components is uniquely determined by $|A|$, $|V|$, and the matching number, it follows that $D(\tau_i)$ and $D(\sigma_{i+1})$ have the same number of components, and so the components must remain the same. This shows that all the elements of the directed cycle have the same Gallai--Edmonds decomposition. 
	\end{proof}
	
	\section{Three special families of graphs} \label{three special}
	
	The basic strategy of the proof of Theorem \ref{link-result} is to decompose our family of graphs into a join, where each term of the join is built up from one of three special families of graphs. The purpose of this section is to define these families. They are essentially motivated by the parts appearing in the Gallai--Edmonds decomposition theorem (Theorem \ref{ge-decomposition}) and the properties described in Remark \ref{ge consequences}. We also give key results concerning acyclic matchings of each of these families together with bounds on the sizes of the critical sets. The proofs of these results will be given in sections \ref{section_size perfect m}, \ref{section_size-d0} and \ref{section_size-yz}.

	\subsection{Perfect matchings}
	Recall that a graph $G$ on the vertex set $V$ has a perfect matching if there is a matching that covers $V$. Note that if $G$ has a perfect matching on $V$, then $|V|$ must be even. 
	
	For a fixed graph $H\subseteq K_V$, define the family 
	\[\mathsf{PM}_H = \{G\subseteq K_V : G \text{ has a perfect matching on }V, H \subseteq G\}.\]
	Note that the family $\mathsf{PM}_H$ is non-empty if and only if $|V|$ is even. When $V$ is empty, we set $\mathsf{PM}_H=\{\emptyset \}$ by convention.
	In section \ref{section_size perfect m} we prove the following.
	
	\begin{prop}\label{size perfect m}
		Given a graph $H\subseteq K_V$ where $|V|$ is even. There exists an acyclic matching on $\mathsf{PM}_H$ such that any critical set $\sigma$ satisfies 
		\[|\sigma| \leq \textstyle{\frac{3}{2}}|V| +|H|.\]
		Moreover, the inequality is strict whenever $V$ is non-empty.
	\end{prop}
	
	\subsection{Factor critical graphs}
	Recall that a graph $G$ is {\em factor critical} on the vertex set $V$ if for every $v\in V$ the induced subgraph $G-v$ has a perfect matching. Note that if $G$ is factor critical on $V$, then $|V|$ must be odd.
	
	For a fixed graph $H\subseteq K_V$, define the family 
	\[\mathsf{FC}_{H} = \{G\subseteq K_V : G \text{ is factor critical on } V, H\subseteq G\}.\]
	Note that when $|V|=1$, then $\mathsf{FC}_{H} = \{\emptyset\}$.
	In section \ref{section_size-d0} we prove the following
	
	\begin{prop}\label{size-d0} 
		Given a graph $H\subseteq K_V$ where  $|V|$ is odd. There exists an acyclic matching on $\mathsf{FC}_{H}$ such that any critical set $\sigma$ satisfies 
		\[|\sigma|\leq \textstyle{\frac{3}{2}}(|V|-1)+|H|.\]
		Moreover, the inequality is strict whenever $H$ contains at least one edge. 
	\end{prop}
	
	\subsection{The bipartite case} It is easy to see that a bipartite graph can not be factor critical, so instead we deal with some variations of this notion. 
	Let $G$ be a bipartite graph with vertex classes $X$ and $Y$. We say that $G$ is {\em $Y$-factor critical} if for every vertex $x\in X$, the graph $G-x$ has a matching which covers $Y$. Note that if $G$ is $Y$-factor critical, then we must have $|X|>|Y|$. (If $Y=\emptyset$, then by convention we say that the empty graph is $Y$-factor critical)
	\begin{remark} \label{Halls observation}
		By Hall's marriage theorem it is easily seen that $G$ is $Y$-factor critical if and only if $|N_G(Y')|>|Y'|$ for every non-empty subset $Y'\subseteq Y$.
	\end{remark}
	
	Now we give an extension of the notion of $Y$-factor critical graphs. As before let $G$ be a bipartite graph with vertex classes $X$ and $Y$. Fix a subset  
	$Z\subset X$. We say that the bipartite graph $G$ is {\em $(Y,Z)$-factor critical} if $G$ is $Y$-factor critical and the induced subgraph  $G[Z,Y]$ is $Z$-factor critical. Note that if $G$ is $Y$-factor critical, then we must have $|X|>|Y|$ when $|Y|>0$, and $|Y|>|Z|$ when $|Z|>0$. %So we assume these conditions.
	
	When $Z$ is empty, then $G$ is $(Y,Z)$-factor critical if and only if $G$ is $Y$-factor critical. Moreover, when $Y$ is empty, $Z$ should  also be empty to satisfy the inequality condition, so by convention the empty graph $\emptyset$ is $(\emptyset,\emptyset)$-factor critical. Note that if $Y$ and $Z$ are both non-empty, and $G$ is $(Y,Z)$-factor critical, then we must have $|Z|<|Y|<|X|$.
	
	For a fixed bipartite graph $H\subseteq K_{X,Y}$ and a subset $Z\subseteq X$, define the family
	\[\mathsf{BFC}_{(X,Y,Z; H)} = \{G\subseteq K_{X,Y}: G \text{ is $(Y,Z)$-factor critical}, H\subseteq G\}.\]
	Note that as long as we have $|X|>|Y|$ when $|Y|>0$, and $|Y|>|Z|$ when $|Z|>0$, then the family $\mathsf{BFC}_{(X,Y,Z; H)}$ is non-empty. When $X$ or $Y$ is empty, we set $\mathsf{BFC}_{(X,Y,Z; H)}=\{\emptyset \}$ by convention.
	In section \ref{section_size-yz} we prove the following.
	
	\begin{prop}\label{size-yz}
		Given a bipartite graph $H\subseteq K_{X,Y}$ and a subset $Z\subseteq X$. There exists an acyclic matching on $\mathsf{BFC}_{(X,Y,Z; H)}$ such that any critical set $\sigma$ satisfies 
		\[|\sigma|\leq 2|Y|+|Z| + |H|.\]
		Moreover, the inequality is strict whenever $H$ contains at least one edge.
	\end{prop}

	%	\section{Proof of Theorem \ref{leray-result} and \ref{link-result}} \label{proof of main}
	\subsection{Structure of the proofs}\label{section_outline}
	
	The families of graphs and the bounds given in Propositions \ref{size perfect m}, \ref{size-d0}, and \ref{size-yz} are the most important technical tools needed for the proof of Theorems \ref{leray-result} and \ref{link-result}. The proofs of these propositions and their roles in the proof of Theorems \ref{leray-result} and \ref{link-result} is quite involved. In the next section we prove Theorems \ref{leray-result} and \ref{link-result} (modulo Propositions \ref{size perfect m}, \ref{size-d0}, and \ref{size-yz}). 
	
	\smallskip
	
	At this point we warn the reader that the order in which we prove the three propositions is in a sense the ``reverse'' of the logical order. More specifically, the proof of Proposition~\ref{size perfect m} for the family $\mathsf{PM}_H$ is given in section~\ref{section_size perfect m} and assumes the validity of Propositions~\ref{size-d0} and~\ref{size-yz}.
	The proof of Proposition \ref{size-d0} for the family $\mathsf{FC}_H$ is given in section~\ref{section_size-d0} and assumes the validity of Proposition~\ref{size-yz}. Finally, the proof of Proposition~\ref{size-yz} for the family $\mathsf{BFC}_{(X,Y,Z;H)}$ is given in section \ref{section_size-yz}.
	
	The reason for ``reversing'' the logical order is that the family $\mathsf{PM}_H$ is conceptually the simplest one and the methods for decomposing this family into parts are easier to explain. These methods and ideas will be refined and used again for the other families as well. At the other end, the family $\mathsf{BFC}_{(X,Y,Z;H)}$, which everything depends on, is the most technical one and requires the deepest analysis. We therefore postpone this one to the end.

	The whole structure of the proof is given in Figure \ref{fig_outline}.

	\begin{figure} 
		\centering
		\begin{tikzpicture}
		
		\begin{scope}[xshift = 0cm, yshift = 14cm]
		\fill[blue, opacity=.05] (-2.5,-.35) --++ (0,1.2) --++ (5,0) --++ (0,-1.2) -- cycle;
		\draw[semithick] (-2.5,-.35) --++ (0,1.2) --++ (5,0) --++ (0,-1.2) -- cycle;
		\node at (0,0.4) {Proposition \ref{size-yz} for $\BFC_{(X,Y,Z;H)}$};
		\node at (0,0) {\footnotesize (Proof in section 7)};
		\end{scope}

		\node at (0,12.3) {\footnotesize Join and Projection};
		
		\draw[-Stealth,double, semithick] (.5,10.4) --++ (-1,0);
		\node at (0,10.7) {\footnotesize Join};
		
		\draw[-Stealth,double, semithick] (-2,12.8) --++ (-.5,-1);
		
		\draw[-Stealth,double, semithick] (2,12.8) --++ (.5,-1);
		
		\draw[-Stealth, double, semithick] (-3,14) arc[start angle=105, end angle=200, radius = 4cm];
		\node at (-5.5,13.3) {\footnotesize Join and};
		\node at (-5.7,13) {\footnotesize Projection};
		
		\draw[-Stealth, double, semithick] (3,14) arc[start angle=75, end angle=-20, radius = 4cm];
		\node at (5.5,13) {\footnotesize Join};
		
		\begin{scope}[xshift = 3cm, yshift = 10.2cm]
		\fill[blue, opacity=.05] (-2,-.35) --++ (0,1.2) --++ (4,0) --++ (0,-1.2) -- cycle;
		\draw[semithick] (-2,-.35) --++ (0,1.2) --++ (4,0) --++ (0,-1.2) -- cycle;
		\node at (0,0.4) {Proposition \ref{size-d0} for $\FC_{H}$};
		\node at (0,0) {\footnotesize (Proof in section 6)};
		\end{scope}
		
		\begin{scope}[xshift = -3cm, yshift = 10.2cm]
		\fill[blue, opacity=.05] (-2,-.35) --++ (0,1.2) --++ (4,0) --++ (0,-1.2) -- cycle;
		\draw[semithick] (-2,-.35) --++ (0,1.2) --++ (4,0) --++ (0,-1.2) -- cycle;
		\node at (0,0.4) {Proposition \ref{size perfect m} for $\PM_{H}$};
		\node at (0,0) {\footnotesize (Proof in section 5)};
		\end{scope}
		
		\node at (-3.1,9.1) {\footnotesize Join};
		
		\node at (-0.6,8.7) {\footnotesize Join};
		
		\draw[-Stealth,double, semithick] (1,9.5) arc[start angle = -40, end angle = -70, radius = 6cm];
		\draw[-Stealth,double, semithick] (-3.5,9.5) --++ (-.3,-1);
		
		\begin{scope}[yshift = 7.8cm, xshift = -4cm]
		\fill[blue, opacity=.05] (-1.7,-.69) --++ (0,1.05) --++ (3.4,0) --++ (0,-1.05) -- cycle;
		\draw[semithick] (-1.7,-.69) --++ (0,1.05) --++ (3.4,0) --++ (0,-1.05) -- cycle;
		\node at (0,0) {Proposition \ref{special case}};
		\node at (0,-.39) {\footnotesize (Proof in section 4.1)}; 
		\end{scope}
		
		\begin{scope}[yshift=7.8cm, xshift = 4cm]
		\fill[blue, opacity=.05] (-1.7,-.69) --++ (0,1.05) --++ (3.4,0) --++ (0,-1.05) -- cycle;
		\draw[semithick] (-1.7,-.69) --++ (0,1.05) --++ (3.4,0) --++ (0,-1.05) -- cycle;
		\node at (0,0) {Proposition \ref{bipartite case}};
		\node at (0,-.39) {\footnotesize (Proof in section 4.2)}; 
		\end{scope}
		
		\begin{scope}[yshift = 6cm]
		\node at (0,.4) {\footnotesize weak Morse inequality};
		\node at (0,0) {\footnotesize (Theorem \ref{morse_fundamental})};
		\end{scope}
		
		\draw[-Stealth,double, semithick] (3.7,6.75) --++ (-.3,-1);
		\draw[-Stealth,double, semithick] (-3.7,6.75) --++ (.3,-1);
		
		\begin{scope}[yshift = 5cm, xshift = -3.3cm]
		\fill[blue, opacity=.05] (-1.7,-.35) --++ (0,0.75) --++ (3.4,0) --++ (0,-.75) -- cycle;
		\draw[semithick] (-1.7,-.35) --++ (0,0.75) --++ (3.4,0) --++ (0,-.75) -- cycle;
		\node at (0,0) {Theorem \ref{link-result} for $K_V$};
		\end{scope}
		
		\begin{scope}[yshift = 5cm, xshift = 3.3cm]
		\fill[blue, opacity=.05] (-1.8,-.35) --++ (0,0.75) --++ (3.6,0) --++ (0,-.75) -- cycle;
		\draw[semithick] (-1.8,-.35) --++ (0,0.75) --++ (3.6,0) --++ (0,-.75) -- cycle;
		\node at (0,0) {Theorem \ref{link-result} for $K_{X,Y}$};
		\end{scope}
		
		\node at (0,3.9) {\footnotesize Proposition \ref{hereditary}};
		\draw[-Stealth,double, semithick] (-1.8,4.25) --++ (1,-1);
		\draw[-Stealth,double, semithick] (1.8,4.25) --++ (-1,-1);
		
		\begin{scope}[yshift = 2.5cm]
		\fill[blue, opacity=.05] (-1.6,-.68) --++ (0,1) --++ (3.2,0) --++ (0,-1) -- cycle;
		\draw[semithick] (-1.6,-.68) --++ (0,1) --++ (3.2,0) --++ (0,-1) -- cycle;
		\node at (0,0) {Theorem \ref{link-result}};
		\node at (0,-.4) {\footnotesize (Proof in section 4)}; 
		\end{scope}
		
		\draw[-Stealth,double, semithick] (0,1.5) --++ (0,-1);
		\node at (1,1) {\footnotesize Corollary \ref{near d ler}};
		
		\begin{scope}
		\fill[blue, opacity=.05] (-1.1,-.35) --++ (0,.65) --++ (2.2,0) --++ (0,-.65) -- cycle;
		\draw[semithick] (-1.1,-.35) --++ (0,.65) --++ (2.2,0) --++ (0,-.65) -- cycle;
		\node at (0,0) {Theorem \ref{leray-result}};
		\end{scope}
		\end{tikzpicture}
		\caption{Outline of the proof of Theorem \ref{leray-result}. Here, \textit{Join} and \textit{Projection} refer to the Join lemma \ref{product} and the Projection lemma \ref{iden}.}
		\label{fig_outline}
	\end{figure}
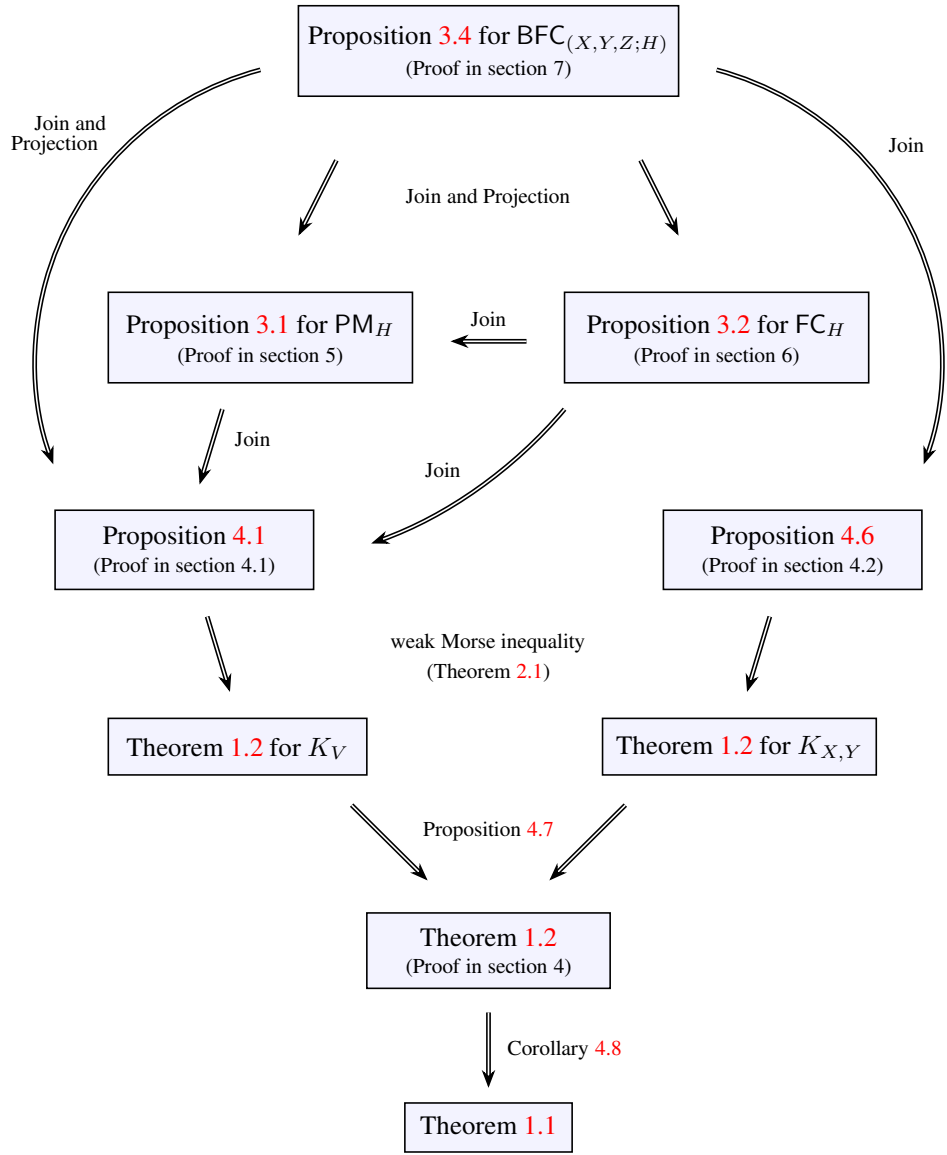

	\section{Proof of Theorem \ref{leray-result} and \ref{link-result}} \label{proof of main}
	
	In this section we temporarily assume validity of Propositions \ref{size perfect m}, \ref{size-d0} and \ref{size-yz} for families of graphs $\mathsf{PM}_H$, $\mathsf{FC}_H$ and $\mathsf{BFC}_{(X,Y,Z;H)}$ introduced in Section \ref{three special}. (These will be proved in subsequent sections.) We use these propositions and two technical lemmas of discrete Morse theory - the join lemma \ref{product} and the projection lemma \ref{iden} to obtain Theorem \ref{link-result}. We first prove Theorem \ref{link-result} for the case when $G$ is the complete graph $K_V$ in Section \ref{complete case}. This proof contains all the main ideas and in Section \ref{complete bipartite} we show how the arguments can be modified to deal with the case when $G$ is a complete bipartite graph $K_{X,Y}$. Finally we deduce the theorem for arbitrary graphs $G$ by a general argument based on simplicial homology which is given in Section \ref{homology}.
	
	\subsection{Complete graphs} \label{complete case}
	Fix a graph $H\subseteq K_V$ with $1\leq \nu(H)<k$ and 
	define the family
	\[\mathsf{F}_H = \{ G\subseteq K_V : \nu(G)< k, H\subseteq G\}. \]
	\begin{prop}\label{special case}
		There is an acyclic matching on $\mathsf{F}_H$ such that any critical set $\sigma$ satisfies 
		\[|\sigma| \leq 3k-4+|H|.\]
	\end{prop}
	
	Let us denote the link of $H$ in $\mathsf{NM}_k(K_V)$ by $\mathsf{L}_H$.
	Note that the face poset of $\mathsf{L}_H$ is isomorphic to $\mathsf{F}_H$ (where the members of $\mathsf{F}_H$ are ordered by inclusion). In fact, we have $\mathsf{F}_H = \mathsf{L}_H*\{H\}$, that is, every member of $\mathsf{F}_H$ can be obtained by adding $H$ to a member of $\mathsf{L}_H$, and every member of the link $\mathsf{L}_H$ can be obtained by removing $H$ from a member of $\mathsf{F}_H$. Therefore the acyclic matching in Proposition \ref{special case} can be transformed into an acyclic mathcing on $\mathsf{L}_H$ where any critical set $\sigma$ satisfies $|\sigma|\leq 3k-4$,
	and by Theorem \ref{morse_fundamental} it follows that $\mathsf{L}_H$ has vanishing homology in all dimensions $d\geq 3k-4$.
	
	\medskip
	
	We now start our proof of Proposition \ref{special case}. 
	The strategy is to decompose the family $\mathsf{F}_H$ into simpler parts which can be expressed as joins of the families defined in section \ref{three special}. Propositions \ref{size perfect m}, \ref{size-d0}, \ref{size-yz} together with 
	Lemma \ref{product} allow us to obtain an acyclic matching on $\mathsf{F}_H$ with the desired bound on the size of the critical sets.
	
	\subsection*{First reduction}
	We start by observing that when $|V|<2k$, then the condition $\nu(G)<k$ is satisfied for any $G\subseteq K_V$. This means that $\mathsf{F}_H = \{G\subseteq K_V : H\subseteq G\}$, and we can find either a complete acyclic matching, or an acyclic matching with a single critical set of size $|H|$ (by Lemma \ref{boolean}). In either case we are done, so from here on we assume that $|V|\geq 2k$. This also implies that $H$ is a proper subgraph of $K_V$. 
	
	Without loss of generality, let $v$ be the vertex of minimum degree in $H$, that is, \[\deg_{H}(v)=\min\{\deg_H(w):w\in V\}.\] Note that the degree of $v$ in $H$ could equal zero. Let $V'=V\setminus \{v\}$. If we set $W = (V'\setminus N_H(v))$, 
	then the following properties are satisfied:
	\begin{enumerate}[(i)]
		\item $W\neq\emptyset$, and
		\item \label{non-triv-H} $H$ has an edge not incident with $v$.
	\end{enumerate}
	Let $S$ denote the set of edges in $K_V$ which are incident to $v$ but do not belong to $H$, that is, $S = K_{W, \{v\}}$.
	For every $G\in \mathsf{F}_H$ define $S_G\subseteq S$ as
	\[S_G=\{e\in S: G+e \in \mathsf{F}_H\}.\]
	Now define subfamilies
	\[\begin{array}{lcr}
	\mathsf{F}_0 & = & \{G\in \mathsf{F}_H : S_G \ne \emptyset \}, \\
	\mathsf{F}_1 & = & \{G\in \mathsf{F}_H : S_G = \emptyset \}.
	\end{array}\]
	Thus we have a partition $\mathsf{F}_H = \mathsf{F}_0 \cup \mathsf{F}_1$.
	
	\begin{claim} \label{step0}
		There is a complete acyclic matching $\MM_0$ on $\mathsf{F}_0$. Furthermore, if $\MM_1$ is any acyclic matching on $\mathsf{F}_1$, then $\MM=\MM_0 \cup \MM_1$ is an acyclic matching on $\mathsf{F}_H$.
	\end{claim}

	\begin{comment}
	\newpage
	
	\subsection{Graphs with a fixed Gallai--Edmonds decomposition} Let $H'$ be a graph on the vertex set $V$ with $1\leq \nu(H')< k$. 
	Define the family
	\[\mathsf{F} = \{G\subseteq K_V : H'\subseteq G, \nu(G) = k-1\},\]
	and consider a non-empty subfamily $\mathsf{F}_{(D_1, \dots, D_r;A; C)}\subseteq \mathsf{F}$. 
	%Note that by Theorem \ref{ge-decomposition}, the partition $V = D_1\cup \cdots \cup D_r \cup A \cup C$ satisfies 
	%\[\textstyle{\sum_{i=1}^c}(|D_i|-1) + 2|A| + |C| = 2(k-1).\]
	We have the following.
	\begin{prop} \label{main bound}
	The family $\mathsf{F}_{(D_1\, \dots, D_r; A; C)}$ has an acyclic matching where any critical set $\sigma$ satisfies $|\sigma| \leq 3k-4 + |H'|$.
	\end{prop}
	
	This is our most technical result. The proof is quite involved and spans over sections ... For the remainder of this section we show how Theorem \ref{link-result} can be deduced from Proposition \ref{main bound}.
	
	\subsection{Proof of Theorem \ref{link-result}} 
	
	\end{comment}
	
	\begin{proof}
		We first prove the existence of a complete acyclic matching on $\mathsf{F}_0$. Let $\mathsf{Q} = \{ (G\setminus S): G \in \mathsf{F}_0\}$, which is clearly a subfamily of $\mathsf{F}_0$. 
		Consider the map $\varphi: \mathsf{F}_0 \to \mathsf{Q}$ defined as $\varphi(G)=(G\setminus S)$. Note that $f$ is monotone with respect to inclusion. For every $G\in \mathsf{Q}$ we have $\varphi^{-1}(G) = \{G\cup S' : S'\subseteq S_G\}$; If not, there would exists a subset $S'\subseteq S_G$ such that $\nu(G\cup S')= k$. Then, a maximum matching of size $k$ in $G\cup S'$ should use exactly one edge $e$ in $S'$, which implies that $\nu(G+e)=k$. This contradicts the assumption that $e \in S_G$. 
		Hence, for any edge $e_0 \in S_G$, we have
		$$\varphi^{-1}(G)=\{\sigma: \sigma+e_0, \sigma-e_0  \in\varphi^{-1}(G)\},$$	
		and by Lemma \ref{boolean}
		we can find a complete acyclic matching on $\varphi^{-1}(G)$ for every $G\in \mathsf{Q}$. By this and  Lemma \ref{cluster}, there is a complete acyclic matching $\MM_0$ on $\mathsf{F}_0$.
		
		Now consider an (arbitrary) acyclic matching  $\MM_1$ on $\mathsf{F}_1$. 
		Set $\MM=\MM_0\cup \MM_1$, and for contradiction assume there is a directed cycle
		\[(\sigma_1, \tau_1, \dots, \sigma_t,\tau_t),\] which satisfies the conditions of Lemma \ref{cycle}. Consider the case that $(\sigma_i,\tau_i)\in \MM_0$ for some $i$. This means that $\tau_i\setminus \sigma_i=\{e_i \} \subseteq S$, and therefore $e_i \in \sigma_{i+1}$ which implies that $(\sigma_{i+1}, \tau_{i+1}) \in \MM_0$. Repeating this argument shows that $(\sigma_j, \tau_j) \in \MM_0$ for every $j$, which is impossible since $\MM_0$ is an acyclic matching. Therefore, it must be the case that  $(\sigma_j, \tau_j) \in \MM_1$ for every $j$, contradicting the assumption that $\MM_1$ is an acyclic matching.
	\end{proof}

	\subsection*{Second reduction}
	By Claim \ref{step0} our problem has been reduced to finding an acyclic matching on the family $\mathsf{F_1}$. Note that for any $G\in \mathsf{F}_1$, the neighborhood $N_G(v) = V'\setminus W = N_H(v)$. Therefore every graph $G\in \mathsf{F}_1$ is uniquely determined by its induced subgraph $G[V']$.  
	Consequently, we can further reduce our problem to finding an acyclic matching on the family \[\mathsf{F}=\{G[V']  : G\in \mathsf{F}_1\},\] since this will uniquely determine an acyclic matching on $\mathsf{F}_1$. 
	
	\medskip
	
	The family $\mathsf{F}$ has a relatively simple characterization given in the claim below. Note that every graph $G\in\mathsf{F}$ contains the subgraph $H' = H[V']$ (which contains at least one edge by property (\ref{non-triv-H}) above). Also, recall that $D(G) = D_1\cup \cdots \cup D_r$ where $(D_1, \dots, D_r; A; C)$ is the Gallai--Edmonds decomposition of $G$.

	\begin{claim}\label{d(g)=m}
		The family $\mathsf{F}$ consists of all graphs $G$ on the vertex set $V'$ where $H'\subseteq G$, $\nu(G)=k-1$, and $D(G)=W$.
	\end{claim}
	
	\begin{proof}
		We already noted that $H'\subseteq G$ for every $G\in \mathsf{F}$. To prove the rest of the claim, first recall that $S = K_{W, \{v\}}$ and let $\hat{S} = K_{(V'\setminus W), \{v\}}$. 
		
		Consider a graph $G\in \mathsf{F}$ and let $\hat{G} = G\cup \hat{S}$. Note that $\hat{G} \in \mathsf{F}_1$ by definition. Hence, $\nu(G) \leq \nu(\hat{G})\leq k-1$, and for every $e\in S$, $\nu(\hat{G}+e) \geq k$. This implies that $\nu(\hat{G}+e)=k$ since adding a single edge increases the matching number by at most one. In particular, any maximum matching $M_e$ of $\hat{G}+e$ must contain the edge $e$ which is incident to $v$, and therefore $M_e \setminus \{e\}\subseteq G$. So we can conclude that $\nu(G)=k-1$. Moreover, if $e=uv$ for some $u\in W$, then the maximum matching $M_e\setminus \{e\}$ in $G$ misses the vertex $u$. This implies that $W\subseteq D(G)$.
		Now we show that $D(G)\subseteq W$. If not, there is some $u\in D(G)\cap (V'\setminus W)$ and a maximum matching $M_u$ in $G$ of size $k-1$ which does not cover the vertex $u$. But this would mean $M_u + uv$ is a matching in $\hat{G}$ of size $k$, which is a contradiction. Thus we have shown that every graph $G\in \mathsf{F}$ satisfies the conditions of the claim. 
		
		For the other direction, suppose $G$ is a graph satisfying the conditions of the claim. We will show that $\hat{G}=G\cup \hat{S} \in \mathsf{F}_1$. If $\nu(\hat{G})\geq k$, then there is a maximum matching in $\hat{G}$ which uses an edge from $\hat{S}$. Deleting this edge we find a maximum matching in $G$ which misses vertex $u\in (V'\setminus W)$. This is impossible by the condition $D(G) = W$. So we have $\nu(\hat{G})<k$. 
		
		For any edge $e = uv\in S$, let $M_j$ be a maximum matching which misses vertex $u\in W = D(G)$. Then, $M_j + e$ is a matching of size $k$ in $\hat{G} + e$. This shows that $S_{\hat{G}} = \emptyset$, and so by definition we have $\hat{G}\in \mathsf{F}_1$. 
	\end{proof}
	
	\subsection*{The join structure}
	Since all graphs in $\mathsf{F}$ have the same matching number, we can use Lemma \ref{combine-ge} to further reduce our problem to finding an acyclic matching on each non-empty subfamily $\mathsf{F}_{(D_1, \dots, D_r; A; C)} \subseteq \mathsf{F}$. We fix such a subfamily and let $D = D_1\cup \cdots \cup D_r$. The next step is to give a join structure on  $\mathsf{F}_{(D_1, \dots, D_r; A; C)}$.
	
	\medskip

	We define a projection map from the complete bipartite graph $K_{D, A}$. Define an index set \[I = K_{[r], A} = \{(t,a) : t\in [r], a\in A\}\] and partition the edges of $K_{D,A}$ as 
	\[K_{D,A} = \textstyle{\bigcup}_{i\in I} E_i, \]
	where $E_{(t, a)} = K_{D_t, \{a\}}$. Let $\pi: 2^{K_{D,A}} \to 2^{K_{[r], A}}$ be the corresponding  projection map,
	and define the family
	\[\mathsf{Proj}_{H'[D,A]} = \{G\subseteq K_{D,A} : \pi(G) \text{ is $A$-factor critical}, H'[D,A]\subseteq G \}. \]
	
	%Figure Claim 4.4
	\begin{figure}
		\centering
		
		\begin{tikzpicture}
		
		\begin{scope}
		
		\coordinate (d1) at (-2.8,2);
		\coordinate (d2) at (-1.4,2);
		\coordinate (d3) at (0,2);
		\coordinate (d4) at (1.4,2);
		\coordinate (d5) at (2.8,2);
		
		\coordinate (d11) at ($(d1)+(000+00:.3cm)$);
		\coordinate (d12) at ($(d1)+(120+00:.3cm)$);
		\coordinate (d13) at ($(d1)+(240+00:.3cm)$);
		
		\coordinate (d21) at ($(d2)+(000+40:.3cm)$);
		\coordinate (d22) at ($(d2)+(120+40:.3cm)$);
		\coordinate (d23) at ($(d2)+(240+40:.3cm)$);
		
		\coordinate (d31) at ($(d3)+(000+72:.3cm)$);
		\coordinate (d32) at ($(d3)+(120+72:.3cm)$);
		\coordinate (d33) at ($(d3)+(240+72:.3cm)$);
		
		\coordinate (d41) at ($(d5)+(000+41:.3cm)$);
		\coordinate (d42) at ($(d5)+(120+41:.3cm)$);
		\coordinate (d43) at ($(d5)+(240+41:.3cm)$);
		
		\coordinate (a1) at (-1.5,0.5);
		\coordinate (a2) at (-0.5,0.5);
		\coordinate (a3) at (0.5,0.5);
		\coordinate (a4) at (1.5,0.5);
		
		\coordinate (c1) at (-1.5,-.75);
		\coordinate (c2) at (-0.5,-.75);
		\coordinate (c3) at (0.5,-.75);
		\coordinate (c4) at (1.5,-.75);
		
		\coordinate (v1) at (5.6,1.8);
		\coordinate (v2) at (6.2,1.8);
		\coordinate (v3) at (6.8,1.8);
		\coordinate (v4) at (7.4,1.8);
		\coordinate (v5) at (8.0,1.8);
		
		\coordinate (w1) at (5.9,0.8);
		\coordinate (w2) at (6.5,0.8);
		\coordinate (w3) at (7.1,0.8);
		\coordinate (w4) at (7.7,0.8);
		
		\draw [orange, semithick] 
		(d11) -- (d12) -- (d13) --cycle
		(d21) -- (d22) -- (d23) --cycle
		(d31) -- (d32) -- (d33) --cycle
		(d41) -- (d42) -- (d43) --cycle;
		
		\draw [blue, semithick] 
		(a1) -- (d13)
		(a1) -- (d32)
		(a1) -- (d33)
		(a2) -- (d33)
		(a2) -- (d43)
		(a4) -- (d23)
		(a4) -- (d42)
		
		(w1) -- (v1)
		(w1) -- (v3)
		(w2) -- (v3)
		(w2) -- (v5)
		(w4) -- (v2)
		(w4) -- (v5);
		
		\draw [red, semithick]
		(a1) -- (a2) .. controls ($(a3) + (0,.2)$) .. (a4)
		(a1) .. controls ($(a2) + (0,-.3)$) and ($(a3) + (0,-.3)$) .. (a4)
		
		(a1) -- (c1) (a1) -- (c2) (a1) -- (c3) (a1) -- (c4)
		(a2) -- (c1) (a2) -- (c2) (a2) -- (c3) (a2) -- (c4)
		(a4) -- (c1) (a4) -- (c2) (a4) -- (c3) (a4) -- (c4);
		
		\draw[green!50!black, semithick]
		(c1) -- (c2) -- (c3) -- (c4) .. controls ($(c3)+(0,-.4)$) and ($(c2)+(0,-.4)$) .. (c1);
		
		\filldraw[black] 
		(d11) circle (1pt)
		(d12) circle (1pt)
		(d13) circle (1pt)
		
		(d21) circle (1pt)
		(d22) circle (1pt)
		(d23) circle (1pt)
		
		(d31) circle (1pt)
		(d32) circle (1pt)
		(d33) circle (1pt)
		
		(d41) circle (1pt)
		(d42) circle (1pt)
		(d43) circle (1pt)
		
		(a1) circle (1pt)
		(a2) circle (1pt)
		(a4) circle (1pt)
		
		(c1) circle (1pt)
		(c2) circle (1pt)
		(c3) circle (1pt)
		(c4) circle (1pt)
		
		(v1) circle (1pt)
		(v2) circle (1pt)
		(v3) circle (1pt)
		(v5) circle (1pt)
		
		(w1) circle (1pt)
		(w2) circle (1pt)
		(w4) circle (1pt);
		
		\node[above] at ($(d1)+(0,0.3)$) {\footnotesize $\FC_{H[D_1]}$};
		\node[above] at ($(d2)+(0,0.3)$) {\footnotesize $\FC_{H[D_2]}$};
		\node[above] at ($(d3)+(0,0.3)$) {\footnotesize $\FC_{H[D_3]}$};
		\node[above] at ($(d5)+(0,0.3)$) {\footnotesize $\FC_{H[D_r]}$}; 
		\node at (d4) {$\dots$};
		\node at (a3) {$\dots$};
		\node at (-3.6,2) {$D:$};
		\node at (-3.6,0.5) {$A:$};
		\node at (-3.6,-0.75) {$C:$};
		\fill [white, opacity =.75] (0,.9) --++ (-.75,0) --++ (0,.4) --++ (1.5,0) --++ (0,-.4) -- cycle;
		\node at (0.0,1.1) {\footnotesize $\mathsf{Proj}_{H'[D,A]}$};
		\node at (2.7,-0.1) {\footnotesize $\subseteq H'$};
		\node at (2.7,-0.75) {\footnotesize $\PM_{H'[C]}$};
		\node at (4.46,1.5) {\footnotesize $\pi$};
		\draw[-Latex, semithick] (4,1.3) -- (5,1.3);
		\node at (v4) {$\dots$};
		\node at (w3) {$\dots$};
		\node at ($(w2)+(.35,-.4)$) {\footnotesize $\BFC_{([r], A, \emptyset; \pi(H'[D,A]))}$};
		\end{scope}
		\end{tikzpicture}    
		\caption{The join structure in Claim \ref{join structure}.}
		\label{fig_claim4.4}
	\end{figure}
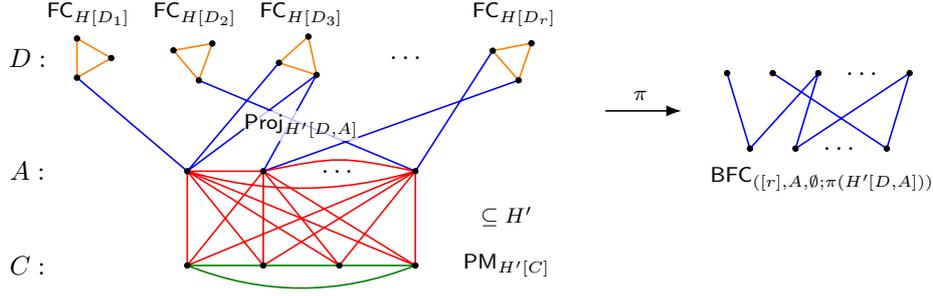
	
	\begin{claim} \label{join structure}
		If $H'[A] = K_A$ and $H'[A,C] = K_{A,C}$, then  \[\mathsf{F}_{(D_1, \dots, D_r;A;C)} = \mathsf{FC}_{H'[D_1]} * \cdots * \mathsf{FC}_{H'[D_r]} * \mathsf{Proj}_{H'[D,A]} * \mathsf{PM}_{H'[C]} * \{K_A\} * \{K_{A,C}\}.\]
	\end{claim}
	\begin{proof}
		(See Figure \ref{fig_claim4.4} for an illustration.) We first show  that $\mathsf{F}_{(D_1, \dots, D_r;A;C)}$ is included in the join.
		In other words, our goal is to show that any  graph $G\in \mathsf{F}_{(D_1 ,\dots,D_r ;A;C)}$ is the union of 
		\begin{itemize}
			\item a factor critical graph on each $D_i$ which contains $H'[D_i]$,
			\item a graph on $C$ with a perfect matching which  contains $H'[C]$,
			\item a bipartite graph in $\mathsf{Proj}_{H'[D,A]}$ connecting vertices in $D$ to vertices in $A$,
			\item a complete graph on $A$, and
			\item a complete bipartite graph connecting vertices in $A$ to vertices in $C$.
		\end{itemize}
		By definition, $G$ has Gallai--Edmonds decomposition $(D_1, \dots, D_r ; A; C)$ and satisfies the properties stated in Claim \ref{d(g)=m}. In particular, $G$ contains $H'$, and so by the assumptions that $H'[A]$ and $H'[A,C]$ are complete, the last two properties above are automatically satisfied. The first three properties are simple consequences of the Gallai--Edmonds decompostion (Theorem \ref{ge-decomposition}). Namely, 
		$G[D_i]$ is factor critical for every $i\in [r]$, $G[C]$ has a perfect matching, and for every $i\in [r]$ there is a matching $M_i$ in $G[D\setminus D_i, A]$ covering $A$ such that $|N_{M_i}(A) \cap D_j|\leq 1$ for every $j\in [r]$. The last conclusion is what guarantees that $G[D,A] \in \mathsf{Proj}_{H'[D,A]}$. This shows that $\mathsf{F}_{(D_1, \dots, D_r;A;C)}$ is contained in the join.
		
		\smallskip
		
		It remains to show the opposite inclusion. For a graph $G$ in the join, we need to show that $G$ satisfies the properties of Claim \ref{d(g)=m} and that $G$ has Gallai--Edmonds decomposition $(D_1, \dots, D_r ; A; C)$. 
		Since the decomposition is chosen so that $\mathsf{F}_{(D_1, \dots, D_r ; A; C)}$ is non-empty, it follows trivially that 
		$H'$ is contained in $G$. Moreover, by Theorem \ref{ge-decomposition}, we have  $W=D(G) = D_1 \cup \cdots \cup D_r$ and $2\nu(G) = |A|+|V'|-r=2(k-1)$. So it suffices to show  that $G$ has Gallai--Edmonds decomposition $(D_1, \dots, D_r ; A; C)$. 
		
		We first claim there is a matching in $G-v$ of size 
		\[k-1=\textstyle{\frac{1}{2}}(\textstyle{\sum_{j=1}^r}(|D_j|-1) + 2|A| + |C|).\]
		for any vertex $v\in D$. (This equality comes from Theorem \ref{ge-decomposition}.) 
		Let $D_i$ denote the component of $D$ which contains the vertex $v$. Since $G[D,A]\in \mathsf{Proj}_{H'[D,A]}$, we can find a matching $M_{D,A}$ in $G[(D\setminus D_i), A]$ which covers $A$ such that all the edges go to distinct components of $D$. Next, we can extend the matching $M_{D,A}$ further to obtain a matching which covers an additional $\textstyle{\sum_{j=1}^r}(|D_j|-1)+|C|$ vertices in $D\cup C$. This follows from the assumptions that $G[D_j] \in \mathsf{FC}_{H'[D_j]}$ for each $j\in [r]$ and $G[C] \in \mathsf{PM}_{H'[C]}$. 
		
		Now we show that $\nu(G) = k-1$, and that any matching of size $k-1$ covers every vertex in $A\cup C$. This will imply that $G\in \mathsf{F}_{(D_1, \dots, D_r; A; C)}$. Consider an arbitrary matching $M$ in $G$ and let $D'\subseteq D$, $A'\subseteq A$, and $C'\subseteq C$ be the sets of the vertices covered by $M$. Since each $|D_j|$ is odd and $N_{G}(D_j)\subseteq A$ we must have $|D'|\leq |D| - (r-|A|)$. Together with the trivial bounds $|A'|\leq |A|$ and $|C'| \leq |C|$ we get 
		\[2|M| = |D'|+|A'|+|C'| \leq |D|-(r-|A|) + |A| + |C| = 2(k-1). \]
		Therefore, we have $\nu(G)=k-1$. Also when $|M|=k-1$, we have $A'=A$ and $C'=C$. This completes the proof.
	\end{proof}
	
	\subsection*{The endgame}
	With the join structure of Claim \ref{join structure} we are left with finding an acyclic matching for each term of the join (Lemma \ref{product}). The $\mathsf{FC}$ terms and the $\mathsf{PM}$ term can be handled by Propositions \ref{size perfect m} and \ref{size-d0}. It remains to  deal with the term $\mathsf{Proj}_{H'[D,A]}$.
	
	\begin{claim} \label{first-projection}
		There is an acyclic matching on $\mathsf{Proj}_{H'[D,A]}$ such that any critical set $\sigma$ satisfies 
		\[|\sigma| \leq 2|A| + |H'[D,A]|.\]
	\end{claim}
	
	\begin{proof}
		We apply Lemma \ref{iden} with $\tau = H'[D,A]$ and $\mathsf{Q} = \mathsf{BFC}_{([r], A, \emptyset; \emptyset)}$. Note that $\mathsf{Q}$ is the family consisiting of all $A$-factor critical subgraphs of $K_{[r],A}$. By part {\em (\ref{proj lemma easy})} of Lemma \ref{iden} we have
		$\pi(\mathsf{Proj}_{H'[D,A]})
		= \mathsf{BFC}_{([r],A,\emptyset; \pi(H'[D,A]))}$.
		By Proposition \ref{size-yz} there is an acyclic matching on $\pi(\mathsf{Proj}_{H'[D,A]})$ where any critical set $\sigma$ has size at most $2|A| + |\pi(H'[D,A])|$. Applying part {\em (\ref{proj lemma main})} of Lemma \ref{iden} there is an acyclic matching on $\mathsf{Proj}_{H'[D,A]}$ where any critical set $\sigma$ has size at most $|\sigma| \leq 2|A| + |H'[D,A]|$.
	\end{proof}
	
	We are ready to finish the proof of Proposition \ref{special case}. Let $\mathsf{P} = \mathsf{F}_{(D_1, \dots, D_r; A; C)}$ be a non-empty subfamily of $\mathsf{F}$. 
	First we deal with the case when $H'[A] \neq K_A$ or $H'[A,C]\neq K_{A,C}$. In this case fix an edge $e\in K_A\cup K_{A,C}$ which is not an edge of $H$. By Lemma \ref{same-ge}, we have 
	$G-e, G+e \in \mathsf{P}$ for every graph $G\in \mathsf{P}$. This implies that $\mathsf{P}$ has a complete acyclic matching (Lemma \ref{boolean}).
	
	We may therefore assume that $H'[A] = K_A$ and $H'[A,C] = K_{A,C}$, and Proposition \ref{join structure} applies. By the Lemma \ref{product} it suffices to find an acyclic matching for each factor of the join, and sum up the sizes of critical sets in each factor. 
	
	By Propositions \ref{product} and \ref{size-d0}, the join  
	$\mathsf{FC}_{H'[D_1]} * \cdots * \mathsf{FC}_{H'[D_r]}$ has an acyclic  matching $\mathcal{M}_{\mathsf{FC}}$ where any critical set $\sigma$ satisfies
	\begin{align*}
	|\sigma| & \leq \textstyle{\sum_{i=1}^r} \left( \textstyle{\frac{3}{2}} (|D_i| - 1) + |H'[D_i]|\right) \\
	%&= \textstyle{\frac{3}{2}} \left( |D| - c\right) + |H'[D]|\\
	%&= \textstyle{\frac{3}{2}} \left( |D| - (|A|+n-2k+2)\right) + |H'[D]| \\  
	%&= \textstyle{\frac{3}{2}} \left( 2(k-1) - (n + |A| - |D|) \right) + |H'[D]| \\ 
	&= \textstyle{\frac{3}{2}} \left( 2(k-1) - 2|A| - |C|) \right) + |H'[D]|\\
	& = 3k-3-3|A|-\textstyle{\frac{3}{2}}|C| + |H'[D]|,
	\end{align*}
	with strict inequality whenever $H'[D]$ contains at least one edge.
	
	By Claim \ref{first-projection},  the  $\mathsf{Proj}_{H'[D,A]}$ term has an acyclic matching $\mathcal{M}_{\mathsf{Proj}}$ where any critical set $\sigma$ satisfies 
	\[|\sigma| \leq 2|A| + |H'[D,A]|.\]
	
	For the term $\mathsf{PM}_{H'[C]}$ we use Proposition \ref{size perfect m} to find an acyclic  matching $\mathcal{M}_{\mathsf{PM}}$ where any critical set $\sigma$ satisfies 
	\[|\sigma| \leq \textstyle{\frac{3}{2}} |C| + |H'[C]|,\]
	with strict inequality whenever $C$ is non-empty. 
	
	Finally, the terms $\{K_A\}$ and $\{K_{A,C}\}$ both have empty acyclic matchings with single critical sets of size $|K_A| = |H'[A]|$ and $|K_{A,C}| = |H'[A,C]|$, respectively.
	
	We now combine all these matchings using the Lemma \ref{product}. Noting that 
	\[|H'| = |H'[D]| + |H'[D,A]| + |H'[A]| + |H'[A,C]| + |H'[C]|,\] 
	we find that $\mathsf{P}$ has an acyclic matching $\mathcal{M} = \mathcal{M}_{\mathsf{FC}} \cup \mathcal{M}_{\mathsf{Proj}} \cup \mathcal{M}_{\mathsf{PM}}$ where any critical set $\sigma$ satisfies
	$|\sigma| \leq 3k - 3 - |A| + |H'|$ with strict inequality whenever $C$ is non-empty. Therefore, when $A\cup C$ is non-empty, we have $|\sigma| \leq 3k-4 +|H'|$.
	
	So, suppose $A\cup C$ is empty. By assumption $H'$ contains at least one edge (this was condition 
	(\ref{non-triv-H}) when we chose the vertex $v$). This implies that $H'[D]$ must contain at least one edge, in which case we must have strict inequality. Consequently, we have $|\sigma| \leq 3k-4 +|H'|$.
	
	The bound on $|\sigma|$ holds for any non-empty family $\mathsf{P} = \mathsf{F}_{(D_1, \dots, D_r; A; C)}$. By Lemma \ref{combine-ge} we have an acyclic matching on $\mathsf{F}$ where the same bound holds. Since $\mathsf{F}_1  = \mathsf{F} * \{K_{N_H(v), \{v\}}\}$ (recall that $N_G(v)=N_H(v)$ for any $G\in \mathsf{F}_1$) we get an acyclic matching on $\mathsf{F}_1$ where any critical set $\sigma$ satisifies $|\sigma|\leq 3k-4 +|H|$. \qed
	
	\subsection{Complete bipartite graphs} \label{complete bipartite}
	Fix a bipartite graph $H\subseteq K_{X,Y}$ with $1\leq \nu(H)<k$ and 
	define the family
	\[\mathsf{B}_H = \{ G\subseteq K_{X,Y} : \nu(G)< k, H\subseteq G\}. \]
	\begin{prop}\label{bipartite case}
		There is an acyclic matching on $\mathsf{B}_H$ such that any critical set $\sigma$ satisfies 
		\[|\sigma| \leq 2k-3+|H|.\]
	\end{prop}
	This result implies Theorem \ref{link-result} for the case when $G$ is a complete bipartite graph (by the same argument using Theorem \ref{morse_fundamental}, as we did for complete graphs in Section \ref{complete case}). 
	
	\begin{proof}[Proof of Proposition \ref{bipartite case}] We follow the same strategy as in the proof of Proposition 
		\ref{complete case}. It may  be assumed that $X$ and $Y$ are both non-empty and that $H\neq K_{X,Y}$.
		
		This first part is identical to the previous proof.  
		Start by choosing a vertex $v_0 \in X\cup Y$ of minimal degree in $H$. Note that $\deg_H(v_0)$ could equal zero and that $H$ contains at least one edge not incident to $v_0$. 
		
		Without loss of generality we assume that $v_0\in Y$, and we set $W = X\setminus N_H(v_0)$ and $S = K_{W, \{v_0\}}$. Note that our assumption $H\neq K_{X,Y}$ implies $W\neq \emptyset$. For a graph $G\in \mathsf{B}_H$ define the subset $S_G\subseteq S$ as 
		\[S_G = \{e\in S : G+e \in \mathsf{B}_H\}.\]
		Now define the subfamilies 
		\begin{align*}
		\mathsf{B}_0 & = \{G\in \mathsf{B}_H : S_G \neq \emptyset \}, \\
		\mathsf{B}_1 & = \{G\in \mathsf{B}_H : S_G = \emptyset \} .
		\end{align*}
		Thus we get a partition
		\[\mathsf{B}_H = \mathsf{B}_0 \cup \mathsf{B}_1.\]
		It turns out that the conclusion of Claim \ref{step0} holds in this situation as well. That is, $\mathsf{B}_0$ has a complete acyclic matching and for any acyclic matching on $\mathsf{B}_1$, their union is an acyclic matching on $\mathsf{B}_H$. The proof we gave earlier also works here, and is therefore omitted. (The key property needed is that the edges in $S$ are all incident to a common vertex.)
		
		Note that $N_G(v_0) = X\setminus W = N_H(v_0)$ for every graph $G\in \mathsf{B}_1$. Therefore, if we
		define the family \[\mathsf{B} = \{G - \{v_0\} : G \in \mathsf{B}_1\},\]
		then $\mathsf{B}_1 = \mathsf{B} * \{K_{(X\setminus W), \{v_0\}}\}$ and by Proposition \ref{product} our problem is reduced to finding an acyclic matching on $\mathsf{B}$. Set $Y' = Y \setminus \{v_0\}$ and $H' = H[X,Y']$, and  observe that $\mathsf{B}$ is the family of all bipartite graphs $G\subseteq G_{X, Y'}$ which satisfy:
		\begin{itemize}
			\item $H' \subseteq G$,
			\item $\nu(G) = \nu(G\cup K_{(X\setminus W), \{v_0\}}) =  k - 1$, and
			\item $\nu(G+e) = k$ for every $e\in S = K_{W, \{v_0\}}$.
		\end{itemize}
		By Lemma \ref{combine-ge} our problem is reduced to finding an acyclic matching on each non-empty subfamily $\mathsf{B}_{(D; A; C)} \subseteq \mathsf{B}$. Note that we use a simpler notation $(D;A;C)$, because the components of $D$ in the Gallai--Edmonds decomposition of a bipartite graph consists of singletons and are therefore uniquely determined by the set $D$. This follows from that a bipartite graph with at least two vertices cannot be factor critical.
		
		Let us fix a non-empty subfamily $\mathsf{B}_{(D; A; C)}$ and introduce the notation
		\[\def\arraystretch{1.35}
		\begin{array}{rclcrcl}
		D_X & = & D\cap X   &, &    D_Y & = & D\cap Y' \\
		A_X & = & A\cap X   &, &    A_Y & = & A\cap Y' \\
		C_X & = & C\cap X   &, &    C_Y & = & C\cap Y' .\\
		\end{array}
		\]
		It follows from the defining properties of $\mathsf{B}$ that $D_X = W$. Note also that $|C_X| = |C_Y|$,  $N_H(v_0) = A_X\cup C_X$, and $\frac{1}{2}|C|+|A| = k-1$. Moreover, for any $G\in \mathsf{B}_{(D; A; C)}$ we have $N_G(D_X) = A_Y$ and $N_G(D_Y) = A_X$.
		
		\medskip
		
		First consider the case $N_H(v_0) = \emptyset$. In this case $D_X=X$, $A = A_Y$, $C = \emptyset$,  and $D_Y$ is just a set of isolated vertices. It follows that  
		\[\mathsf{B}_{(D;A;C)} = \mathsf{BFC}_{(X, A, \emptyset; H')}.\]
		Since $H'$ is non-empty and $|A| = k-1$, it follows from Proposition \ref{size-yz} that $\mathsf{B}_{(D; A; C)}$ has an acyclic matching where any critical set $\sigma$ satisifies $|\sigma| \leq 2(k-1) - 1 + |H'|$. This gives us the desired bound. 
		
		\medskip
		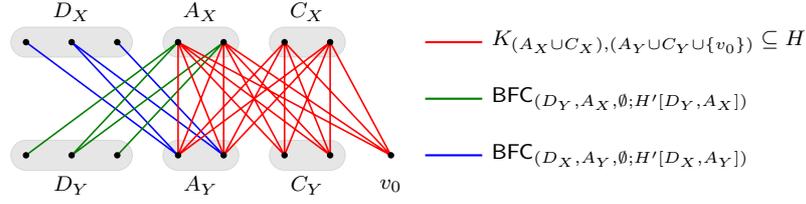
\begin{figure} 
			\centering
			\begin{tikzpicture}
			\coordinate (dx1) at (0.0,1.5);
			\coordinate (dx2) at (0.6,1.5);
			\coordinate (dx3) at (1.2,1.5);
			\coordinate (ax1) at (2.0,1.5);
			\coordinate (ax2) at (2.6,1.5);
			\coordinate (cx1) at (3.4,1.5);
			\coordinate (cx2) at (4.0,1.5);
			
			\coordinate (dy1) at (0.0,0);
			\coordinate (dy2) at (0.6,0);
			\coordinate (dy3) at (1.2,0);
			\coordinate (ay1) at (2.0,0);
			\coordinate (ay2) at (2.6,0);
			\coordinate (cy1) at (3.4,0);
			\coordinate (cy2) at (4.0,0);
			
			\coordinate (v0)  at (4.8,0);
			
			\coordinate (R) at (6,1.5);
			\coordinate (B) at (6,0.0);
			\coordinate (G) at (6,0.75);

			\filldraw [black, thin, opacity=0.1] ($(dx1)+(0,0.2)$) -- ($(dx3)+(0,0.2)$) arc[start angle = 90, end angle = -90, radius = 0.2] -- ($(dx1)+(0,-0.2)$) arc[start angle = 270, end angle = 90, radius = 0.2]
			
			($(ax1)+(0,0.2)$) -- ($(ax2)+(0,0.2)$) arc[start angle = 90, end angle = -90, radius = 0.2] -- ($(ax1)+(0,-0.2)$) arc[start angle = 270, end angle = 90, radius = 0.2]
			
			($(cx1)+(0,0.2)$) -- ($(cx2)+(0,0.2)$) arc[start angle = 90, end angle = -90, radius = 0.2] -- ($(cx1)+(0,-0.2)$) arc[start angle = 270, end angle = 90, radius = 0.2]
			
			($(dy1)+(0,0.2)$) -- ($(dy3)+(0,0.2)$) arc[start angle = 90, end angle = -90, radius = 0.2] -- ($(dy1)+(0,-0.2)$) arc[start angle = 270, end angle = 90, radius = 0.2]
			
			($(ay1)+(0,0.2)$) -- ($(ay2)+(0,0.2)$) arc[start angle = 90, end angle = -90, radius = 0.2] -- ($(ay1)+(0,-0.2)$) arc[start angle = 270, end angle = 90, radius = 0.2]
			
			($(cy1)+(0,0.2)$) -- ($(cy2)+(0,0.2)$) arc[start angle = 90, end angle = -90, radius = 0.2] -- ($(cy1)+(0,-0.2)$) arc[start angle = 270, end angle = 90, radius = 0.2];
			
			\draw[red, thick] (R) -- ($(R)+(-.75,0)$);
			\draw[blue, thick] (B) -- ($(B)+(-.75,0)$);
			\draw[green!50!black, thick] (G) -- ($(G)+(-.75,0)$);
			
			\node[right] at (R) {\footnotesize $K_{(A_X\cup C_X), (A_Y\cup C_Y\cup \{v_0\})} \subseteq H$};
			
			\node[right] at (G) {\footnotesize $\BFC_{(D_Y, A_X, \emptyset; H'[D_Y,A_X])}$};
			
			\node[right] at (B) {\footnotesize $\BFC_{(D_X, A_Y, \emptyset; H'[D_X,A_Y])}$};
			
			\draw [blue, semithick] 
			(dx1) -- (ay1)
			(dx2) -- (ay1)
			(dx2) -- (ay2)
			(dx3) -- (ay2); 
			
			\draw [green!50!black, semithick] 
			(dy1) -- (ax1)
			(dy2) -- (ax1)
			(dy2) -- (ax2)
			(dy3) -- (ax2); 
			
			\draw [red, semithick]
			(ax1) -- (ay1)
			(ax1) -- (ay2)
			(ax1) -- (cy1)
			(ax1) -- (cy2)
			(ax1) -- (v0)
			(ax2) -- (ay1)
			(ax2) -- (ay2)
			(ax2) -- (cy1)
			(ax2) -- (cy2)
			(ax2) -- (v0)
			(cx1) -- (ay1)
			(cx1) -- (ay2)
			(cx1) -- (cy1)
			(cx1) -- (cy2)
			(cx1) -- (v0)
			(cx2) -- (ay1)
			(cx2) -- (ay2)
			(cx2) -- (cy1)
			(cx2) -- (cy2)
			(cx2) -- (v0);
			
			\filldraw[black]
			(dx1) circle (1pt)
			(dx2) circle (1pt)
			(dx3) circle (1pt)
			(ax1) circle (1pt)
			(ax2) circle (1pt)
			(cx1) circle (1pt)
			(cx2) circle (1pt)
			
			(dy1) circle (1pt)
			(dy2) circle (1pt)
			(dy3) circle (1pt)
			(ay1) circle (1pt)
			(ay2) circle (1pt)
			(cy1) circle (1pt)
			(cy2) circle (1pt)
			(v0) circle (1pt);
			
			\node at ($(dx2)+(0,0.4)$) {\footnotesize $D_X$};
			\node at ($(ax1)+(0.3,0.4)$) {\footnotesize $A_X$};
			\node at ($(cx1)+(0.3,0.4)$) {\footnotesize $C_X$};
			\node at ($(dy2)+(0,-0.4)$) {\footnotesize $D_Y$};
			\node at ($(ay1)+(0.3,-0.4)$) {\footnotesize $A_Y$};
			\node at ($(cy1)+(0.3,-0.4)$) {\footnotesize $C_Y$};
			\node at ($(v0)+(0,-0.4)$) {\footnotesize $v_0$};
			\end{tikzpicture}
			\caption{The join structure in the proof of Proposition \ref{bipartite case} when $N_H(v_0)\ne \emptyset$ and $K_{(A_X \cup C_X), (A_Y\cup C_Y)}\subseteq H'$.}
			\label{fig_section4.2}
		\end{figure}

		Now suppose $N_H(v_0) \neq \emptyset$. Since $N_H(v_0) = A_X\cup C_X$, the minimality assumption on $\deg_H(v_0)$ therefore implies that $A_X\cup C_X = N_H(u)$ for every vertex $u\in C_Y$. In particular, $K_{(A_X\cup C_X), C_Y}\subseteq H'$. Next, if there exists an edge $e\in K_{(A_X\cup C_X), A_Y}$ which is not an edge in $H'$, then $G-e, G+e \in \mathsf{B}_{(D;A;C)}$ for every graph $G\in \mathsf{B}_{(D;A;C)}$ by Lemma \ref{same-ge}. In this case, $\mathsf{B}_{(D;A;C)}$ has a complete acyclic matching by Lemma \ref{boolean}.
		
		We may therefore assume that $K_{(A_X\cup C_X), (A_Y\cup C_Y)} \subseteq H'$. This gives us the join structure 
		\[\mathsf{B}_{(D;A;C)} = \mathsf{BFC}_{(D_X, A_Y, \emptyset; H'[D_X, A_Y])} * \mathsf{BFC}_{(D_Y, A_X, \emptyset; H'[D_Y, A_X])} * \{ K_{(A_X\cup C_X), (A_Y\cup C_Y)}\}. \]
		(See Figure \ref{fig_section4.2}.) Applying Proposition \ref{size-yz} to each of the $\mathsf{BFC}$ terms, we find an acyclic matching on $\mathsf{B}_{(D;A:C)}$ where any critical set $\sigma$ satisifies 
		\[|\sigma| \leq 2|A_X| + 2|A_Y| + |H'|= 2|A| + |H'|.\]
		Moreover this inequlity is strict whenever $H'[D_X,A_Y]$ or $H'[D_Y, A_Y]$ contains at least one edge. By assumption, $D_X\neq \emptyset$, and since $N_H(v_0)\neq \emptyset$ and $D_X\cap N_H(v_0)=\emptyset$
		we must have $\deg_{H'}(u)>0$ for every $u\in D_X$. Consequently we have strict inequality above, and since $|A|\leq k-1$ we have an acyclic matching on $\mathsf{B}_{(D;A;C)}$ where any critical set $\sigma$ satisfies $|\sigma|\leq 2k-3 + |H'|$. 
	\end{proof}

	\subsection{General case} \label{homology}
	Here we deduce the general case of Theorem \ref{link-result} from the special cases shown in the previous subsections. The arguments here deal with general simplicial complexes (so in particular they hold for graph complexes). Theorem \ref{leray-result} will also be proved here.
	
	\smallskip
	
	Recall from subsection \ref{subs:mains} that a simplicial complex $\mathsf{K}$ is $d$-Leray if $\tilde{H}_i(\mathsf{L}) = 0$ for all $i\geq d$ and for every induced subcomplex $\mathsf{L}\subseteq \mathsf{K}$. An immediate consequence of this definition is that the $d$-Leray property is {\em hereditary}, meaning that any induced subcomplex of a $d$-Leray complex is also $d$-Leray. This fact is less obvious if we consdider the equivalent definition of the $d$-Leray property which states that $\mathsf{K}$ is $d$-Leray if $\tilde{H}_i(\mathsf{lk_K}(\sigma))=0$ for every $i\geq d$ and every $\sigma\in \mathsf{K}$. 
	
	Here we are concerned with a property that is slightly weaker than the $d$-Leray property. 
	Let $\mathsf{K}$ be a simplicial complex with the property:
	
	\medskip
	
	\begin{centerline}{
			($\ast$) {\em For every  non-empty face $\sigma\in K$, the link $\lk_\mathsf{K}(\sigma)$ 
				%= \{\tau \subseteq V : \tau\cap \sigma= \emptyset, \tau\cup \sigma\in K\}$ 
				has vanishing homology in all dimensions $d\geq d_0$.}}
	\end{centerline}
	
	\medskip

	\noindent The following proposition tells us that property ($\ast$) is hereditary.
	
	\begin{prop} \label{hereditary}
		Let $\mathsf{K}$ be a simplicial complex on the vertex set $V$ which satisfies property {\em ($\ast$)}.
		Then, for any non-empty subset $S\subseteq V$, the induced subcomplex $\mathsf{K}[S]$ also satisfies property {\em ($\ast$)}.
	\end{prop}
	
	\begin{proof}
		It is enough to show that $\mathsf{K}-v$ also satisfies ($\ast$) for an arbitrary vertex $v\in V$. That is, we show that for every non-empty face $\sigma \in \mathsf{K}-v$, we have $\tilde{H}_d(\lk_{\mathsf{K}-v}(\sigma))=0$ for all $d\geq d_0$. Since $\lk_{\mathsf{K}-v}(\sigma)=\lk_\mathsf{K}(\sigma)-v$, we consider reduced homology groups of $\lk_\mathsf{K}(\sigma)-v$. Note that we only need to consider the case when $v \in \lk_\mathsf{K}(\sigma)$, since $\lk_\mathsf{K}(\sigma)-v=\lk_\mathsf{K}(\sigma)$ otherwise.
		
		Let $\mathsf{X}=\lk_\mathsf{K}(\sigma)$, and define the {\em star} of $v$ in $\mathsf{X}$ as
		\[\st_\mathsf{X}(v)=\{\tau\in \mathsf{X}: \tau\cup \{v\}\in \mathsf{X} \}.\]
		Applying the Mayer--Vietoris sequence to the pair $\mathsf{X}-v$, $\st_\mathsf{X}(v)$ and using the fact that $\st_\mathsf{X}(v)$ is contractible, implies exactness of the sequence
		\begin{align*}\label{mayer-vietoris}
		\cdots \rightarrow \tilde{H}_d(\lk_\mathsf{X}(v)) \rightarrow  \tilde{H}_d(\mathsf{X} - v) \rightarrow \tilde{H}_d(\mathsf{X})\rightarrow \cdots.
		\end{align*}
		Since $\mathsf{K}$ satisfies ($\ast$), the last term $\tilde{H}_d(\mathsf{X})=\tilde{H}_d(\lk_\mathsf{K}(\sigma))$ vanishes for all $d\geq d_0$. Using ($\ast$) again, the identity 
		\[ \lk_\mathsf{X}(v)=\lk_{\lk_\mathsf{K}(\sigma)}(v)=\lk_\mathsf{K}(\sigma \cup \{v\})\]
		implies that the first term also vanishes for all $d\geq d_0$. Therefore $\tilde{H}_d(\mathsf{X}-v)$ also vanishes.
	\end{proof}

	%$d_0\geq 0$ and 
	\begin{cor} \label{near d ler}
		Let $\mathsf{K}$ be a simplicial complex on the vertex set $V$ which satisfies property {\em ($\ast$)}. Then $\mathsf{K}$ has vanishing homology in all dimension $d\geq d_0+1$.
	\end{cor}
	\begin{proof}
		For contradiction, suppose that $\mathsf{K}$ does not satisfy the conclusion. Let $W\subseteq V$ be an inclusion minimal subset such that the induced subcomplex $\mathsf{L} = \mathsf{K}[W]$ satisfies:
		\begin{itemize}
			\item There is a $d \geq d_0 + 1$ such that $\tilde{H}_d (\mathsf{L}) \ne 0$.
			\item We have $\tilde{H}_d(\mathsf{M}) = 0$ for every proper induced subcomplex $\mathsf{M} \subseteq \mathsf{L}$ and $d \geq d_0 + 1$.
		\end{itemize}
		By Proposition \ref{hereditary}, $\mathsf{L}$ satisfies property ($\ast$). Note that $|W|\geq 2$, otherwise $\tilde{H}_k(\mathsf{L})=0$ for every integer $k$.
		
		Fix a vertex $v$ of $\mathsf{L}$ and apply the Mayer--Vietoris sequence to the pair $\mathsf{L} - v$,  $\st_\mathsf{L}(v)$. This implies exactness of the sequence
		\[\cdots \to \tilde{H}_d(\lk_\mathsf{L}(v)) \to
		\tilde{H}_d(\mathsf{L}-v) \to \tilde{H}_d(\mathsf{L}) \to
		\tilde{H}_{d-1}(\lk_\mathsf{L}(v)) \to \cdots \]
		For all $d\geq d_0+1$, property ($\ast$) implies that the first and last terms are zero which implies that the two middle terms are isomorphic. The second term is zero by the minimality assumption, and so  $\tilde{H}_d(\mathsf{L}) =0$.
	\end{proof}

	Now we can prove Theorem \ref{link-result} in full generality. Consider an arbitrary graph $G\subseteq K_V$. The non-matching complex $\mathsf{NM}_k(G)$ is an induced subcomplex of $\mathsf{NM}_k(K_V)$, and in Section \ref{complete case} we showed that $\mathsf{NM}_k(K_V)$ satisfies property ($\ast$) with $d_0 = 3k-4$. By Proposition \ref{hereditary} it follows that  $\mathsf{NM}_k(G)$ also satisfies ($\ast$). If $G$ is bipartite, then $\mathsf{NM}_k(G)$ is an induced subcomplex of $\mathsf{NM}_k(K_{X,Y})$, and therefore satisfies ($\ast$) with $d_0= 2k-3$.
	This proves Theorem \ref{link-result}, and Theorem \ref{leray-result} now follows from Corollary \ref{near d ler}. \qed

	\section{Proof of Proposition \ref{size perfect m}}  \label{section_size perfect m}
	Fix a graph $H$ on the vertex set $V$.
	Our goal is to find an acyclic matching $\mathcal{M}$ on $\mathsf{PM}_H$ such that any critical set $\sigma$ satisfies 
	\[|\sigma|\leq \textstyle{\frac{3}{2}}|V|-2+\max\{|H|-1,0\},\]
	whenever $|V|$ is an even positive integer. With the obvious inequality for the case when $|V|=0$, we have the desired inequality in Proposition \ref{size perfect m}. We assume that $H\neq K_V$, otherwise it is obvious.
	
	\subsection*{First reduction}
	Fix an edge $e_0 = vw \in K_V\setminus H$ with the additional condition that $\deg_H(w)>0$ if $|H|>0$. This is possible since $H\neq K_V$.
	Define the subfamily $\mathsf{F}_0\subseteq \mathsf{PM}_H$ as
	\[\mathsf{F}_0 = \{G\subseteq K_V : G-e_0, G+e_0 \in \mathsf{PM}_H\},\]
	and set $\mathsf{F}_1 = \mathsf{PM}_H \setminus \mathsf{F}_0$. Note that for any graph $G\in \mathsf{PM}_H$, $G+e_0$ has a perfect matching, that is, $G+e_0 \in \mathsf{PM}_H$.	Note also that  $\mathsf{F}_1$ consists of those graphs in $\mathsf{PM}_H$ for which {\em every} perfect matching contains the edge $e_0$. 
	By Lemma \ref{boolean}
	our problem is reduced to finding a suitable acyclic matching on $\mathsf{F}_1$. Define the family \[\mathsf{F} = \{G-e_0 : G\in \mathsf{F}_1\},\] and note that $\mathsf{F}_1 = \mathsf{F} * \{e_0\}$. This reduces our problem to finding an acyclic matching on $\mathsf{F}$ (by Lemma \ref{product}). Since $\nu(G) = \frac{|V|}{2}-1$ for every $G\in \mathsf{F}$, the problem is further reduced to finding an acyclic matching for each non-empty subfamily $\mathsf{F}_{(D_1, \dots, D_r; A; C)}\subseteq \mathsf{F}$ (by Lemma \ref{combine-ge}). Note also that by Theorem \ref{ge-decomposition} we have $|A| = r-2$.

	\subsection*{Join structure}
	Our next step is to give a join structure on the family $\mathsf{F}_{(D_1, \dots, D_r; A; C)}$. 
	We first observe that the vertices $v$ and $w$  belong to distinct components of $D$ (recall $e_0 = vw$). To see this, consider any graph $G\in \mathsf{F}_{(D_1, \dots, D_r; A; C)}$. A perfect matching $M$ in $G+e_0$ must contain the edge $e_0$. Therefore $M\setminus \{e_0\}$ is a maximum matching in $G$ that avoids vertices $v$ and $w$, which must lie in distinct components of $D$ (by Theorem \ref{ge-decomposition}). 
	
	Relabel the components of $D$ (if necessary) such that $v\in D_{r-1}$ and $w\in D_r$. Note also that if $|H|>0$, then the assumption $\deg_H(w)>0$ implies that $H[D_r]$ or $H[D,A]$ is non-empty. 
	
	Consider the complete bipartite graph $K_{D,A}$. Define the index set 
	\[I = \{(t, a) : 1\leq t\leq r-1, a\in A\}\]
	and partition the edges of $K_{D,A}$ as 
	\[K_{D,A} = \textstyle{\bigcup}_{(t,a)\in I} E_{(t,a)} ,\]
	where 
	$E_{(t,a)} = 
	\begin{cases}
	K_{D_t, \{a\}}& \text{ when } t<r-1,\\
	K_{(D_{r-1}\cup D_r), \{a\}} & \text{ when } t=r-1.
	\end{cases}$
	
	\vspace{.7ex}
	
	Let $\pi : 2^{K_{D,A}} \to 2^{K_{[r-1], A}}$ be the corresponding projection map, and 
	%set $Z = [c-2]$.
	%Partition the vertices of $(D\cup A)$ as
	%\[D_1 \cup \cdots \cup D_{c-2} \cup (D_{c-1}\cup D_r) 
	%\cup \{a_1\} \cup \cdots \cup \{a_{c-2}\}\]
	%and let $\pi : 2^{K_{D\cup A}} \to 2^{K_{[c-1] \cup A}}$ denote the resulting projection map. 
	%and set $Z = [c-2]$. By the assumption $v\in D_{c-1}$ and $w\in D_r$, and our %previous observation, we have that for any $G \in \mathsf{F}_{(D_1, \dots, D_r; A; C)}$, the projection %$\pi(G[D,A])[Z,A]$ has a perfect matching. 
	%and define the family
	%\vspace{.5ex}
	define %$\mathsf{Proj}_{H[D,A]}$ to be 
	the family 
	%of bipartite graphs $G\subseteq K_{D,A}$ which satisfy:
	%\begin{itemize}
	%	\item $H[D,A]\subseteq G$, 
	%    \item $\pi(G)$ as $A$-factor critical, and
	%    \item $\pi(G)[Z,A]$ has a perfect matching.
	%\end{itemize}
	\[\mathsf{Proj}_{H[D,A]} = \{G\subseteq K_{D,A} : \pi(G) 
	\text{ is $A$-factor critical}, H[D,A]\subseteq G 
	\}.\]
	
	%Figure Claim 5.1
	
	\begin{figure}
		\centering
		
		\begin{tikzpicture}
		
		\begin{scope}
		
		\coordinate (d1) at (-2.8,2);
		\coordinate (d2) at (-1.4,2);
		\coordinate (d3) at (0,2);
		\coordinate (d4) at (1.8,2);
		\coordinate (d5) at (3.2,2);
		
		\coordinate (d11) at ($(d1)+(000+00:.3cm)$);
		\coordinate (d12) at ($(d1)+(120+00:.3cm)$);
		\coordinate (d13) at ($(d1)+(240+00:.3cm)$);
		
		\coordinate (d21) at ($(d2)+(000+40:.3cm)$);
		\coordinate (d22) at ($(d2)+(120+40:.3cm)$);
		\coordinate (d23) at ($(d2)+(240+40:.3cm)$);
		
		\coordinate (d31) at ($(d3)+(000+72:.3cm)$);
		\coordinate (d32) at ($(d3)+(120+72:.3cm)$);
		\coordinate (d33) at ($(d3)+(240+72:.3cm)$);
		
		\coordinate (r41) at ($(d4)+(000+31:.3cm)$);
		\coordinate (r42) at ($(d4)+(120+31:.3cm)$);
		\coordinate (r43) at ($(d4)+(240+31:.3cm)$);
		
		\coordinate (d41) at ($(d5)+(000+41:.3cm)$);
		\coordinate (d42) at ($(d5)+(120+41:.3cm)$);
		\coordinate (d43) at ($(d5)+(240+41:.3cm)$);
		
		\coordinate (a1) at (-1.5,0.5);
		\coordinate (a2) at (-0.5,0.5);
		\coordinate (a3) at (0.5,0.5);
		\coordinate (a4) at (1.5,0.5);
		
		\coordinate (c1) at (-1.5,-.75);
		\coordinate (c2) at (-0.5,-.75);
		\coordinate (c3) at (0.5,-.75);
		\coordinate (c4) at (1.5,-.75);
		
		\coordinate (v1) at (5.6,1.8);
		\coordinate (v2) at (6.2,1.8);
		\coordinate (v3) at (6.8,1.8);
		\coordinate (v4) at (7.4,1.8);
		\coordinate (v5) at (8.0,1.8);
		
		\coordinate (w1) at (5.9,0.8);
		\coordinate (w2) at (6.5,0.8);
		\coordinate (w3) at (7.1,0.8);
		\coordinate (w4) at (7.7,0.8);
		
		\fill [violet, opacity = .2]
		($(d4)+(-.5,.1)$) arc [start angle = 180, end angle = 90, radius = .26] --++ (1.8,0) arc [start angle = 90, end angle = 0, radius = .26] --++ (0,-.3) arc [start angle = 0, end angle = -90, radius = .26] --++ (-1.8,0) arc [start angle = -90, end angle = -180, radius = .26]
		;
		
		\fill [violet, opacity = .3]
		(v5) circle (3pt);
		
		\draw [orange, semithick] 
		(d11) -- (d12) -- (d13) --cycle
		(d21) -- (d22) -- (d23) --cycle
		(d31) -- (d32) -- (d33) --cycle
		(r41) -- (r42) -- (r43) --cycle
		(d41) -- (d42) -- (d43) --cycle;
		
		\draw [blue, semithick] 
		(a1) -- (d13)
		(a1) -- (d32)
		(a1) -- (d33)
		(a2) -- (d33)
		(a2) -- (r43)
		(a4) -- (d23)
		(a4) -- (d42)
		
		(w1) -- (v1)
		(w1) -- (v3)
		(w2) -- (v3)
		(w2) -- (v5)
		(w4) -- (v2)
		(w4) -- (v5);
		
		\draw [red, semithick]
		(a1) -- (a2) .. controls ($(a3) + (0,.2)$) .. (a4)
		(a1) .. controls ($(a2) + (0,-.3)$) and ($(a3) + (0,-.3)$) .. (a4)
		
		(a1) -- (c1) (a1) -- (c2) (a1) -- (c3) (a1) -- (c4)
		(a2) -- (c1) (a2) -- (c2) (a2) -- (c3) (a2) -- (c4)
		(a4) -- (c1) (a4) -- (c2) (a4) -- (c3) (a4) -- (c4);
		
		\draw[green!50!black, semithick]
		(c1) -- (c2) -- (c3) -- (c4) .. controls ($(c3)+(0,-.4)$) and ($(c2)+(0,-.4)$) .. (c1);
		
		\filldraw[black] 
		(d11) circle (1pt)
		(d12) circle (1pt)
		(d13) circle (1pt)
		
		(d21) circle (1pt)
		(d22) circle (1pt)
		(d23) circle (1pt)
		
		(d31) circle (1pt)
		(d32) circle (1pt)
		(d33) circle (1pt)
		
		(r41) circle (1pt)
		(r42) circle (1pt)
		(r43) circle (1pt)
		
		(d41) circle (1pt)
		(d42) circle (1pt)
		(d43) circle (1pt)
		
		(a1) circle (1pt)
		(a2) circle (1pt)
		(a4) circle (1pt)
		
		(c1) circle (1pt)
		(c2) circle (1pt)
		(c3) circle (1pt)
		(c4) circle (1pt)
		
		(v1) circle (1pt)
		(v2) circle (1pt)
		(v3) circle (1pt)
		(v5) circle (1pt)
		
		(w1) circle (1pt)
		(w2) circle (1pt)
		(w4) circle (1pt);
		
		\node[above] at ($(d1)+(0,0.3)$) {\footnotesize $\FC_{H[D_1]}$};
		\node[above] at ($(d2)+(0,0.3)$) {\footnotesize $\FC_{H[D_2]}$};
		\node[above] at ($(d3)+(0,0.3)$) {\footnotesize $\FC_{H[D_3]}$};
		\node[above] at ($(d4)+(0,0.3)$) {\footnotesize $\FC_{H[D_{r-1}]}$}; 
		\node[above] at ($(d5)+(0,0.3)$) {\footnotesize $\FC_{H[D_r]}$}; 
		\node at ($(d4) + (-.9,0)$) {$\dots$};
		\node at (a3) {$\dots$};
		\node at (-3.6,2) {$D:$};
		\node at (-3.6,0.5) {$A:$};
		\node at (-3.6,-0.75) {$C:$};
		\fill [white, opacity =.75] (0,.9) --++ (-.75,0) --++ (0,.4) --++ (1.5,0) --++ (0,-.4) -- cycle;
		\node at (0.0,1.1) {\footnotesize $\mathsf{Proj}_{H[D,A]}$};
		\node at (2.7,-0.1) {\footnotesize $\subseteq H$};
		\node at (2.7,-0.75) {\footnotesize $\PM_{H[C]}$};
		\node at (4.46,1.5) {\footnotesize $\pi$};
		\draw[-Latex, semithick] (4,1.3) -- (5,1.3);
		\node at (v4) {$\dots$};
		\node at (w3) {$\dots$};
		\node at ($(w2)+(.35,-.4)$) {\footnotesize $\BFC_{([r-1], A, \emptyset; \pi([D,A]))}$};
		\end{scope}
		\end{tikzpicture}    
		\caption{The join structure in Claim \ref{perfect matching join}. Here, all vertices in $D_{r-1}\cup D_r$ are identified as a single vertex at the projection.}
		\label{fig:claim5.1}
	\end{figure}
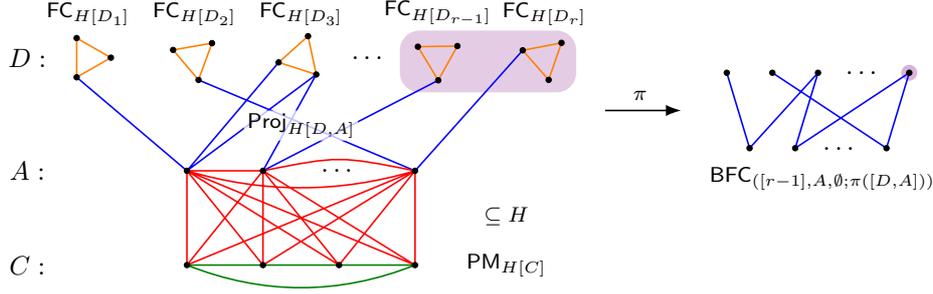
	
	\begin{claim}\label{perfect matching join}
		If $H[A] = K_A$ and $H[A,C] = K_{A,C}$, then 
		\[\mathsf{F}_{(D_1, \dots, D_r; A; C)} = \mathsf{FC}_{H[D_1]} * \cdots * \mathsf{FC}_{H[D_r]} * \mathsf{Proj}_{H[D,A]} * \mathsf{PM}_{H[C]} * \{K_A\} * \{K_{A,C}\}\]
	\end{claim}
	\begin{proof}
		(See Figure \ref{fig:claim5.1} for an illustration.) Consider a graph $G\in \mathsf{F}_{(D_1, \dots, D_r; A; C)}$. 
		To show that $G$ belongs to the join it suffices to show that $G' = G[D,A]\in \mathsf{Proj}_{H[D,A]}$ (the other terms are obvious). 
		Clearly $H[D,A]\subseteq G'$, 
		so we need to show that $\pi(G')$ is $A$-factor critical. We prove this by showing that $|N_{\pi(G')}(A')|>|A'|$ for each non-empty subset $A'\subseteq A$. Set $Z = [r-2]$, and
		note that our previous discussion which showed that there is a maximum matching in $G$ which avoids the vertices $v$ and $w$ implies that $\pi(G')[Z,A]$ has a perfect matching. In particular $|N_{\pi(G')}(A')\cap Z| \geq |A'|$ for every subset $A'\subseteq A$. When the inequality is strict, we are done. So suppose that $|N_{\pi(G')}(A')\cap Z| = |A'|$.
		
		Define an auxiliary projection map $\pi' : 2^{K_{D,A}} \to 2^{K_{[r], A}}$ corresponding to the partition
		\[K_{D,A} = \textstyle{\bigcup} K_{D_s, \{a\}}, \text{ where } s\in [r] \text{ and }a\in A.\]
		Note that $(N_{\pi(G')}(A') \cap Z) = (N_{\pi'(G')}(A') \cap Z)$. 
		It follows from Theorem \ref{ge-decomposition} that $\pi'(G')$ is $A$-factor critical, which by Hall's marriage theorem implies that $|N_{\pi'(G')}(A')|>|A'|$. With $|N_{\pi(G')}(A')\cap Z| = |A'|$, this implies that $(N_{\pi'(G')}(A')\cap \{r, r-1\})$ is non-empty. This again implies that $r-1 \in N_{\pi(G')}(A')$. Therefore $G'$ is in $\mathsf{Proj}_{H[D,A]}$.

		Now we show the opposite inclusion.
		Let $G$ be a graph in the join. It is obvious that $H\subseteq G$, and so our goal is to show that $G$ has Gallai--Edmonds decomposition $(D_1, \dots, D_r; A; C)$ and that $G+e_0$ has a perfect matching. 
		
		Consider a matching $M$ in $G$ and let $D'\subseteq D$, $A'\subseteq A$, and $C'\subseteq C$ denote the vertices covered by $M$. Since $|D_i|$ is odd and $N_G(D_i) \subseteq A$ for all $i$, we have $|D'|\leq |D| - (r - |A|) = |D|-2$. Therefore 
		\[|M| = \textstyle{\frac{1}{2}}(|D'|+|A'|+|C'|) \leq \textstyle{\frac{1}{2}}(|D|-2+|A|+|C|) = \textstyle{\frac{1}{2}}(|V|-2),\]
		and it is easily seen that $\nu(G)= \frac{1}{2}(|V|-2)$. Moreover, since any maximum matching in $G$ covers all but two vertices, these uncovered vertices must belong to distinct $D_i$. The defining conditions of $\mathsf{Proj}_{H[D,A]}$ imply that for any vertex $v\in D_1\cup \cdots \cup D_r$ there is a maximum matching which does not cover $v$. Therefore $G$ has the desired Gallai--Edmonds decomposition. 
		Finally, the fact that $G+e_0$ has a perfect matching follows from the condition that $\pi(G[D,A])[Z,A]$ has a perfect matching. 
	\end{proof}

	\subsection*{The endgame}
	In order to apply Lemma \ref{product} we need to find suitable acyclic matchings for each of the term of the join in Claim \ref{perfect matching join}. The acyclic matchings on the factors $\mathsf{FC}_{H[D_i]}$ are given by Proposition \ref{size-d0}, and for the factor $\mathsf{PM}_{H[C]}$ we can apply induction on $|C|$. It remains to deal with the term $\mathsf{Proj}_{H[D,A]}$.
	
	\begin{claim}\label{second project}
		There is an acyclic matching on $\mathsf{Proj}_{H[D,A]}$ such that any critical set $\sigma$ satisfies
		\[|\sigma| \leq 2|A| + \max\{|H[D,A]|-1,0\}.\]
	\end{claim}
	
	\begin{proof}
		We apply Lemma \ref{iden} with $\tau = H[D,A]$ and $\mathsf{Q} = \mathsf{BFC}_{([c-1], A, \emptyset; \emptyset)}$. Note that $\mathsf{Q}$ consists of all $A$-factor critical subgraphs of $K_{[c-1], A}$. By part {\em (\ref{proj lemma easy})} of Lemma \ref{iden} it follows that $\pi(\mathsf{Proj}_{H[D,A]}) = \mathsf{BFC}_{([c-1], A, \emptyset; \pi(H[D,A]))}$. Therefore $\pi(\mathsf{Proj}_{H[D,A]})$ has an acyclic matching where any critical set has size at most $2|A| + \max\{|\pi(H[D,A])|-1,0\}$ (by Proposition \ref{size-yz}). Applying part {\em (\ref{proj lemma main})} of Lemma \ref{iden} there is an acyclic matching on $\mathsf{Proj}_{H[D,A]}$ where any critical set $\sigma$ has size at most $ 2|A| + \max\{|H[D,A]|-1, 0\}$.
	\end{proof}
	
	We are ready to finish the proof of Proposition \ref{size perfect m}. We will apply induction on $|V|$. It is easy to check that the bound holds when $|V|=2$, so we assume $|V|\geq 4$ is even and that the bound holds for vertex sets of even size strictly less than $|V|$.
	
	Let $\mathsf{P} = \mathsf{F}_{(D_1, \dots, D_r; A; C)}$ be a non-empty subfamily of $\mathsf{F}$. First consider the case when there exists an edge $e\in K_A \cup K_{A,C}$ which is not an edge in $H$. For any graph $G\in \mathsf{P}$ it follows from Lemma \ref{same-ge} that $G-e$ and $G+e$ both contain $H$ and have the same Gallai--Edmonds decompostion $(D_1, \dots, D_r; A; C)$. Moreover, any perfect matching in $G+e_0$ does not contain the edge $e$, and therefore $G-e, G+e\in \mathsf{P}$. It follows that $\mathsf{P}$ has a complete acyclic matching (Lemma \ref{boolean}). 
	
	We may therefore assume that $H[A] =K_A$ and $H[A,C] = K_{A,C}$, and so Claim \ref{perfect matching join} can be applied. By Lemma \ref{product} it suffices to find an acyclic matching for each term of the join and sum up the sizes of the critical sets of each term. 
	
	By Lemma \ref{product} and Proposition \ref{size-d0}, the join $\mathsf{FC_{H[D_1]}} * \cdots * \mathsf{FC_{H[D_r]}}$ has an acyclic matching where any critical set $\sigma$ satisfies
	\[\def\arraystretch{1.65}
	\begin{array}{rcl}
	|\sigma| & \leq & \textstyle{\sum}_{i=1}^r\left( \textstyle{\frac{3}{2}}(|D_i|-1) + |H[D_i]| \right) \\
	%      & = & \textstyle{\frac{3}{2}}(|D|-c) + |H[D]| \\ 
	& = &   \textstyle{\frac{3}{2}}(|D|-|A|-2) + |H[D]|.
	\end{array}
	\]
	%
	%\begin{align*}
	%    |\sigma| &\leq  \textstyle{\sum}_{i=1}^c\left( \textstyle{\frac{3}{2}}(|D_i|-1) + |H[D_i]| \right) \\
	%            &=  \textstyle{\frac{3}{2}}(|D|-c) + |H[D]|,
	%\end{align*}
	Here we used that $r = |A|+2$ in the last equality.
	
	By Claim \ref{second project} the term $\mathsf{Proj}_{H[D,A]}$ has an acyclic matching where any critical set $\sigma$ satisfies 
	\[\def\arraystretch{1.65}
	\begin{array}{rcl}
	|\sigma| & \leq & 2|A| + |H[D,A]|.
	\end{array}
	\]
	%with strict inequality whenever $H[D,A]$ contains at least one edge. Recall that our choice of the edge $e_0 = vw$ required that $\deg_H(w)>0$ if $|H|>0$. Moreover, we have showed that $w\in D$, so therefore we may assume that one of the inequalities above is strict whenever $H$ is non-empty. 
	
	For the term $\mathsf{PM}_{H[C]}$ we can apply induction since $|C|<|V|$. Therefore this term has an acyclic matching where any critical set $\sigma$ satisfies 
	\[\def\arraystretch{1.65}
	\begin{array}{rcl}
	|\sigma| & \leq & \textstyle{\frac{3}{2}}|C| + |H[C]|.
	\end{array}
	\]
	
	Finally, the terms $\{K_A\}$ and $\{K_{A,C}\}$ both have empty acyclic matchings with single critical sets of size $|K_A| = |H[A]|$ and $|K_{A,C}| = |H[A,C]|$, respectively.

	We now sum up these bounds and apply Lemma \ref{product}.  Thus we find an acyclic matching on $\mathsf{P}$ where any critical set $\sigma$ satisfies
	$|\sigma|\leq \textstyle{\frac{3}{2}}|V|-3 +|H|$. 	Moreover, if $H$ is non-empty, then our choice of $e_0=vw$ implies that $H[D]$ or $H[D,A]$ is also non-empty. Therefore the above inequality is strict whenever $H$ is non-empty.  This bound holds for any non-empty family $\mathsf{P} = \mathsf{F}_{(D_1, \dots, D_r; A; C)} \subseteq \mathsf{F}$, so by Lemma \ref{combine-ge} there is an acyclic matching on $\mathsf{F}$ where every critical set satisfy the same bound. Finally, since $\mathsf{F}_1 = \mathsf{F} * \{e_0\}$ we get an acyclic matching in $\mathsf{F}_1$ where any critical set $\sigma$ satisfies $|\sigma|\leq \frac{3}{2}|V|-2+|H|$ with strict inequality whenever $H$ is non-empty.
	\qed

	\section{Proof of Proposition \ref{size-d0}} \label{section_size-d0}
	
	Fix a graph $H$ on the vertex set $V$.
	Our goal is to find an acyclic matching $\mathcal{M}$ on $\mathsf{FC}_H$ such that any critical set $\sigma$ satisfies 
	\[|\sigma|\leq \textstyle{\frac{3}{2}}(|V|-1)+\max\{ |H|-1 , 0\}.\]
	We may assume that  $|V|$ is odd and $H\neq K_V$.

	\subsection*{Reduction step}
	The first part of the proof is similar to the proof of Proposition \ref{size perfect m}. Fix an edge $e_0 = vw \in K_V\setminus H$ with the additional condition that $\deg_H(w)>0$ if $|H|>0$. 
	
	Define the subfamily $\mathsf{F}_0\subseteq \mathsf{FC}_H$ as 
	\[\mathsf{F}_0 = \{ G\subseteq K_V : G-e_0, G+e_0\in \mathsf{FC}_H\},\]
	and set $\mathsf{F}_1 = \mathsf{FC} \setminus \mathsf{F}_0$. 
	Note that $\mathsf{F}_1$ consists of the graphs $G\in \mathsf{FC}_H$ such that $e_0\in G$ and  $G-e_0$ is not factor critical.
	Just as in the proof of Proposition \ref{size perfect m}, by applying Lemma \ref{boolean}, our problem is reduced to finding a suitable acyclic matching on $\mathsf{F}_1$. Define the family \[\mathsf{F} = \{G-e_0 : G\in \mathsf{F}_1\},\] and note that $\mathsf{F}_1 = \mathsf{F} * \{e_0\}$. 
	
	Since $G+e_0$ is factor critical for any  $G\in \mathsf{F}$ it follows that $\nu(G) = \frac{|V|-1}{2}$ (both graphs $(G+e_0)-\{v\}$ and $(G+e_0)-\{w\}$ have perfect matchings).
	Therefore Lemma \ref{same-ge} implies that we can further reduce the problem to finding an acyclic matching on each non-empty subfamily $\mathsf{F}_{(D_1, \dots, D_r; A; C)} \subseteq \mathsf{F}$. Note that by Theorem \ref{ge-decomposition} we have $|A| = r-1$.
	%Moreover, we must have $c>1$. If not, then $|A|=0$ and $|C|>0$, otherwise $V= D_1$ and $G$ is factor critical. This is impossible because then $G+e$ is disconnected and can not be factor critical. Thus $c>1$ and $A$ is non-empty.

	%and we set $A = \{a_1, \dots, a_{c-1}\}$. 
	%Since $(G+e)-\{v\}$ and $(G+e)-\{w\}$ have perfect matchings it follows that $v$ and $w$ belong to $D$.  
	%Observe also that $|A|>0$
	
	%\medskip
	
	\subsection*{Join structure} Our next goal is to give a join structure on the family $\mathsf{F}_{(D_1, \dots, D_r; A; C)}$. We first show that $v$ and $w$ belong to distinct components of $D$. To see why, consider any  $G\in \mathsf{F}_{(D_1, \dots, D_r; A; C)}$. The assumption that $G+e_0$ is factor critical implies that $v$ and $w$ both belong to $D$. Since $G$ is not factor critical, we have $A\cup C$ is non-empty. For any $a\in A\cup C$, $(G+e_0)-\{a\}$ has a perfect matching $M$ on $V\setminus \{a\}$. This is impossible when $v$ and $w$ are in the same components of $D$, since each of $r=|A|+1$ components of $D$ should have at least one vertex to be matched via $M$ with a vertex in $A\setminus \{a\}$. Thus $r>1$ and $A$ is non-empty.

	We set $A = \{a_1, \dots, a_{r-1}\}$ and after relabeling the components of $D$ (if necessary) we may assume that $v\in D_{r-1}$ and $w\in D_r$. Note that if $|H|>0$, then the assumption $\deg_H(w)>0$ implies that $H[D_r]$ or $H[D,A]$ is non-empty. 
	
	\medskip
	
	Now we define a projection map from the complete bipartite graph $K_{D,A}$. Define the index set 
	$I = \{(j,a) : j\in [r], a\in A\}$ and set $E_{(j,a)} = K_{D_j, \{a\}}$.
	This gives us a partition,
	\[K_{D,A} = \textstyle{\bigcup}_{e \in K_{[r],A}} \: E_e,\]
	with a corresponding projection map $\pi : 2^{K_{D,A}} \to 2^{K_{[r],A}}$. Set $Z = [r-2]$ and define the family
	\[\mathsf{Proj}_{H[D,A]} = \{G\subseteq K_{D,A} : \pi(G) \text{ is $(A,Z)$-factor critical, } H[D,A]\subseteq G\}.\]
	\medskip
	
	Next we define a projection map from $K_{A,C} \cup K_C$.
	For an edge $e \in K_{(\{\alpha\} \cup C)}$ define
	%$E_e = K_{A, \{v\}}$ for 
	\[E_e = 
	\begin{cases}
	K_{A, \{v\}} & \text{ for } e  = \alpha v\in K_{\{ \alpha \},C} \\
	\{e\} & \text{ for } e \in K_{C}. 
	\end{cases}\]
	This gives us a partition, 
	%\[
	$K_{A,C} \cup K_C = \textstyle{\bigcup}_{e\in K_{(\{ \alpha \}\cup C)}} \: E_e$,
	%\]
	%For $v\in C$ set $E_v = K_{A, \{v\}}$ and for $e\in K_C$ set $E_e = \{e\}$.
	%This gives us a partition 
	%\[K_{A, C} \cup K_C =  \textstyle{\bigcup}_{x\in C \cup K_C } \: E_x \]
	with a corresponding projection map $\pi' : 2^{K_{A,C} \cup K_C} \to 2^{K_{(\{ \alpha \} \cup C)}}$.
	%
	%
	%Set $C = \{c_1, \dots, c_{2j}\}$ and partition the vertices of $(A\cup C)$ as
	%\[A \cup \{c_1\} \cup \cdots \cup \{c_{2j}\}.\]
	%Let $\pi':2^{K_{A\cup C}} \to 2^{K_{\{\alpha\}\cup C}}$ be the resulting projection map. In other words, the map $\pi'$ takes a graph $G\subseteq K_{A\cup C}$ and identifies all vertices in $A$ as a single vertex $\alpha$.
	%
	%Define $\mathsf{Proj}_{H[A,C]}$ to be the family of all graphs $G\subseteq K_{A \cup C}$ which satisfy:
	%\begin{itemize}
	%    \item $H[A,C]\cup H[C]\subseteq G$, and 
	%    \item $\pi'(G)$ is factor critical.
	%\end{itemize}
	Define the family
	\[\mathsf{Proj}_{H[A,C]\cup H[C]} = 
	\{
	G\subseteq K_{A,C}\cup K_C : \pi'(G) \text{ is factor critical, } H[A,C]\cup H[C]\subseteq G\}.
	\]
	
	%Figure Claim 6.1
	\begin{figure}
		\centering
		
		\begin{tikzpicture}
		
		\begin{scope}
		
		\coordinate (d1) at (-2.8,2);
		\coordinate (d2) at (-1.4,2);
		\coordinate (d3) at (0,2);
		\coordinate (d4) at (1.4,2);
		\coordinate (d5) at (2.8,2);
		
		\coordinate (d11) at ($(d1)+(000+00:.3cm)$);
		\coordinate (d12) at ($(d1)+(120+00:.3cm)$);
		\coordinate (d13) at ($(d1)+(240+00:.3cm)$);
		
		\coordinate (d21) at ($(d2)+(000+40:.3cm)$);
		\coordinate (d22) at ($(d2)+(120+40:.3cm)$);
		\coordinate (d23) at ($(d2)+(240+40:.3cm)$);
		
		\coordinate (d31) at ($(d3)+(000+72:.3cm)$);
		\coordinate (d32) at ($(d3)+(120+72:.3cm)$);
		\coordinate (d33) at ($(d3)+(240+72:.3cm)$);
		
		\coordinate (d41) at ($(d5)+(000+41:.3cm)$);
		\coordinate (d42) at ($(d5)+(120+41:.3cm)$);
		\coordinate (d43) at ($(d5)+(240+41:.3cm)$);
		
		\coordinate (a1) at (-1.5,0.5);
		\coordinate (a2) at (-0.5,0.5);
		\coordinate (a3) at (0.5,0.5);
		\coordinate (a4) at (1.5,0.5);
		
		\coordinate (c1) at (-1.5,-.75);
		\coordinate (c2) at (-0.5,-.75);
		\coordinate (c3) at (0.5,-.75);
		\coordinate (c4) at (1.5,-.75);
		
		\coordinate (v1) at (5.6,2.2);
		\coordinate (v2) at (6.2,2.2);
		\coordinate (v3) at (6.8,2.2);
		\coordinate (v4) at (7.4,2.2);
		\coordinate (v5) at (8.0,2.2);
		
		\coordinate (w1) at (5.9,1.2);
		\coordinate (w2) at (6.5,1.2);
		\coordinate (w3) at (7.1,1.2);
		\coordinate (w4) at (7.7,1.2);
		
		\coordinate (y) at (6.8,0);
		
		\coordinate (x1) at (5.9,-1);
		\coordinate (x2) at (6.5,-1);
		\coordinate (x3) at (7.1,-1);
		\coordinate (x4) at (7.7,-1);
		
		\draw [orange, semithick] 
		(d11) -- (d12) -- (d13) --cycle
		(d21) -- (d22) -- (d23) --cycle
		(d31) -- (d32) -- (d33) --cycle
		(d41) -- (d42) -- (d43) --cycle;
		
		\draw [blue, semithick] 
		(a1) -- (d13)
		(a1) -- (d32)
		(a1) -- (d33)
		(a2) -- (d33)
		(a2) -- (d43)
		(a4) -- (d23)
		(a4) -- (d42)
		
		(w1) -- (v1)
		(w1) -- (v3)
		(w2) -- (v3)
		(w2) -- (v5)
		(w4) -- (v2)
		(w4) -- (v5);
		
		\draw [red, semithick]
		(a1) -- (a2) .. controls ($(a3) + (0,.2)$) .. (a4)
		(a1) .. controls ($(a2) + (0,-.3)$) and ($(a3) + (0,-.3)$) .. (a4);
		
		\draw[green!50!black, semithick]
		(c1) -- (c2) -- (c3) -- (c4) .. controls ($(c3)+(0,-.4)$) and ($(c2)+(0,-.4)$) .. (c1)
		(a1) -- (c1)
		(a1) -- (c4)
		(a2) -- (c2)
		(a4) -- (c3)
		(a4) -- (c4)
		(y) -- (x1)
		(y) -- (x2)
		(y) -- (x3)
		(y) -- (x4)
		(x1) -- (x2) -- (x3) -- (x4) .. controls ($(x3)+(0,-.4)$) and ($(x2)+(0,-.4)$) .. (x1);
		
		\filldraw [red, opacity=.2] (y) circle (3pt);
		
		\filldraw[black] 
		(d11) circle (1pt)
		(d12) circle (1pt)
		(d13) circle (1pt)
		
		(d21) circle (1pt)
		(d22) circle (1pt)
		(d23) circle (1pt)
		
		(d31) circle (1pt)
		(d32) circle (1pt)
		(d33) circle (1pt)
		
		(d41) circle (1pt)
		(d42) circle (1pt)
		(d43) circle (1pt)
		
		(a1) circle (1pt)
		(a2) circle (1pt)
		(a4) circle (1pt)
		
		(c1) circle (1pt)
		(c2) circle (1pt)
		(c3) circle (1pt)
		(c4) circle (1pt)
		
		(v1) circle (1pt)
		(v2) circle (1pt)
		(v3) circle (1pt)
		(v5) circle (1pt)
		
		(w1) circle (1pt)
		(w2) circle (1pt)
		(w4) circle (1pt)
		
		(y) circle (1pt)
		(x1) circle (1pt)
		(x2) circle (1pt)
		(x3) circle (1pt)
		(x4) circle (1pt);
		
		\node[above] at ($(d1)+(0,0.3)$) {\footnotesize $\FC_{H[D_1]}$};
		\node[above] at ($(d2)+(0,0.3)$) {\footnotesize $\FC_{H[D_2]}$};
		\node[above] at ($(d3)+(0,0.3)$) {\footnotesize $\FC_{H[D_3]}$};
		\node[above] at ($(d5)+(0,0.3)$) {\footnotesize $\FC_{H[D_r]}$}; 
		\node at (d4) {$\dots$};
		\node at (a3) {$\dots$};
		\node at (-3.6,2) {$D:$};
		\node at (-3.6,0.5) {$A:$};
		\node at (-3.6,-0.75) {$C:$};
		\fill [white, opacity =.75] (0,.9) --++ (-.75,0) --++ (0,.4) --++ (1.5,0) --++ (0,-.4) -- cycle;
		\node at (0.0,1.1) {\footnotesize $\mathsf{Proj}_{H[D,A]}$};
		\node at (2.7,.5) {\footnotesize $K_A \subseteq H$};
		\node at (2.8,-0.45) {\footnotesize $\mathsf{Proj}_{H[A,C]\cup H[C]}$};
		\node at (4.4,1.6) {\footnotesize $\pi$};
		\draw[-Latex, semithick] (4,1.3) -- (5,1.5);
		\node at (v4) {$\dots$};
		\node at (w3) {$\dots$};
		\node at ($(w2)+(.35,-.4)$) {\footnotesize $\BFC_{([r], A, \emptyset; \pi(H[D,A]))}$};
		\node at (4.46,0) {\footnotesize $\pi'$};
		\draw[-Latex, semithick] (4,-.2) -- (5,-.4);
		\node at ($(x2) + (.3,-.6)$) {\footnotesize $\mathsf{FC}_{\pi'(H[A,C]\cup H[C])}$};
		\node [right] at ($(y)+(.1,0)$) {\footnotesize $\alpha$};
		\end{scope}
		\end{tikzpicture}    
		\caption{The join structure in Claim \ref{fc join}. Here, all vertices in $A$ are identified as the vertex $\alpha$ in the projection. }
		\label{fig:claim6.1}
	\end{figure}
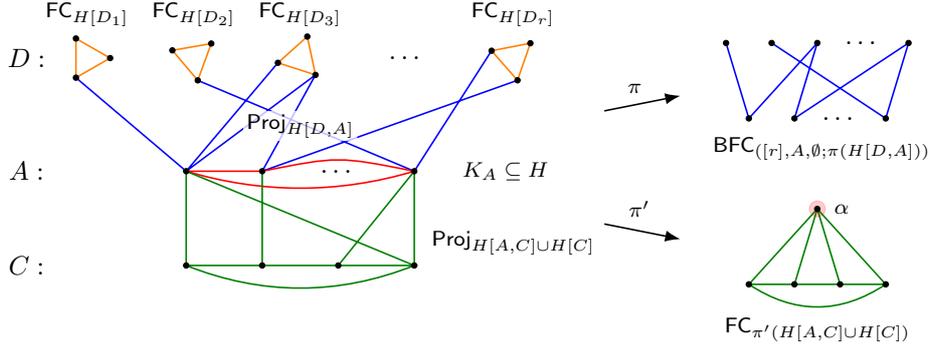
	
	\begin{claim}\label{fc join}
		If $H[A] = K_A$, then
		\[\mathsf{F}_{(D_1, \dots, D_r; A; C)} = 
		\mathsf{FC}_{H[D_1]} * \cdots * \mathsf{FC}_{H[D_r]} * 
		\mathsf{Proj}_{H[D,A]} * 
		\mathsf{Proj}_{H[A,C]\cup H[C]} * 
		\{ K_A \}.\]
	\end{claim}\textbf{}
	\begin{proof}
		(See Figure \ref{fig:claim6.1} for an illustration.) First consider a graph $G\in \mathsf{F}_{(D_1, \dots, D_r; A; C)}$. We start by showing that $\pi(G[D,A])[Z,A]$ is $Z$-factor critical. Fix an arbitrary vertex $a\in A$. Note that $(G+e_0)-\{a\}$ contains a perfect matching $M$, and each component $D_i$ for $i\in [r-2]$ should have a unique vertex to be matched with a vertex in $A\setminus \{a\}$ via $M$. So we obtain a matching covering $Z$ in $\pi(G[D,A])[Z,A\setminus \{a\}]$ for any $a\in A$.

		It remains to show that $G_{A,C} = \pi'(G[A,C] \cup G[C])$ is factor critical. (The other terms of the join are obvious.) Let $u$ be a vertex in $C$. By assumption the graph $(G+e_0)-\{u\}$ has a perfect matching using the edge $e_0$. In this perfect matching, exactly $r-2$ vertices of $A$ are matched to vertices in $D$, while the remaining vertex of $A$ together with $C\setminus\{u\}$ admits a perfect matching. This shows that there is a perfect matching in $G_{A,C} - \{u\}$, and therefore $G_{A,C}$ is factor critical.
		
		Now we show the opposite inclusion. Let $G$ be a graph in the join. It is obvious that $H\subseteq G$. We first show that $G$ has Gallai--Edmonds decomposition $(D_1, \dots, D_r; A; C)$. Consider a matching $M$ in $G$ and let $D'\subseteq D$, $A'\subseteq A$, and $C'\subseteq C$ denote the vertices covered by $M$. Since $|D_i|$ is odd and $N_G(D_i) \subseteq A$ for all $i$, we have $|D'|\leq |D|- (c-|A|) = |D|-1$. Therefore
		\[|M| \leq
		\textstyle{\frac{1}{2}}(|D'| + |A'|+ |C'|) \leq
		\textstyle{\frac{1}{2}}(|D|-1 + |A|+ |C|) = \textstyle{\frac{1}{2}}(|V|-1),\]
		and it is easily seen that $\nu(G) = \frac{1}{2}(|V|-1)$. Moreover, equality is attained only when $A' =A$ and $C' = C$. From this it follows easily that $G$ has desired Gallai--Edmonds decomposition. We can also conclude that $G$ is not factor critical; for any vertex $a\in A$ the graph $G-\{a\}$ has no perfect matching. It remains to show that $G+e_0$ is factor critical, and for this it is sufficient to show that $(G+e_0)-\{u\}$ has a perfect matching for any vertex $u\in A\cup C$. 
		
		If $u\in A$, then $G[D_1\cup \cdots \cup D_{r-2} \cup (A\setminus\{u\})]$ has a perfect matching by the condition that  $\pi(G[D,A])[Z,A]$ is $Z$-factor critical. Together with perfect matchings on $G[C]$ and $G[D_{r-1}\cup D_r]+e_0$, this gives a perfect matching in $(G+e_0)-\{u\}$.
		
		If $u\in C$, then the condition on $\mathsf{Proj}_{H[A,C]\cup H[C]}$ implies that there is a vertex $u'\in A$ such that $G[\{u'\} \cup (C\setminus \{u\})]$ has a perfect matching. By the same argument above it follows that $G[D_1 \cup \dots \cup D_{r-2} \cup (A\setminus\{u'\})]$ has a perfect matching. Together with perfect matching on $G[D_{r-1}\cup D_r]+e_0$, we get a perfect matching in $(G+e_0)-\{u\}$.
	\end{proof}
	
	\subsection*{The endgame}
	In order to apply Lemma \ref{product} we need to find suitable acyclic matchings for each of the terms of the join in Claim \ref{fc join}. The acyclic matchings on the factors $\mathsf{FC}_{H[D_i]}$ can be dealt with by induction on $|D_i|$.
	
	\begin{claim}\label{projection 1 DA}
		There is an acyclic matching on $\mathsf{Proj}_{H[D,A]}$ such that any critical set $\sigma$ satisfies
		\[|\sigma| \leq 3|A|-1 + |H[D,A]|.\]
		Moreover the inequality is strict whenever $H[D,A]$ contains at least one edge. 
	\end{claim}
	
	\begin{proof}
		We apply Lemma \ref{iden} with $\tau = H[D,A]$ and $\mathsf{Q} = \mathsf{BFC}_{([r],A,Z; \emptyset)}$. By part {\em (\ref{proj lemma easy})} of Lemma \ref{iden} we have $\pi(\mathsf{Proj}_{H[D,A]}) = \mathsf{BFC}_{([r],A,Z; \pi(H[D,A]))}$. Proposition \ref{size-yz} implies that there is an acyclic matching on $\pi(\mathsf{Proj}_{H[D,A]})$ where any critical set $\sigma$ satisfies 
		\[\begin{array}{rcl}
		|\sigma| & \leq & 2|A| + |Z| + |\pi(H[D,A])|  \\
		& = & 3|A|-1 + |\pi(H[D,A])|.
		\end{array}\]
		Here we used that $|Z| = |A|-1$. Also, it follows from Proposition \ref{size-yz} that the inequality is strict whenever $\pi(H[D,A])$ contains at least one edge. By part {\em (\ref{proj lemma main})} of Lemma \ref{iden} we get an acyclic matching on $\mathsf{Proj}_{H[D,A]}$ where any critical set $\sigma$ satisfies $|\sigma|\leq 3|A|-1 + |H[D,A]|$. Moreover, if $H[D,A]$ contains at least one edge, then so does $\pi(H[D,A])$ which implies strict inequality. 
	\end{proof}
	
	\begin{claim}\label{projection 2 AC}
		There is an acyclic matching on $\mathsf{Proj}_{H[A,C]\cup H[C]}$ where any critical set $\sigma$ satisfies 
		\[|\sigma| \leq \textstyle{\frac{3}{2}}|C| + |H[A,C]| + |H[C]|.\]
	\end{claim}
	
	\begin{proof}
		We apply Lemma \ref{iden} with $\tau = H[A,C]\cup H[C]$ and 
		\[\mathsf{Q} = \{G\subseteq K_{(\{\alpha\}\cup C)} : G \text{ is factor critical}\}.\]
		By part {\em (\ref{proj lemma easy})} of Lemma \ref{iden} we have
		$\pi'(\mathsf{Proj}_{H[A,C]\cup H[C]}) = \mathsf{FC}_{\pi'(\tau)}$. Since $|C|+1 <|V|$ we can apply induction and find an acyclic matching on $\pi'(\mathsf{Proj}_{H[A,C]\cup H[C]})$ where any critical set $\sigma$ satisfies 
		$|\sigma| \leq \textstyle{\frac{3}{2}}|C| + |\pi'(\tau)|$.
		By part {\em (\ref{proj lemma main})} of Lemma \ref{iden} there is an acyclic matching on $\mathsf{Proj}_{H[A,C]\cup H[C]}$ where any critical set $\sigma$ satisfies $|\sigma|\leq \frac{3}{2}|C| + |H[A,C]| + |H[C]|$.
	\end{proof}
	
	We are now ready to finish the proof of Proposition \ref{size-d0}. Let $\mathsf{P} = \mathsf{F}_{(D_1,\dots, D_r; A; C)}$ be a non-empty subfamily of $\mathsf{F}$. Suppose there is an edge $e\in K_A$ which is not an edge in $H$. For any graph $G\in \mathsf{P}$, Lemma \ref{same-ge} implies that $G-e$ and $G+e$ have Gallai--Edmonds decomposition $(D_1, \dots, D_r; A; C)$. Deleting any vertex from $A$ shows that neither $G-e$ nor $G+e$ are factor critical. For any vertex $u$ the graph $(G+e_0)- \{u\}$ has a perfect matching which does not use any edge in $K_A$. This shows that $G-e, G+e \in \mathsf{P}$, and implies that $\mathsf{P}$ has a complete acyclic matching (Lemma \ref{boolean}). 
	
	We may therefore assume that $H[A] =K_A$, and so Claim \ref{fc join} applies. Since $|A|>0$ we have $|D|<|V|$, so by induction there is an acyclic matching on $\mathsf{F}_{H[D_1]}*\cdots *\mathsf{F}_{H[D_r]}$ where any critical set $\sigma$ satisfies
	\[\def\arraystretch{1.65}
	\begin{array}{rcl}
	|\sigma| & \leq & \textstyle{\sum}_{i=1}^r 
	\left( \textstyle{\frac{3}{2}}(|D_i|-1) + |H[D_i]| \right) \\
	& = & \textstyle{\frac{3}{2}}(|D|-|A|-1) + |H[D]|.
	\end{array}\]
	Here we used $|A| = r-1$. Also, we may assume that the inequality is strict whenever $H[D]$ contains at least one edge. 
	
	Finally, the term $\{K_{A}\}$ has an empty acyclic matching with a  single critical set of size $|K_{A}| = |H[A]|$. Summing up these bounds together with the bounds from Claims \ref{projection 1 DA} and \ref{projection 2 AC}, we find that there is an acylic matching on $\mathsf{P}$ where any critical set $\sigma$ satisfies 
	\[\def\arraystretch{1.65}
	\begin{array}{rcl}
	|\sigma| & \leq & \textstyle{\frac{3}{2}}(|D|-|A|-1) +3|A|-1 +\textstyle{\frac{3}{2}}|C| + |H| \\
	& = & \textstyle{\frac{3}{2}}(|V|-1) -1 +|H|.
	\end{array}\]
	Moreover, if $H$ is non-empty, then our choice of $e_0=vw$ implies that $H[D]$ or $H[D,A]$ is also non-empty. Therefore the above inequality is strict whenever $H$ is non-empty. 
	
	As in the previous proofs, it follows from Lemma \ref{combine-ge} that the union of the acyclic matchings on every non-empty family $\mathsf{F}_{(D_1, \dots, D_r; A; C)}$ gives an acyclic matching on $\mathsf{F}$ with the same bounds on the critical sets. Using the fact that $\mathsf{F}_1 = \mathsf{F}*\{e_0\}$ finishes the proof. \qed

	\section{Proof of Proposition \ref{size-yz}} \label{section_size-yz}
	Let $H$ be a fixed bipartite graph on the vertex classes $X$ and $Y$. Our goal is to find an acyclic mathcing $\mathcal{M}$ on $\mathsf{BFC}_{(X, Y, Z; H)}$ such that any critical set $\sigma$ satisfies 
	\begin{equation} \label{BFC-bound}
	|\sigma| \leq 2|Y| + |Z| + \max\{|H|-1, 0\}.
	\end{equation}
	When $X$ or $Y$ is empty, then by our convention, we have $\mathsf{BFC}_{(X, Y, Z; H)}=\{\emptyset \}$. In this case the ineqaulity obviously hold. So we only focus on the case when $X$ and $Y$ are both non-empty, which implies that $|X|>|Y|>|Z|$. The proof goes by induction on $|X\cup Y|$. It is easy to check that the bound holds when $|X\cup Y|\leq 3$.
	
	As in the proofs of Propositions \ref{size perfect m} and \ref{size-d0}, we are going to reduce the problem of finding the acyclic matching $\mathcal{M}$ by decomposing our family into simpler parts for which we can find suitable acyclic matchings. These will then be combined to form $\mathcal{M}$.

	\subsection*{A special case} We first deal with the special case when $H[(X\setminus Z), Y]  = K_{(X\setminus Z), Y}$. In this case we claim that  $\mathsf{BFC}_{(X,Y,Z;H)}$ has the join structure
	\[\mathsf{BFC}_{(X,Y,Z;H)} = \mathsf{BFC}_{(Y,Z,\emptyset; H[Z,Y])} * \{K_{(X\setminus Z), Y}\}.\]
	
	The inclusion $\subseteq$ is trivial since $G[Z,Y]$ is $Z$-factor critical for any $G \in \mathsf{BFC}_{(X,Y,Z; H)}$. 
	For the opposite inclusion, let $G$ be a graph in the join. By assumption, $Z$ can be perfectly matched with a subset $Z'\subseteq Y$, and $(Y\setminus Z')$ can be perfectly matched to a subset $Z''\subseteq X\setminus Z$. Moreover, $X\setminus (Z\cup Z'')$ is non-empty and any vertex in $x \in X\setminus (Z\cup Z'')$ is neighbor to every vertex in $Y$. It follows that $G$ is $Y$-factor critical, which proves the equality.
	
	By induction, there is an acyclic matching $\mathcal{M}'$ on $\mathsf{BFC}_{(Y,Z,\emptyset; H[Z,Y])}$ such that any critical set $\sigma$ satisfies
	\[|\sigma| \leq 2|Z| + \max\{|H[Z,Y]|-1, 0\}, \]
	and the term $\{K_{(X\setminus Z), Y}\}$ has an empty acyclic matching with a single critical set. Since $|Z|<|Y|$ and $|H| = |H[Z,Y]| + |H[(X\setminus Z), Y]|$, the bound in \eqref{BFC-bound} holds by Lemma \ref{product}, and finishes the special case when $H[(X\setminus Z), Y]  = K_{(X\setminus Z), Y}$.
	
	\subsection*{The general case} We assume from now on that $H[(X\setminus Z), Y] \neq K_{(X\setminus Z), Y}$, and we fix an edge $e_0 = vw$ where $v\in Y$, $w\in (X\setminus Z)$, and $e_0\notin H$. If possible, we choose $e_0$ such that $N_H(v)\neq \emptyset$. Note that if this is not possible, then for any vertex $y\in Y$ either $\deg_H(y)= 0$ or $(X\setminus Z) \subseteq N_H(y)$.
	
	Once the edge $e_0$ is fixed, define the subfamily $\mathsf{F}_0\subseteq \mathsf{BFC}_{(X,Y,Z; H)}$ as
	\[\mathsf{F_0} = \{G : G-e_0, G+e_0 \in \mathsf{BFC}_{(X,Y,Z;H)}  \}.\]
	This reduces our problem  to finding an acyclic matching on $\mathsf{F}_1 = \mathsf{BFC}_{(X,Y,Z;H)} \setminus \mathsf{F}_0$ (by Lemma \ref{boolean}). Now define the family $\mathsf{F} = \{G-e_0 : G\in \mathsf{F}_1\}$. In other words,  $\mathsf{F}$ is the family of graphs $G\subseteq K_{X,Y}$ which satisfy the conditions:
	\begin{itemize}
		\item $H\subseteq G$ and $e_0\notin G$,
		\item $G$ is not $Y$-factor critical, 
		\item $G+e_0$ is $Y$-factor critical, and
		\item $G[Z,Y]$ is $Z$-factor critical.
	\end{itemize}
	Since $\mathsf{F_1} = \mathsf{F} * \{e_0\}$, our problem is now further reduced to finding an acyclic matching on $\mathsf{F}$ (by Lemma \ref{product}).
	
	\medskip
	
	Recall from Section \ref{complete bipartite} that the components of $D$ in the Gallai--Edmonds decomposition of a bipartite graph consists of singletons. As we did in Section \ref{complete bipartite}, throughout this section we simply denote the Gallai--Edmonds decomposition of 
	$G\in \mathsf{F}$ as $(D;A;C)$. 
	
	\begin{claim} \label{bipartite ge conditions}
		For $G\in \mathsf{F}$ with Gallai--Edmonds decomposition $(D; A; C)$, the following hold:
		\begin{enumerate}
			\item $w\in D \subseteq X$,
			\item $A\subseteq Y$, and
			\item $v\in (C\cap Y)$, in particular $C\neq \emptyset$.
		\end{enumerate}
		Moreover, if $H\neq \emptyset$, then $H[(C\cap X), \{v\}]\neq \emptyset$ or $H[D,A]\neq \emptyset$.
	\end{claim}
	\begin{proof}
		First we observe that $\nu(G) = |Y|$ for every $G\in \mathsf{F}$. This is because $G+e_0$ is $Y$-factor critical, which implies that $(G+e_0)-\{w\}$ has a matching which covers $Y$ (and so does every maximum matching in $G$). This implies that $w\in D \subseteq X$ and $A\subseteq Y$. And for $x\in C\cap X$, $(G+e_0)-\{x\}$ has a matching $M$ covering $Y$, and $M$ should use an edge between $D$ and $C\cap Y$ which must be $e_0$. Hence, $v\in (C\cap Y)$. 
		
		Now suppose $H\neq\emptyset$. Every vertex of $X$ belongs to either $D$ or $C$, and since $v\in (C\cap Y)$ it follows that $N_H(v)\subseteq (C\cap X)$. Therefore $H[(C\cap X), \{v\}]\neq\emptyset$ 
		provided that $N_H(v)\neq\emptyset$. On the other hand, suppose $N_H(v)= \emptyset$ for every possible choice of $v$. (Recall our choice of $v$ when we fixed the edge $e_0 = vw$.) If $H\neq \emptyset$, then there is some vertex $y\in Y$ such that $\deg_H(y)>0$ which implies that $(X\setminus Z)\subseteq N_H(y)$. Since $w\in (X\setminus Z) \cap D$, we must have $y\in A$ and therefore $H[D,A]\neq \emptyset$.
	\end{proof}

	Since all graphs in $\mathsf{F}$ have the same matching number we can apply Lemma \ref{combine-ge}, thereby further reducing our problem  to finding an acyclic matching on each non-empty subfamily $\mathsf{F}_{(D;A;C)} \subseteq \mathsf{F}$.

	Now we fix a non-empty subfamily $\mathsf{F}_{(D;A;C)}$. 
	Our goal is to decompose this family in order to further reduce our problem. We write 
	\[C = C_X \cup C_Y \; \text{and}\; Z = Z_C \cup Z_D\]  where  $C_X = (C\cap X)$,  $C_Y = (C\cap Y)$,  $Z_C = (Z\cap C)$,  and $Z_D = (Z\cap D)$. 
	
	Define $\mathsf{F}_{YC}$
	to be the family of bipartite graphs $G \subseteq K_{C_X, Y}$ which satisfy the conditions:
	\begin{itemize}
		\item $H[C_X, Y]\subseteq G$, 
		\item $G[C]$ has a perfect matching, 
		\item $G[Z_C, Y]$ is $Z_C$-factor critical, and
		\item $G[C_X, (C_Y\setminus \{v\})]$ is $(C_Y\setminus\{v\})$-factor critical.
	\end{itemize}
	
	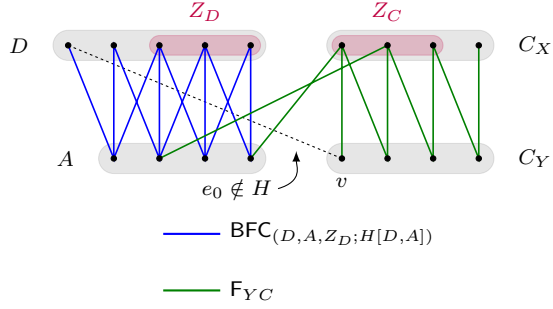
\begin{figure}
		\centering
		\begin{tikzpicture}
		\begin{scope}
		
		\coordinate (d1) at (0.0,3/2);
		\coordinate (d2) at (0.6,3/2);
		\coordinate (d3) at (1.2,3/2);
		\coordinate (d4) at (1.8,3/2);
		\coordinate (d5) at (2.4,3/2);
		
		\coordinate (a1) at (0.6,0);
		\coordinate (a2) at (1.2,0);
		\coordinate (a3) at (1.8,0);
		\coordinate (a4) at (2.4,0);
		
		\coordinate (cx1) at (3.6,3/2);
		\coordinate (cx2) at (4.2,1.5);
		\coordinate (cx3) at (4.8,1.5);
		\coordinate (cx4) at (5.4,1.5);
		
		\coordinate (cy1) at (3.6,0);
		\coordinate (cy2) at (4.2,0);
		\coordinate (cy3) at (4.8,0);
		\coordinate (cy4) at (5.4,0);
		
		\filldraw [black, thin, opacity=0.1] ($(d1)+(0,0.2)$) -- ($(d5)+(0,0.2)$) arc[start angle = 90, end angle = -90, radius = 0.2] -- ($(d1)+(0,-0.2)$) arc[start angle = 270, end angle = 90, radius = 0.2];
		
		\filldraw [black, thin, opacity=0.1] ($(cx1)+(0,0.2)$) -- ($(cx4)+(0,0.2)$) arc[start angle = 90, end angle = -90, radius = 0.2] -- ($(cx1)+(0,-0.2)$) arc[start angle = 270, end angle = 90, radius = 0.2];
		
		\filldraw [black, thin, opacity=0.1] ($(a1)+(0,0.2)$) -- ($(a4)+(0,0.2)$) arc[start angle = 90, end angle = -90, radius = 0.2] -- ($(a1)+(0,-0.2)$) arc[start angle = 270, end angle = 90, radius = 0.2];
		
		\filldraw [black, thin, opacity=0.1] ($(cy1)+(0,0.2)$) -- ($(cy4)+(0,0.2)$) arc[start angle = 90, end angle = -90, radius = 0.2] -- ($(cy1)+(0,-0.2)$) arc[start angle = 270, end angle = 90, radius = 0.2];
		
		\filldraw [purple, thin, opacity=0.2] ($(d3)+(0,0.13)$) -- ($(d5)+(0,0.13)$) arc[start angle = 90, end angle = -90, radius = 0.13] -- ($(d3)+(0,-0.13)$) arc[start angle = 270, end angle = 90, radius = 0.13];
		
		\filldraw [purple, thin, opacity=0.2] ($(cx1)+(0,0.13)$) -- ($(cx3)+(0,0.13)$) arc[start angle = 90, end angle = -90, radius = 0.13] -- ($(cx1)+(0,-0.13)$) arc[start angle = 270, end angle = 90, radius = 0.13];
		
		\draw [thin, dash pattern={on 1pt off 1.2pt}] (d1) -- (cy1);
		
		\draw [blue, semithick]
		(d1) -- (a1) -- (d2) -- (a2) -- (d3) -- (a3) -- (d4) -- (a4) -- (d5) (d3) -- (a1) (d4) -- (a2) (d5) -- (a3);
		
		\draw [green!50!black, semithick]
		(cy1) -- (cx1) -- (cy2) -- (cx2) -- (cy3) -- (cx3) -- (cy4) -- (cx4) (a2) -- (cx2) (a4) -- (cx1);
		
		\filldraw[black] (d1) circle (1pt);
		\filldraw[black] (d2) circle (1pt);
		\filldraw[black] (d3) circle (1pt);
		\filldraw[black] (d4) circle (1pt);
		\filldraw[black] (d5) circle (1pt);
		
		\filldraw[black] (a1) circle (1pt);
		\filldraw[black] (a2) circle (1pt);
		\filldraw[black] (a3) circle (1pt);
		\filldraw[black] (a4) circle (1pt);
		
		\filldraw[black] (cx1) circle (1pt);
		\filldraw[black] (cx2) circle (1pt);
		\filldraw[black] (cx3) circle (1pt);
		\filldraw[black] (cx4) circle (1pt);
		
		\filldraw[black] (cy1) circle (1pt);
		\filldraw[black] (cy2) circle (1pt);
		\filldraw[black] (cy3) circle (1pt);
		\filldraw[black] (cy4) circle (1pt);
		
		\node [below] at ($(cy1)+(0,-0.13)$) {\footnotesize $v$};
		\node[left] at ($(d1)+(-.4,0)$) {\footnotesize $D$};
		\node[left] at ($(a1)+(-.4,0)$) {\footnotesize $A$};
		\node[right] at ($(cx4)+(.4,0)$) {\footnotesize $C_X$};
		\node[right] at ($(cy4)+(.4,0)$) {\footnotesize $C_Y$};
		\node[above] at ($(d4)+(0,.2)$) {\color{purple} \footnotesize $Z_D$};
		\node[above] at ($(cx2)+(0,.2)$) {\color{purple} \footnotesize $Z_C$};
		\draw[thin, -latex] (2.76,-.4) .. controls  (3.1,-.3) and (3,-.1) .. (3,.1);
		\node [left] at (2.8,-.4) {\footnotesize $e_0 \notin H$};
		
		\draw [thick, blue] (2,-1) --++ (-.75,0);
		\draw [thick, green!50!black] (2,-1.75) --++ (-.75,0);
		\node [right] at (2,-1) {\footnotesize $\BFC_{(D,A,Z_D;H[D,A])}$}; 
		\node [right] at (2,-1.75) {\footnotesize $\mathsf{F}_{YC}$}; 
		
		\end{scope}
		\end{tikzpicture}
		\caption{The join structure in Claim \ref{size BFC DAC}.}
		\label{fig:claim7.2}
	\end{figure}
	
	\begin{claim} \label{size BFC DAC} For every non-empty $\mathsf{F}_{(D;A;C)} \subseteq \mathsf{F}$ we have
		\[\mathsf{F}_{(D;A;C)} = \mathsf{BFC}_{(D,A,Z_D; H[D,A])} * \mathsf{F}_{YC}.\]
	\end{claim}
	\begin{proof} (See Figure \ref{fig:claim7.2} for an illustration.) We first show that $G\in \mathsf{F}_{(D;A;C)}$ is contained in the join. It is a straight-forward consequence of Theorem \ref{ge-decomposition}  that $G[D,A] \in \mathsf{BFC}_{(D,A,Z_D; H[D,A])}$, so it remains to show $G[C_X, Y] \in \mathsf{F}_{YC}$. The first three conditions follow from the  defining properties of $\mathsf{F}$.
		To see that $G[C_X, (C_Y\setminus \{v\})]$ is $(C_Y\setminus\{v\})$-factor critical, note that
		$(G+e_0) - \{x\}$ has a matching which covers $C_Y$, for any $x\in C_X$. Such a matching must include the edge $e_0$, while the remaining edges form a perfect matching between $(C_X\setminus\{x\})$ and $(C_Y\setminus \{v\})$.
		
		Now consider a graph $G$ in the join. Clearly $H\subseteq G$ and $e_0\notin G$. 
		Note that $G$ has Gallai--Edmonds decomposition $(D;A;C)$ since $G[D,A]$ is $A$-factor critical and there are no edges between $D$ and $C$.
		This also implies that $G$ is not $Y$-factor critical, but with the condition that $G[C_X, (C_Y\setminus \{v\})]$ is $(C_Y\setminus\{v\})$-factor critical we get that $G+e_0$ is $Y$-factor critical. 
		Finally, we show that $G[Z,Y]$ is $Z$-factor critical. 
		By Hall's theorem this is equivalent to showing that $|N_G(Z')|>|Z'|$ for every $Z'\subseteq Z$. 
		For $Z'\subseteq Z_C$ this is true since $G[Z_C, Y]$ is $Z_C$-factor critical. If $(Z'\cap Z_D) \neq \emptyset$, then \[|N_G(Z')| \geq  |N_{G[C]}(Z'\cap Z_C) \cup N_{G[D,A]}(Z'\cap Z_D)| > |Z'\cap Z_C| + |Z'\cap Z_D| = |Z'|,\] since $G[C]$ has a perfect matching and $G[Z_D,A]$ is $Z_D$-factor critical. 
	\end{proof}
	
	The problem of finding an acyclic matching on $\mathsf{F}_{(D;A;C)}$ can now be reduced further by Lemma \ref{product}. 
	The term $\mathsf{BFC}_{(D,A,Z_D;H[D,A])}$ can be dealt with by induction, and so now we focus on finding an acyclic matching for the family $\mathsf{F}_{YC}$.
	
	\medskip
	
	We now make a further reduction. 
	For disjoint subsets $S,T\subseteq Z_C$ (which may be empty) we say that $G\in \mathsf{F}_{YC}$ is of {\em Type $(S,T)$} if 
	\[N_G(v) \cap Z_C = S \text{ and } N_G(A\cup \{v\})\cap Z_C = (S\cup T),\] 
	and partition the graphs in $\mathsf{F}_{YC}$ according to their Type. Suppose we have an acyclic matching on each part of this partition. We claim that their union is an acyclic matching on $\mathsf{F}_{YC}$. For contradiction, assume there is a directed cycle 
	\[(\sigma_0, \tau_0, \sigma_1, \tau_1, \cdots, \sigma_{t-1}, \tau_{t-1})\]
	where $\sigma_i, \sigma_{i+1}\subseteq \tau_i$ (according to Lemma \ref{cycle}). This would imply that 
	\[(N_{\sigma_i}(v)\cap Z_C) = (N_{\tau_i}(v)\cap Z_C) \supseteq (N_{\sigma_{i+1}}(v)\cap Z_C),\]
	for all $i$ (indices are taken modulo $t$). This shows that $(N_{\sigma_i}(v)\cap Z_C) = (N_{\tau_j}(v)\cap Z_C)$ and (by the same reasoning) $(N_{\sigma_i}(A\cup\{v\}) \cap Z_C) = (N_{\tau_j}(A\cup \{v\})\cap Z_C)$ for all $i$ and $j$, and consequently every graph in the directed cycle have the same Type. We have therefore reduced the problem to finding an acyclic matching on each non-empty family \[\mathsf{F}_{YC}^{(S,T)} = \{G\in \mathsf{F}_{YC} : G \text{ is of Type } (S,T)\}.\]
	
	Now consider a fixed non-empty subfamily $\mathsf{F}_{YC}^{(S,T)}$. We write 
	\[C_X = Q \cup S \cup T \cup R\]
	where $R = Z_C \setminus (S\cup T)$ and $Q = (C_X\setminus Z_C)$. 
	Note that this is not necessarily a partition of $C_X$ since some of the terms could be empty.
	
	Define families
	\[\mathsf{P}_v = \{G\subseteq K_{C_X, \{v\}} 
	: (N_G(v) \cap Z_C) = S, N_G(v)\neq\emptyset, H[C_X, \{v\}]\subseteq G\},\]
	\[\mathsf{P}_Q = \{G\subseteq K_{Q,A} 
	: H[Q, A]\subseteq G\},\]
	and 
	\[\mathsf{P}_A = \{G\subseteq K_{Z_C, A} : T\subseteq N_G(A) \subseteq (S\cup T), H[Z_C, A]\subseteq G\},\] with the additional condition that $N_G(A)\neq\emptyset$ if $Q = \emptyset$. 
	
	\begin{figure}
		\centering
		\begin{tikzpicture}
		\coordinate (cx1) at (0.0,3/2);
		\coordinate (cx2) at (0.6,3/2);
		\coordinate (cx3) at (1.2,3/2);
		\coordinate (cx4) at (1.8,3/2);
		
		\coordinate (cy1) at (0.0,0/2);
		\coordinate (cy2) at (0.6,0/2);
		\coordinate (cy3) at (1.2,0/2);
		\coordinate (cy4) at (1.8,0/2);
		
		\coordinate (a1) at (-1.0,0/2);
		\coordinate (a2) at (-1.6,0/2);
		\coordinate (a3) at (-2.2,0/2);
		\coordinate (a4) at (-2.8,0/2);
		
		\filldraw [black, thin, opacity=0.1] ($(cy1)+(0,0.2)$) -- ($(cy4)+(0,0.2)$) arc[start angle = 90, end angle = -90, radius = 0.2] -- ($(cy1)+(0,-0.2)$) arc[start angle = 270, end angle = 90, radius = 0.2];
		
		\filldraw [black, thin, opacity=0.1] ($(cx1)+(0,0.2)$) -- ($(cx4)+(0,0.2)$) arc[start angle = 90, end angle = -90, radius = 0.2] -- ($(cx1)+(0,-0.2)$) arc[start angle = 270, end angle = 90, radius = 0.2];
		
		\filldraw [black, thin, opacity=0.1] ($(a4)+(0,0.2)$) -- ($(a1)+(0,0.2)$) arc[start angle = 90, end angle = -90, radius = 0.2] -- ($(a4)+(0,-0.2)$) arc[start angle = 270, end angle = 90, radius = 0.2];
		
		\filldraw [purple, thin, opacity=0.2] ($(cx1)+(0,0.13)$) -- ($(cx3)+(0,0.13)$) arc[start angle = 90, end angle = -90, radius = 0.13] -- ($(cx1)+(0,-0.13)$) arc[start angle = 270, end angle = 90, radius = 0.13];
		
		\draw[semithick, violet] 
		(cx1)--(a1) (cx2)--(a3);
		\draw[semithick, gray]
		(cx1)--(cy1);
		\draw[semithick, green!50!black]
		(cx1) --(cy2) --(cx2) --(cy3) --(cx3) --(cy4) --(cx4);
		\draw[semithick, orange]
		(cx4) --(a1) (cx4) --(a2) (cx4) --(a3) (cx4) --(a4);
		
		\filldraw [black]
		(cx1) circle (1pt)
		(cx2) circle (1pt)
		(cx3) circle (1pt)
		(cx4) circle (1pt)
		
		(cy1) circle (1pt)
		(cy2) circle (1pt)
		(cy3) circle (1pt)
		(cy4) circle (1pt)
		
		(a1) circle (1pt)
		(a2) circle (1pt)
		(a3) circle (1pt)
		(a4) circle (1pt);
		
		\node [above] at ($(cx1)+(0,0.19)$) {\footnotesize $S$};
		\node [above] at ($(cx2)+(0,0.19)$) {\footnotesize $T$};
		\node [above] at ($(cx3)+(0,0.19)$) {\footnotesize $R$};
		\node [above] at ($(cx4)+(0,0.13)$) {\footnotesize $Q$};
		\node [below] at ($(cy1)+(0,-0.13)$) {\footnotesize $v$};
		
		\node [right] at ($(cx4)+(.4,0)$) {\footnotesize $C_X$};
		\node [right] at ($(cy4)+(.4,0)$) {\footnotesize $C_Y$};
		\node [left] at ($(cx1)+(-.4,0)$) {\color{purple} \footnotesize $Z_C$};
		\node [left] at ($(a4)+(-.4,0)$) {\footnotesize $A$};

		\draw[thick, gray] (-2,-1) --++(-0.75,0);
		\node[right] at (-2,-1) {\footnotesize $\mathsf{P}_v$}; 
		
		\draw[thick, violet] (0,-1) --++(-0.75,0);
		\node[right] at (0,-1) {\footnotesize $\mathsf{P}_A$}; 
		
		\draw[thick, orange] (-2,-1.75) --++(-0.75,0);
		\node[right] at (-2,-1.75) {\footnotesize $\mathsf{P}_Q = \{K_{Q,A}\}$};
		
		\draw[thick, green!50!black] (-2,-2.5) --++(-0.75,0);
		\node[right] at (-2,-2.5) {\footnotesize $\BFC_{(C_X,C_Y\setminus\{v\}, R; H[C_X,C_Y\setminus\{v\}])}$};

		\end{tikzpicture}    
		\caption{The join structure in Claim \ref{FYC join} when $H[Q,A]=K_{Q,A}$.}
		\label{fig:claim7.3}
	\end{figure}
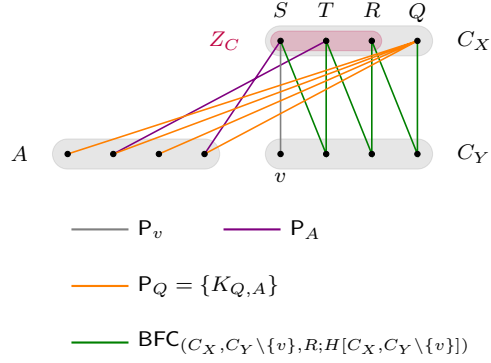
	
	\begin{claim} \label{FYC join}
		For any non-empty $\mathsf{F}_{YC}^{(S,T)}\subseteq \mathsf{F}_{YC}$ we have
		\[\mathsf{F}_{YC}^{(S,T)} = \mathsf{P}_{v} * \mathsf{P}_A * \mathsf{P}_Q *
		\mathsf{BFC}_{(C_X, C_Y\setminus\{v\}, R; H[C_X, C_Y\setminus \{v\}])}
		.\]
	\end{claim}
	
	\begin{proof}
		(See Figure \ref{fig:claim7.3} for an illustration.) Consider a graph $G\in \mathsf{F}_{YC}^{(S,T)}$. Clearly we have $G[Q,A]\in \mathsf{P}_Q$. And it follows from the definition of $\mathsf{F}_{YC}^{(S,T)}$ that $G[C_X, \{v\}]\in \mathsf{P}_v$ and $G[Z_C, A]\in \mathsf{P}_A$. 
		We claim that $G' = G[C_X, (C_Y\setminus \{v\})]$ is $(C_Y\setminus\{v\}, R)$-factor critical. 
		The fact that $G'$ is $(C_Y\setminus\{v\})$-factor critical follows from the defining conditions of $\mathsf{F}_{YC}$. 
		The fact that $G'[R,C_Y\setminus \{v\}]$ is $R$-factor critical follows from the condition that $G[Z_C, Y]$ is $Z_C$-factor critical and the definition of $R$. Therefore $G$ is contained in the join. 
		
		Now consider a graph $G$ in the join. It is obvious that $H[C_X, Y]\subseteq G$. 
		The $\mathsf{BFC}$-term implies that $G[C_X, (C_Y\setminus\{v\})]$ is $(C_Y\setminus\{v\})$-critical, 
		and from the assumption $N_G(v)\neq\emptyset$ it follows that $G[C]$ has a perfect matching. 
		We now show that $G[Z_C, Y]$ is $Z_C$-factor critical. By Hall's marriage theorem this is equivalent to showing that $|N_G(Z')| > |Z'|$ for every $Z'\subseteq Z_C$.
		If $Z' = C_X$, then we must have $Q = \emptyset$. 
		In this case we have the additional condition $N_G(A)\neq \emptyset$. 
		Therefore $Z'$ has a neighbor in $A$, 
		and since $G[C]$ has a perfect matching it follows that $|N_G(Z')| > |Z'|$. 
		Now suppose $Z'\neq C_X$. If $Z'\subseteq R$, then we are done by the condition on the $\mathsf{BFC}$-term, so we are left with the case when $Z'$ contains at least one vertex from $(S\cup T)$. Since $G[C_X,(C_Y\setminus\{v\})]$ is $(C_Y\setminus \{v\})$-factor critical there is a matching in $G[C_X,(C_Y\setminus \{v\})]$ which covers $Z'$. Moreover, every vertex in $(S\cup T)$ has at least one neighbor in $A\cup \{v\}$, and therefore $|N_G(Z')|>|Z'|$. We have shown that $G\in \mathsf{F}_{YC}$, and it follows from the definition of $\mathsf{P}_v$ and $\mathsf{P}_A$ that $G$ is of Type $(S,T)$.
	\end{proof}
	
	We can handle the $\mathsf{BFC}$-term in the join in Claim \ref{FYC join} by induction. The term $\mathsf{P}_Q$ has a complete acyclic matching when $H[Q,A]\ne K_{Q,A}$, and an empty matching with a single critical set $H[Q,A]$ otherwise (by Lemma \ref{boolean}). So it remains to find acyclic matchings for the terms $\mathsf{P}_v$ and $\mathsf{P}_A$.
	
	\begin{claim} \label{size ST}
		There is an acyclic matching on $\mathsf{P}_v *  \mathsf{P}_A$ such that any critical set $\sigma$ satisfies
		\[|\sigma| \leq |S| + |T| + \max\{|H[C_X, \{v\}]|-1, 0\} + |H[Z_C, A]| +1. \]
	\end{claim}
	
	\begin{proof}
		We start with the term $\mathsf{P}_v$. Set $H' = H[C_X, \{v\}]$ and $N' = N_{H'}(v)$. We can write
		\[ \mathsf{P}_v = \{ K_{(N'\cup S), \{v\}} \}
		* \mathsf{P}'
		\]
		where 
		$\mathsf{P}' = \{
		G : G\subseteq K_{(Q\setminus N'), \{v\}} \text{ and } G\neq \emptyset \text{ if }
		(N'\cup S) = \emptyset\}$.
		
		We now use Lemma \ref{boolean} to find an acyclic matching on $\mathsf{P}'$. First consider the case $(N'\cup S) \neq \emptyset$. If $(Q\setminus N') \neq \emptyset$, then $\mathsf{P}'$ has a complete acyclic matching, and otherwise $\mathsf{P}' = \{\emptyset\}$ and there is an acyclic matching on $\mathsf{P}'$ with a single critical set of size $0$. 
		In the case $(N'\cup S) = \emptyset$, we have 
		$\mathsf{P}_v = \{
		G : \emptyset \neq G \subseteq K_{Q, \{v\}} \}$ where $Q\neq \emptyset$. In this case $\mathsf{P}_v$ has an acyclic matching with a single critical set of size $1$.  Using Lemma \ref{product} we see that $\mathsf{P}_v$ has an acyclic matching where any critical set $\sigma$ satisfies
		\[|\sigma|\leq
		\begin{cases}
		|N' \cup S| & \text{ when } (N' \cup S) \neq \emptyset, \\
		1 & \text{ when } (N' \cup S) = \emptyset. 
		\end{cases}
		\]
		In either case, any critical set $\sigma$ will satisfy\textbf{}
		\[|\sigma| \leq |S| + \max\{ |H[C_X, \{v\}]|-1, 0 \} +1.\]
		
		Now we consider the term $\mathsf{P}_A$, and set $H'' = H[Z_C, A]$ and $N'' = N_{H''}(A)$. %Note that $N\subseteq (S\cup T)$. 
		First consider the case when $A = \emptyset$, which implies that $\mathsf{P}_A = \{\emptyset\}$. Here there is an empty acyclic matching with a single critical set of size $0$. By Lemma \ref{product} and the acyclic matching for $\mathsf{P}_v$ found above, we have an acyclic matching on $\mathsf{P}_v * \mathsf{P}_A$ satisfying the desired bound. So we may suppose $A\neq \emptyset$.
		
		If there is an edge $e \in K_{(N''\cup S), A}$ with $e\not\in H''$,  then we have a complete acyclic matching $\{ (G-e, G+e) : G\in \mathsf{P}_A \}$, or an acyclic matching $\{ (G-e, G+e) : G\in \mathsf{P}_A\setminus \{\{e\} \} \}$ with a single critical set $\{e\}$ on $\mathsf{P}_A$. In the former case, we have a complete acyclic matching on $\mathsf{P}_v * \mathsf{P}_A$ (by Lemma \ref{product}). And the latter case occurs exactly when $Q=\emptyset$, $N''=\emptyset$ and $T=\emptyset$. Note that if $Q=\emptyset$, then $S\neq\emptyset$ by assumption that $N_G(v)$ is non-empty for $G\in \mathsf{P}_v$, and $N'\subseteq S$. Therefore by Lemma \ref{product}, we can find an acyclic matching  on $\mathsf{P}_v * \mathsf{P}_A$ where any critical set $\sigma$ satisfies
		$$|\sigma|\leq |S|+1.$$

		So we may assume $H'' = K_{(N''\cup S), A}=K_{N'', A}$.  We can write 
		\[\mathsf{P}_A = \{K_{N'', A}\} * \mathsf{P}'' \]
		where $\mathsf{P}'' = \{G : G\subseteq K_{(T\setminus N''), A}, N_G(A) = (T\setminus N'')\}$. 
		
		We now use Lemma \ref{iden} to find an acyclic matching on $\mathsf{P}''$. 
		If $(T\setminus N'') = \emptyset$, then $\mathsf{P}'' = \{\emptyset\}$ and we have an empty acyclic matching with a single critical set of size $0$.  Now suppose $(T \setminus  N'') = \{v_1, \dots, v_m\} \neq \emptyset$ and 
		set \[E_i = \{e\in K_{(T \setminus N''), A} : e \text{ incident to } v_i\}.\] This  gives us a partition of $K_{(T\setminus N''), A}$ and a projection map $\pi: K_{(T\setminus N''), A} \to K_{(T\setminus N''), \{\bar{a}\}}$. Applying Lemma \ref{iden} with $\tau=\emptyset$ and $\bar{\mathsf{F}} = \{ K_{(T\setminus N''), \{\bar{a}\}} \}$, we find that $\mathsf{P}''$ 
		has an acyclic mathcing with a single critical set of size $|T\setminus N''|$. By Lemma \ref{product} we have an acyclic matching on $\mathsf{P}_A$ where any critical set $\sigma$ satisfies
		\[|\sigma| \leq |K_{N'', A}| + |T\setminus N''| \leq |H[Z_C, A]| + |T|.\]
		Together with the bound from the acyclic matching on $\mathsf{P}_v$ found above, we get the desired bound by Lemma \ref{product}.
	\end{proof}
	
	Using the join structure in Lemma \ref{FYC join}, there is an an acyclic matching on $\mathsf{F}_{YC}^{(S,T)}$ where any critical set $\sigma$ satisfies
	\[
	\def\arraystretch{1.75}
	\begin{array}{rcll}
	|\sigma| & \leq & |S| + |T| + 1 +|H[Z_C, A]|  + \max\{|H[C_X, \{v\}]| - 1, 0\} &  \text{(Claim \ref{size ST})} \\
	&& + |H[Q,A]| & \text{(Lemma \ref{boolean})} \\
	&& + 2|C_Y\setminus \{v\}| + |R| + \max\{|H[C_X, (C_Y\setminus \{v\})]|-1, 0\} & \text{(induction)} \\
	&\leq & 2|C_Y| + |Z_C| - 1 + |H[C_X,A]| + \max\{|H[C_X, C_Y]|-1, 0\}.
	\end{array}\]
	As we observed earlier, the union of all acyclic matchings on the non-empty subfamilies $\mathsf{F}_{YC}^{(S,T)}\subseteq \mathsf{F}_{YC}$ gives us an acyclic matching on $\mathsf{F}_{YC}$ with the same bound as above. 
	
	Using join structure in Claim \ref{size BFC DAC} and the induction hypothesis, there is an acyclic matching on $\mathsf{F}_{(D; A; C)}$ where any critical set $\sigma$ satisfies
	\[
	\def\arraystretch{1.75}
	\begin{array}{rcll}
	|\sigma| &\leq & 2|C_Y| + |Z_C| - 1 + |H[C_X,A]| + \max\{|H[C_X, C_Y]|-1, 0\} \\
	&& + 2|A| + |Z_D| + \max\{|H[D,A]|-1, 0\} & \text{(induction)} \\ 
	&\leq & 2|Y| + |Z| -1  + |H[C_X, A]|  & \\ && + \max\{|H[C_X, C_Y]|-1, 0\} + \max\{|H[D,A]|-1, 0\} &  \\
	& \leq & 2|Y| + |Z| -1 + \max\{|H|-1, 0 \} & \text{(Claim \ref{bipartite ge conditions})}
	\end{array}\]
	Taking the union of all acyclic matchings on the non-empty subfamilies $\mathsf{F}_{(D; A; C)}\subseteq \mathsf{F}$ gives an acyclic matching in $\mathsf{F}$ with the same bound as above. Finally, since $\mathsf{F}_1 =\mathsf{F} * \{e_0\}$ we get the desired bound by Lemma \ref{product}.  \qed

	\section{Application to rainbow matching problems} \label{rainbow applications}
	In this section, we prove Theorems \ref{drisko-general} and \ref{rainbow-matching-result}. 
	A simplicial complex $\mathsf{K}$ is called \textit{near-$d$-Leray} (over the field $\mathbb{F}$) if
	the reduced homology $ \tilde{H}_i(\lk_\mathsf{K}(\sigma))$ over $\mathbb{F}$ vanishes 
	for every non-empty face $\sigma \in K$ and $i\geq d$. With this terminology, Theorem \ref{link-result} can be restated that for every $k\geq 2$,
	\begin{itemize}
		\item $\NM_k(G)$ is near-$(3k-4)$-Leray for an arbitary graph $G$, and
		\item $\NM_k(G)$ is near-$(2k-3)$-Leray for a bipartite graph $G$.
	\end{itemize}

	The near-$d$-Leray property has the following consequence. Here, a matroid on $V$ is a non-void simplicial complex $\mathsf{M}$ which satisfies the augmentation property, that is, if $\sigma, \tau \in \mathsf{M}$ and $|\sigma|<|\tau|$, then there exists $v\in \tau\setminus \sigma$ such that $\sigma \cup \{v\} \in \mathsf{M}$. We only consider loopless matroids, that is, $\{v\} \in \mathsf{M}$ for every $v\in V$. The rank function $\rho$ of $\mathsf{M}$ assigns to every subset $W\subseteq V$ the number $\rho(W)=\max\{|\sigma|:\sigma\in \mathsf{M}, \sigma \subseteq U\}$.

	\begin{thm}[\cite{holmsen_colorful_caratheodory}]
		\label{holmsen}
		Let $\mathsf{K}$ be simplicial complex on $V$ which is near-$d$-Leray over the rational field, and let $\mathsf{M}$ be a
		matroid on $V$ with the rank function $\rho$ such that $\rho(V)\geq d+2$. If $\mathsf{M}$ is a subcomplex of $\mathsf{K}$,
		then there exists a face $\sigma \in \mathsf{K}$ such that $\rho(V \setminus \sigma) \leq d$. 
	\end{thm}

	\begin{proof}[Proof of Theorems \ref{drisko-general} and \ref{rainbow-matching-result}] The proof of both theorems use the same application of Theorem \ref{holmsen} with different values of $d$. So for 
		Theorem \ref{drisko-general} let $d = 2k-3$, and for  Theorem \ref{rainbow-matching-result} let $d = 3k-4$. Suppose we are given $d+2$ non-empty edge sets $E_1, \dots, E_{d+2}$ and set $E=\bigcup_i E_i$. In the case of Theorem \ref{drisko-general} we have the additional assumption that $E$ is the edge set of a bipartite graph.
		
		Let $\tilde{E}_i$ be the set of labelled edges of $E_i$, that is $\tilde{E}_i= \{(e, i): e \in E_i \}$, and let $\tilde{E}= \bigcup_i \tilde{E}_i$. Define the simplicial complex 
		\[\mathsf{K} = \{\sigma  \subseteq  \tilde{E} : \nu(\pi(\sigma)) < k\},\] where $\pi:2^{\tilde{E}}\to 2^{E}$ is the function defined by $\pi(\sigma)=\{e: (e,i)\in \sigma \}$.  
		
		By Theorem \ref{link-result}, it follows that $\NM_k(E)$ is near-$d$-Leray over the rational field. We now show that $\mathsf{K}$ is also near-$d$-Leray. That is, given a non-empty face $\sigma \in \mathsf{K}$, we show that $\mathsf{X}=\lk_K(\sigma)$ has vanishing reduced homology from the dimension $d$ and above. Let us simply denote $\pi^{-1}(\{e\})$ for $e\in E$ by $\tau_e$.
		If there is $e \in E$ such that $\emptyset\ne \sigma\cap \tau_e \subsetneq \tau_e$, then  $\mathsf{X}=2^{\tau_e\setminus \sigma} * \mathsf{X}[\tilde{E}\setminus (\tau_e\cup \sigma)]$. Using a complete matching on $2^{\tau_e\setminus \sigma}$, we can find a complete matching on $\mathsf{X}$ by using Lemma \ref{product}, so by Theorem \ref{morse_fundamental} it follows  that $\mathsf{X}$ has vanishing homology in all dimensions. 
		
		Hence, we assume that for all $e \in E$ we have that either $\sigma\cap \tau_e=\tau_e$ or $\sigma\cap \tau_e=\emptyset$. In this case, one can see that $\mathsf{X}$ is homotopy equivalent to $\pi(\mathsf{X})=\lk_{\NM_k(E)}(\pi(\sigma))$, for example by finding a collapsing sequence from $\mathsf{X}$ to a copy of $\pi(\mathsf{X})$ inside $\mathsf{X}$. (The argument is very similar to \cite[Proposition 2.1]{aharoni_holzman_jiang}.) Since $\tilde{H}_i(\pi(\mathsf{X}))=0$ for every $i\geq d$, we have that $\tilde{H}_i(\mathsf{X}) =0$ for every $i\geq d$.
		
		Now, let $\mathsf{M}$ be the partition matroid on the partition $\tilde{E}_1 \cup \cdots \cup \tilde{E}_{d+2}$. That is, let $\mathsf{M}$ be the matroid on $\tilde{E}$ defined by
		\[\mathsf{M}=\{\tilde{E}'\subseteq \tilde{E}: |\tilde{E}'\cap (E_i\times \{i\})|\leq 1 \textrm{ for every $i\in [d+2]$}\}.\] 
		Note that for the rank function $\rho$ of $\mathsf{M}$, $\rho(\tilde{E}')$ is the number of sets $\tilde{E}_i$ which $\tilde{E}'$ intersects. Therefore $\rho(\tilde{E})=d+2>d$.
		
		Suppose that $E_1\cup \cdots E_{d+2}$ does not contain any
		rainbow matchings of size $k$. Then, $\mathsf{M}$ is a subcomplex of $\mathsf{K}$. Thus we see that $\mathsf{K}$ and $\mathsf{M}$ satisfy the conditions in Theorem \ref{holmsen}. It follows that there is a face $\sigma \in \mathsf{K}$, and two distinct sets $\tilde{E}_i$ and $\tilde{E}_j$ such that $\tilde{E}_i\cup \tilde{E}_j \subseteq \sigma$. This implies that
		$$ \nu(E_i\cup E_j)=\nu(\pi(\tilde{E}_i\cup \tilde{E}_j))\leq \nu(\pi(\sigma))<k,$$
		which contradicts the assumption.\end{proof}
	
	Let us also remark that the proof method above allows us to generalize Theorems \ref{drisko-general} and \ref{rainbow-matching-result} to arbitrary matroids. (We leave the proof to the reader.)
	
	\begin{cor}
		Let $\mathsf{M}$ be a matroid on the edge set $E$ with rank function $\rho$ and suppose $\rho(E)\geq 3k-2$. If $\nu(F)\geq k$ for every flat $F\subset E$ of rank 2, then there is a matching of size $k$ which is independent in $\mathsf{M}$. 
		The same conclusion holds for a bipartite edge set $E$ under a weaker assumption that $\rho(E)\geq 2k-1$.
	\end{cor}
	
	\section{Final Remarks} \label{finals}
	
	One of the main open problems that remains is to determine the minimum number of matchings of size $k$ needed to guarantee the existence of a rainbow matching. As remarked in the introduction some further progress was made recently in \cite{rainbow_matching_general_improvement} and \cite{collaborative_bipartite}. However, since the Leray numbers of the non-matching complex can not be reduced in general, we expect that topological methods will not be useful in making further progress on this problem. 
	
	\vspace{1ex}
	
	Another intriguing question was raised in the paper by Linusson et al. \cite{linusson_bounded_matching_complex}. They asked whether the non-matching complex $\mathsf{NM}_k(G)$ is homotopy equivalent to a wedge of spheres in the case when $G$ is a complete multipartite graph. Using  the methods developed in this paper we can prove a special case: when $G$ is a complete multipartite graph on at least three vertex classes {\em and} where one of the vertex classes consists of a single vertex, then $\mathsf{NM}_k(G)$ is homotopy equivalent to a wedge of spheres of dimension $3k-4$. We expect that with further development of our tools, the problem can be fully settled. 
	
	\section*{Acknowledgements}
	The authors are thankful to Ron Aharoni, Joseph Briggs, Minho Cho, Jinha Kim, Minki Kim, and Alan Lew for helpful discussion and useful comments. This work was initiated during Ron Aharoni's visit to KAIST in summer 2018. We also give special thanks to Jinha Kim for letting us to reproduce her simple proof of Proposition \ref{hereditary}. Last but not least, we would like to thank the anonymous referees for their helpful comments and suggestions that helped us to improve the presentation.

	\bibliographystyle{alpha}
	\bibliography{ref}
	
	%\end{document}
	
	%Additional Tikz figures
	
	%Do not delete figures. Just move the \end{document} command
\end{document}